\documentclass[11pt,letterpaper]{article}

\usepackage[letterpaper, top=25.4mm, bottom=25.4mm, left=25.4mm, right=25.4mm, includefoot]{geometry}

\usepackage[utf8]{inputenc}
\usepackage[T1]{fontenc}
\usepackage{lmodern}
\usepackage{alphabeta}

\usepackage{titling}
\usepackage{xspace}
\usepackage{soul} 
\usepackage{ifthen}

\usepackage{ifthen}

\usepackage{tikz}
\usepackage{xparse}

\newcommand{\colorname}[1]{%
  \ifnum#1=1 red\else%
    \ifnum#1=2 blue\else%
      \ifnum#1=3 green\else%
        \ifnum#1=4 magenta\else%
          \ifnum#1=5 orange\else%
            black%
          \fi%
        \fi%
      \fi%
    \fi%
  \fi%
}

\usepackage[inline]{enumitem}

\usepackage{dsfont}
\usepackage{setspace}

\usepackage{bm}
\usepackage{microtype}
\usepackage{amsmath}
\usepackage{amssymb}
\usepackage{amsfonts}
\usepackage{mathtools}
\usepackage[mathscr]{euscript}

\usepackage[hyphens]{url}
\usepackage{amsthm}

\usepackage{tikz}
\usepackage{xcolor}
\usepackage{subcaption}
\usepackage{graphicx}
\usepackage{wrapfig}

\usepackage{hyperref}
\usepackage{cleveref}

\usepackage{nicefrac}
\usepackage[textwidth=2.2cm]{todonotes}

\usepackage{aliascnt} 

\colorlet{myGreen}{green!50!black}
\colorlet{myLightgreen}{green}
\colorlet{myRed}{red!90!black}
\definecolor{myBlue}{rgb}{0.25, 0.0, 1.0}
\definecolor{myLightBlue}{rgb}{0.39, 0.58, 0.93}
\colorlet{myViolet}{myBlue!55!myRed}
\definecolor{myOrange}{rgb}{1.0, 0.66, 0.07}

\definecolor{CornflowerBlue}{rgb}{0.39, 0.58, 0.93}
\definecolor{DarkGoldenrod}{rgb}{0.72, 0.53, 0.04}
\definecolor{BritishRacingGreen}{rgb}{0.0, 0.26, 0.15}
\definecolor{DarkMagenta}{rgb}{0.55, 0.0, 0.55}
\definecolor{AO}{rgb}{0.0, 0.5, 0.0}
\definecolor{BostonUniversityRed}{rgb}{0.8, 0.0, 0.0}
\definecolor{myRed}{rgb}{0.8, 0.0, 0.0}
\definecolor{DarkMidnightBlue}{rgb}{0.0, 0.2, 0.4}
\definecolor{DarkTangerine}{rgb}{1.0, 0.66, 0.07}
\definecolor{AppleGreen}{rgb}{0.55, 0.71, 0.0}
\definecolor{BrightUbe}{rgb}{0.82, 0.62, 0.91}
\definecolor{Amethyst}{rgb}{0.6, 0.4, 0.8}
\definecolor{DarkGray}{rgb}{0.52, 0.52, 0.51}
\definecolor{Gray}{rgb}{0.66, 0.66, 0.66}
\definecolor{BananaYellow}{rgb}{1.0, 0.88, 0.21}
\definecolor{Amber}{rgb}{1.0, 0.75, 0.0}
\definecolor{LightGray}{rgb}{0.83, 0.83, 0.83}
\definecolor{PrincetonOrange}{rgb}{1.0, 0.56, 0.0}
\definecolor{DeepCarrotOrange}{rgb}{0.91, 0.41, 0.17}
\definecolor{CarrotOrange}{rgb}{0.93, 0.57, 0.13}
\definecolor{MidnightBlue}{rgb}{0.1, 0.1, 0.44}
\definecolor{Magenta}{rgb}{0.50, 0.0, 0.50}
\definecolor{BrightPink}{rgb}{1.0, 0.0, 0.5}
\definecolor{BrilliantRose}{rgb}{1.0, 0.33, 0.64}
\definecolor{ChromeYellow}{rgb}{1.0, 0.65, 0.0}
\definecolor{HotMagenta}{rgb}{1.0, 0.11, 0.81}

\definecolor{DarkTangerine}{rgb}{1.0, 0.66, 0.07}
\definecolor{darkyellow}{rgb}{.7, .6, 0.0}
\definecolor{CornflowerBlue}{rgb}{0.39, 0.58, 0.93}
\definecolor{DarkGoldenrod}{rgb}{0.72, 0.53, 0.04}
\definecolor{BritishRacingGreen}{rgb}{0.0, 0.26, 0.15}
\definecolor{AO}{rgb}{0.0, 0.5, 0.0}
\definecolor{MidnightBlack}{rgb}{0.1,0.1,.34}
\definecolor{MidnightBlue}{rgb}{0.1,0.1,0.43}
\definecolor{Black}{rgb}{0,0, 0}
\definecolor{Blue}{rgb}{0, 0 ,1}
\definecolor{Red}{rgb}{1, 0 ,0}
\definecolor{White}{rgb}{1, 1, 1}
\definecolor{DeepMagenta}{rgb}{0.8, 0.0, 0.8}
\definecolor{grey}{rgb}{.6, .6, .6}
\definecolor{darkgrey}{rgb}{.33, .33, .33}
\definecolor{Mygreen}{rgb}{.0, .7, .0}
\definecolor{Yellow}{rgb}{.55,.55,0}
\definecolor{Mustard}{rgb}{1.0, 0.86, 0.35}
\definecolor{applegreen}{rgb}{0.55, 0.71, 0.0}
\definecolor{darkturquoise}{rgb}{0.0, 0.81, 0.82}
\definecolor{celestialblue}{rgb}{0.29, 0.59, 0.82}
\definecolor{green_yellow}{rgb}{0.68, 1.0, 0.18}
\definecolor{crimsonglory}{rgb}{0.75, 0.0, 0.2}
\definecolor{darkmagenta}{rgb}{0.30, 0.0, 0.30}
\definecolor{magenta}{rgb}{0.50, 0.0, 0.50}
\definecolor{internationalorange}{rgb}{1.0, 0.31, 0.0}
\definecolor{darkorange}{rgb}{1.0, 0.55, 0.0}
\definecolor{ao}{rgb}{0.0, 0.5, 0.0}
\definecolor{awesome}{rgb}{1.0, 0.13, 0.32}
\definecolor{darkcyan}{rgb}{0.0, 0.50, 0.50}
\definecolor{violet}{rgb}{0.93, 0.51, 0.93}
\definecolor{brown}{rgb}{0.65, 0.16, 0.16}
\definecolor{orange}{rgb}{1.0, 0.65, 0.0}
\definecolor{DarkGreen}{rgb}{0,.5,0}
\definecolor{BostonUniversityRed}{rgb}{0.8, 0.0, 0.0}

\vfuzz \hfuzz
\setlist[itemize]{topsep=0pt,partopsep=0pt,itemsep=0pt,parsep=0pt}
\setlist[itemize,1]{label={\small\textbullet}}
\setlist[itemize,2]{label={\tiny\textbullet}}
\setlist[itemize,3]{label=$\cdot$}
\setlist[enumerate]{topsep=0pt,partopsep=0pt,itemsep=0pt,parsep=0pt}
\setlist[enumerate,1]{label=\roman*)}
\setlist[enumerate,2]{label=\alph*)}
\setlist[enumerate,3]{label=\arabic*)}

\hypersetup{
colorlinks=true,
linkcolor=celestialblue!65!black,
citecolor=celestialblue!65!black,
urlcolor=celestialblue!65!black,
bookmarksopen=true,
bookmarksnumbered,
bookmarksopenlevel=2,
bookmarksdepth=3
}

\newcommand{\DeclareCleverTheorem}[4]{%
  \newaliascnt{#1}{#2}%
  \newtheorem{#1}[#1]{#3}%
  \aliascntresetthe{#1}
  \crefname{#1}{#3}{#4}%
  \crefformat{#1}{##2#3~##1##3}%
  \Crefformat{#1}{##2#3~##1##3}%
  \newtheorem*{#1*}{#3}
  \crefname{#1*}{#3}{#4}%
  \crefformat{#1*}{##2~#1##3}%
  \Crefformat{#1*}{##2~#1##3}%
}

\DeclareCleverTheorem{proposition}{environment}{Proposition}{Propositions}
\DeclareCleverTheorem{corollary}{environment}{Corollary}{Corollaries}
\DeclareCleverTheorem{theorem}{environment}{Theorem}{Theorems}
\DeclareCleverTheorem{conjecture}{environment}{Conjecture}{Conjectures}
\DeclareCleverTheorem{observation}{environment}{Observation}{Observations}
\DeclareCleverTheorem{example}{environment}{Example}{Examples}
\DeclareCleverTheorem{remark}{environment}{Remark}{Remarks}
\DeclareCleverTheorem{notation}{environment}{Notation}{Notations}
\DeclareCleverTheorem{question}{environment}{Question}{Questions}
\DeclareCleverTheorem{problem}{environment}{Problem}{Problems}
\DeclareCleverTheorem{definition}{environment}{Definition}{Definitions}
\DeclareCleverTheorem{lemma}{environment}{Lemma}{Lemmas}

\crefname{figure}{figure}{figures}
\crefformat{figure}{#2Figure~#1#3}
\Crefformat{figure}{#2Figure~#1#3}

\crefname{equation}{equation}{Equations}
\crefformat{equation}{#2Equation~#1#3}
\Crefformat{equation}{#2Equation~#1#3}

\crefname{chapter}{chapter}{chapters}
\crefformat{chapter}{#2Chapter~#1#3}
\Crefformat{chapter}{#2Chapter~#1#3}

\crefname{section}{section}{sections}
\crefformat{section}{#2Section~#1#3}
\Crefformat{section}{#2Section~#1#3}

\crefname{algorithm}{algorithm}{algorithms}
\crefformat{algorithm}{#2Algorithm~#1#3}
\Crefformat{algorithm}{#2Algorithm~#1#3}

\newtheorem{claim}{Claim}
\crefname{claim}{claim}{claims}
\crefformat{claim}{#2Claim~#1#3}
\Crefformat{claim}{#2Claim~#1#3}

\usetikzlibrary{calc}
\usetikzlibrary{fit}
\usetikzlibrary{decorations}
\usetikzlibrary{decorations.pathmorphing}
\usetikzlibrary{decorations.text}
\usetikzlibrary{shapes,hobby}
\usetikzlibrary{positioning}

\tikzset{
	position/.style args={#1:#2 from #3}{
		at=($(#3)+(#1:#2)$)
	}
}

\tikzset{
  v:main/.style = {draw, circle, scale=0.8, thick,fill=black,inner sep=0.7mm},
  v:ghost/.style = {inner sep=0pt,scale=1},
  >={latex},
  e:marker/.style = {line width=8.5pt,line cap=round,opacity=0.35,color=DarkGoldenrod},
  e:main/.style = {line width=1pt},
}

\newcommand{\Bcal}{\mathcal{B}}
\newcommand{\Ccal}{\mathcal{C}}
\newcommand{\Dcal}{\mathcal{D}}
\newcommand{\Ecal}{\mathcal{E}}
\newcommand{\Fcal}{\mathcal{F}}

\newcommand{\Hcal}{\mathcal{H}}

\newcommand{\Lcal}{\mathcal{L}}

\newcommand{\Ocal}{\mathcal{O}}
\newcommand{\Pcal}{\mathcal{P}}
\newcommand{\Qcal}{\mathcal{Q}}
\newcommand\Rcal{\mathcal{R}}
\newcommand{\Scal}{\mathcal{S}}
\newcommand{\Tcal}{\mathcal{T}}

\newcommand{\Vcal}{\mathcal{V}}

\newcommand{\Ycal}{\mathcal{Y}}

\newcommand{\Nbbb}{\mathbb{N}}

\RequirePackage{stmaryrd}
\usepackage{textcomp}
\DeclareUnicodeCharacter{2286}{\subseteq}
\DeclareUnicodeCharacter{2192}{\ifmmode\to\else\textrightarrow\fi}
\DeclareUnicodeCharacter{2203}{\ensuremath\exists}
\DeclareUnicodeCharacter{183}{\cdot}
\DeclareUnicodeCharacter{2200}{\forall}
\DeclareUnicodeCharacter{2264}{\leq}
\DeclareUnicodeCharacter{2265}{\geq}
\DeclareUnicodeCharacter{8614}{\mathbin{\mapsto}}
\DeclareUnicodeCharacter{8656}{\Leftarrow}
\DeclareUnicodeCharacter{8657}{\Uparrow}
\DeclareUnicodeCharacter{8658}{\Rightarrow}
\DeclareUnicodeCharacter{8659}{\Downarrow}
\DeclareUnicodeCharacter{8669}{\rightsquigarrow}
\newcommand{\eqdef}{\stackrel{{\scriptsize\rm def}}{=}}
\DeclareUnicodeCharacter{8797}{\eqdef}
\DeclareUnicodeCharacter{8870}{\vdash}
\DeclareUnicodeCharacter{8873}{\Vdash}
\DeclareUnicodeCharacter{22A7}{\models}
\DeclareUnicodeCharacter{9121}{\lceil}
\DeclareUnicodeCharacter{9123}{\lfloor}
\DeclareUnicodeCharacter{9124}{\rceil}
\DeclareUnicodeCharacter{2208}{\in}
\DeclareUnicodeCharacter{9126}{\rfloor}
\DeclareUnicodeCharacter{9655}{\triangleright}
\DeclareUnicodeCharacter{9665}{\triangleleft}
\DeclareUnicodeCharacter{9671}{\diamond}
\DeclareUnicodeCharacter{9675}{\circ}
\DeclareUnicodeCharacter{10178}{\bot}
\DeclareUnicodeCharacter{10214}{} 
\DeclareUnicodeCharacter{10215}{} 
\DeclareUnicodeCharacter{10229}{\longleftarrow}
\DeclareUnicodeCharacter{10230}{\longrightarrow}
\DeclareUnicodeCharacter{10231}{\longleftrightarrow}
\DeclareUnicodeCharacter{10232}{\Longleftarrow}
\DeclareUnicodeCharacter{10233}{\Longrightarrow}
\DeclareUnicodeCharacter{10234}{\Longleftrightarrow}
\DeclareUnicodeCharacter{10236}{\longmapsto}
\DeclareUnicodeCharacter{10238}{\Longmapsto} 
\DeclareUnicodeCharacter{10503}{\Mapsto}    
\DeclareUnicodeCharacter{10971}{\mathrel{\not\hspace{-0.2em}\cap}}
\DeclareUnicodeCharacter{65294}{\ldotp}
\DeclareUnicodeCharacter{65372}{\mid}

\newcommand{\poly}{\mathsf{poly}\xspace}

\newcommand{\bd}{{\sf bd}\xspace}

\newcommand{\remove}[1]{}

\newcommand{\lst}{\mathsf{lst}\xspace}

\newcommand{\trace}{\mathsf{trace}\xspace}

\newcommand{\cupall}{{\pmb{\bigcup}}}

\newenvironment{cproof}{\proof[Proof of claim]}{\endproof}

\newcommand{\lin}[1]{\langle #1\rangle}

\newboolean{appear}
\setboolean{appear}{true}

 \ifthenelse{\boolean{appear}}{%

  \newcommand{\SW}[1]{%
    \todo[linecolor=red,bordercolor=white,color=Magenta,
    backgroundcolor=CornflowerBlue!40,textcolor=BrightPink,
    size=\footnotesize]{\textsf{SW: #1}}}

  \newcommand{\SWin}[1]{%
    \todo[inline,bordercolor=white,color=Magenta,
    backgroundcolor=CornflowerBlue!40,textcolor=BrightPink,
    size=\footnotesize]{\textsf{SW: #1}}}

  \newcommand{\EP}[1]{%
    \todo[linecolor=red,bordercolor=white,color=DeepCarrotOrange,
    backgroundcolor=BananaYellow!40,textcolor=DeepCarrotOrange,
    size=\footnotesize]{\textsf{EP: #1}}}

  \newcommand{\EPin}[1]{%
    \todo[inline,bordercolor=white,color=DeepCarrotOrange,
    backgroundcolor=BananaYellow!40,textcolor=DeepCarrotOrange,
    size=\footnotesize]{\textsf{EP: #1}}}

  \newcommand{\sed}[1]{%
    \todo[linecolor=red,bordercolor=white,color=BrilliantRose,
    backgroundcolor=BrilliantRose!50,size=\footnotesize]{%
    \textsf{\footnotesize SED: #1}}}

  \newcommand{\VP}[1]{%
    \todo[linecolor=red,bordercolor=white,color=purple!40,
    backgroundcolor=cyan!20,textcolor=magenta,
    size=\footnotesize]{\sf Vagelis: #1}}

  \newcommand{\VPin}[1]{%
    \todo[inline,bordercolor=white,color=purple!40,
    backgroundcolor=cyan!20,textcolor=magenta,
    size=\footnotesize]{\sf Vagelis: #1}}

}{%

  \newcommand{\SW}[1]{}
  \newcommand{\SWin}[1]{}
  \newcommand{\EP}[1]{}
  \newcommand{\EPin}[1]{}
  \newcommand{\sed}[1]{}
  \newcommand{\VP}[1]{}
  \newcommand{\VPin}[1]{}

}

\title{The Local Structure Theorem for Graph Minors with finite index\thanks{The first, second, and third author were supported by the French-German Collaboration ANR/DFG Project UTMA (ANR-20-CE92-0027). The first and third author were also supported by the ANR project GODASse ANR-24-CE48-4377. The second author was also supported by the ERC project BUKA (n°\! 101126229). The third author was also supported by the Franco-Norwegian project PHC AURORA 2024-2025 (Projet n°\! 51260WL).}}

\author{
Christophe Paul\\{\small LIRMM, Univ Montpellier, CNRS, Montpellier, France}\thanks{\href{mailto:paul@lirmm.fr}{paul@lirmm.fr}}
\and
Evangelos Protopapas\\{\small Faculty of Mathematics, Informatics and Mechanics, University of Warsaw, Poland}\thanks{\href{mailto:eprotopapas@mimuw.edu.pl}{eprotopapas@mimuw.edu.pl}}
\and
Dimitrios M. Thilikos\\{\small LIRMM, Univ Montpellier, CNRS, Montpellier, France}\thanks{\href{mailto:sedthilk@thilikos.info}{sedthilk@thilikos.info}}
\and
Sebastian Wiederrecht\\{\small KAIST, School of Computing, Daejeon, South Korea}\thanks{\href{mailto:wiederrecht@kaist.ac.kr}{wiederrecht@kaist.ac.kr}}
}

\date{}

\begin{document}

\maketitle
\vspace{-15mm}

\begin{abstract}
\noindent The \textsl{Local Structure Theorem} (LST) for Graph Minors roughly states that for every $H$-minor-free graph $G$ that contains a sufficiently large wall $W$, there is a small vertex subset $A,$ whose removal yields a graph that admits an ``almost embedding'' $\delta$ on a surface $\Sigma$ on which $H$ does not embed.
By \emph{almost embedding}, we mean that there exists a hypergraph $\mathcal{H}$ whose vertex set is a subset of the vertex set of $G - A$ and an embedding of $\Hcal$ on $\Sigma$ such that the drawing of each hyperedge of $\Hcal$ corresponds to a cell of $\delta,$ the boundary of each cell intersects only the vertices of the corresponding hyperedge, and all remaining vertices and edges of $G - A$ are drawn in the interior of cells.
The cells corresponding to hyperedges of arity at least $4$ -- called \emph{vortices} -- are few in number and have small ``depth'', while ``most'' of the wall $W$ is disjoint from the vortices and is ``grounded'' in the embedding $\delta$.

Suppose that the subgraphs drawn inside each of the non-vortex cells are equipped with some finite index, i.e., each such cell is assigned a \emph{color} from a finite set. 
We prove a version of the LST in which the set $C$ of colors assigned to the non-vortex cells exhibits ``large'' \emph{bidimensionality}: $G - A$ contains a minor model of a large grid $\Gamma$ such that, for every color $\alpha \in C,$ the model of \textsl{each} vertex of $\Gamma$ contains the subgraph drawn within an $\alpha$-colored cell.
Moreover, $\Gamma$ can be chosen in a way that is ``well-connected'' to the original wall $W$.
\end{abstract}


\thispagestyle{empty}

\newpage
\thispagestyle{empty}

\tableofcontents
\thispagestyle{empty}
\newpage

\setcounter{page}{1}

\section{Introduction}\label{sec_intro}

A graph $H$ is a \emph{minor} of a graph $G$ if it can be obtained from a subgraph of $G$ by \textsl{contracting} edges.
A central result, emerging from the highly influential Graph Minors Series by Robertson and Seymour, is the \textsl{Graph Minor Structure Theorem (GMST)}, which describes the structure of graphs that exclude a non-planar graph as a minor.
Considered one of the deepest results in discrete mathematics, the GMST has since found countless applications in algorithmic and structural graph theory, and has significantly contributed to the development of parameterized complexity \cite{Thilikos2012GraphMinors,FellowsL94Nonconstructive,LokshtanovSZ20Efficient}.

At the heart of the proof of the GMST -- and perhaps one of the most important result in the Graph Minors series -- lies the so-called \textsl{Local Structure Theorem (LST)} \cite{RobertsonS03GraphMinorsXVI}, which describes the structure of graphs excluding a graph $H$ as a minor ``locally’’.
The original proof by Robertson and Seymour, which appears as Theorem 3.1 in \cite{RobertsonS03GraphMinorsXVI}, does not provide explicit bounds for the functions involved.
A second group of authors -- Kawarabayashi, Thomas, and Wollan \cite{KawarabayashiTW18anew,KawarabayashiTW2021Quickly} -- offered a new proof of the LST and GMST, introducing the first explicit bounds, which are single-exponential in the size of the excluded minor. They also conjectured that polynomial bounds should be achievable.
These two papers by KTW also greatly simplified the proof of the GMST. Recently, Gorsky, Seweryn, and Wiederrecht \cite{GorskySW2025Polynomial} confirmed this conjecture, providing the first polynomial bounds for the LST and, consequently, for the GMST.
Meanwhile, numerous other variants and refinements of the LST and GMST have been developed in various settings (see \cite{MorellePTW25Excluding,ThilikosW25TheGraph,Chuzhoy15Improved,DvorakT14List} for a small sample).

\medskip
In this work we prove a stronger version of the LST by extracting additional structure provided by a finite index on boundaried graphs.
To explain the LST, as well as our results, we introduce a few key concepts.

\paragraph{Tangles.}

Informally, the LST decomposes an $H$-minor-free graph into parts that are ‘highly linked''.
In the original formulation of the LST by Robertson and Seymour, ``high linkedness'' is captured via the notion of ``tangles''.

A \emph{separation} in a graph $G$ is a pair $(A, B)$ such that $A \cup B = V(G)$ and there is no edge in $G$ with one endpoint in $A \setminus B$ and the other in $B \setminus A$.
The \emph{order} of $(A, B)$ is $|A \cap B|$.
A \emph{tangle} of order $k$ in $G$ is a subset $\Tcal$ of the set of all separations of order $< k$ in $G$ such that 
\begin{enumerate}
\item for every such separation $(A, B)$ exactly one of $(A, B)$ and $(B, A)$ belongs to $\Tcal$ and 
\item for all $(A_{1}, B_{1}), (A_{2}, B_{2}), (A_{3}, B_{3}) \in \Tcal$ it holds that $G[A_{1}] \cup G[A_{2}] \cup G[A_{3}] \neq G$.
\end{enumerate}
Given a tangle $\Tcal$ and $(A, B) \in \Tcal$, we call $A$ the \emph{small side} and $B$ the \emph{big side} of $(A, B)$.
Tangles were introduced by Robertson and Seymour in \cite{RobertsonS1991GraphMinorsX}, have since been extensively studied in combinatorics, and have applications in other fields (see the recent book by Diestel \cite{Diestel24Tangles}).

In order to present the LST, we first need to understand how a tangle in a graph $G$ may \textsl{control} a minor $H$ of $G$.
For this, we first need a concrete object to model $H$ as a minor of $G$.
It is well-known that $H$ is a minor of $G$ if and only if there exists a function $\mu$ that maps vertices of $H$ to pairwise disjoint connected subsets of $V(G)$ such that for each edge $uv \in E(H),$ there exists an edge $xy \in E(G)$ with $x \in \mu(u)$ and $y \in \mu(v)$.
We call $\mu$ a \emph{model} of $H$ in $G$ and the sets $\mu(u)$ the \emph{branch sets} of $\mu$.
If $\Tcal$ is a tangle in $G$, we say that a model $\mu$ of $H$ in $G$ is \emph{controlled} by $\Tcal$ if for every separation $(A, B) \in \Tcal$ of order $< |V(H)|$ all branch sets of $\mu$ intersect the big side of $(A, B)$.

In their version of the LST, Kawarabayashi, Thomas, and Wollan \cite{KawarabayashiTW2021Quickly} use \textsl{walls} instead of tangles (for a wall, see \cref{fig_grid_wall}).
One reason for this is that walls are concrete objects that witness ``well linkedness'' and as such are easier to work with without sacrificing any generality as walls naturally induce tangles.
To formalize this, let $G$ be a graph and $W$ be a $k$-wall in $G$ for some $k \geq 3$.
Notice that for every separation $(A, B)$ of order $< k$ in $G$, there is a unique side, say $B$, that contains the vertex set of both a horizontal and a vertical path of $W$.
Then, the set $\Tcal_{W}$ of all such separations oriented towards $B$ defines a tangle of order $k$ in $G$.
Now, given a tangle $\Tcal$ in $G$, we say that $\Tcal$ \emph{controls} $W$ if $\Tcal_{W} \subseteq \Tcal$.

\paragraph{The local structure theorem.}

The next step is to understand the notion of \textsl{almost embedding} a graph in a \textsl{surface}.
A \emph{boundaried subgraph} of $G$ is a pair $(H, B)$ where $H$
is a subgraph of $G$ and $B$ is a set of distinguished vertices 
of $H$ such that all vertices of $H$ that are adjacent to vertices not in $H$ belong to $B$.
A rough way to describe a \emph{$\Sigma$-decomposition} of $G$ is a collection $\delta = \{(F_{1},B_{1}), \ldots, (F_{\ell},B_{\ell})\}$
of boundaried subgraphs of $G$ such that 1) $G = F_{1} \cup \ldots \cup F_{\ell}$ and 2) there exists an embedding of a hypergraph with vertex set $B_{1} \cup \ldots \cup B_{\ell}$ and hyperedge set $\{ B_{1}, \ldots, B_{\ell}\}$ in $\Sigma$ such that a) each hyperedge $B_{i}$ is drawn as a closed curve whose boundary intersects only the vertices in $B_{i}$ and b) the drawing of any two hyperedges can intersect only in vertices.
Note that the embedding of $\Hcal$ in $\Sigma$ defines a cyclic ordering of each boundary $B_{i}$, say $\Omega_{i}$.
We call each boundaried subgraph $(F_{i}, B_{i})$ with $|B_{i}| \leq 3$ a \emph{cell} of $\delta$ and with $|B_{i}| \geq 4$ a \emph{vortex} of $\delta$.
The \emph{depth} of a vortex $(F_{i},B_{i})$ is the maximum number 
of vertex-disjoint paths between a subset $S \subseteq B$ appearing consecutively on $\Omega_{i}$ to its complement $B \setminus S$.
The \emph{breadth} of $\delta$ is the total number of its vortices while its \emph{depth} is the maximum depth of its vortices.
A $k$-wall $W \subseteq G$ is \emph{grounded} in $\delta$ if every cycle of $W$ uses edges from at least two different cells of $\delta$ and none of its edges belongs to a vortex. 

With these definitions the LST can be stated as follows.

\begin{proposition}[\!\!\cite{GorskySW2025Polynomial}]\label{prop_whatisknown}
There exist functions $\lst^{1} \colon \Nbbb \to \Nbbb$ and 
$\mathsf{wall}^{1} \colon \Nbbb^{2} \to \Nbbb$ with $\lst^{1}(t) \in \mathsf{poly}(k)$ and $\mathsf{wall}^{1}(k, r) \in \mathsf{poly}(k) \cdot r$, such that for every positive integer $r$, every $k$-vertex graph $H$, and every graph $G$ with a $\mathsf{wall}^{1}(k, r)$-wall subgraph $W$, one of the following holds.
\begin{enumerate}
    \item there exists a model of $H$ in $G$ controlled by $\Tcal_{W}$, or 
    \item there exists a surface $\Sigma$ where $H$ does not embed, a set $A \subseteq V(G)$ with $|A| \leq \lst^{1}(k)$, and an $r$-wall $W'$ that is a subwall of $W$ such that $G - A$ has a $\Sigma$-decomposition $\delta$ of breadth at most $ \nicefrac{1}{2}(k - 3)(k - 4)$ and depth at most $\lst^{1}(k)$ and where $W'$ is grounded.
\end{enumerate}
\end{proposition}

Note that the fact that $W'$ is a subwall of $W$ in the statement above implies that the tangle $\Tcal_{W}$ controls $W'$ and combined with the fact that $W'$ is grounded in $\delta$ witness the fact that the decomposition is local: It describes the structure of $G$ with respect to $W$, i.e., the tangle $\Tcal_{W}$.

\paragraph{Boundaried subgraphs with respect to a fixed tangle.}

Let us pick apart and analyze the LST.
Let $G$ be an $H$-minor-free graph and $\Tcal$ be a tangle in $G$ of large order.
As we have already discussed, the LST describes the structure of $G$ locally with respect to $\Tcal$.
After possibly removing a bounded (in the size of $H$) number of the so-called \textsl{apex vertices} in $A$, $G' \coloneqq G - A$ admits a $\Sigma$-decomposition $\delta = \{ (F_{1}, B_{1}), \ldots, (F_{\ell}, B_{\ell})\}$ into boundaried subgraphs of $G'$ in a surface $\Sigma$ of bounded Euler genus.
An important thing to note is that the boundaried graphs in the $\Sigma$-decomposition $\delta$ \textit{depend} on the choice of tangle $\Tcal$, hence the locality of the structure (relatively to $\Tcal$).
Moreover, the boundaried graphs in $\delta$ may be partitioned into the cells $\Ccal$ of $\delta$ and the vortices $\Vcal$ of $\delta$.
Notice that, even though vortices may have arbitrarily large boundary, they are however restricted to be few and of bounded depth.
On the other hand, the decomposition does not retain any information or impose any condition regarding cells, which moreover may be arbitrarily many.
Hence, a natural attempt for a stronger structure would be to capture some finite-index property of the boundaried graphs in $\Ccal$.
Towards this, we introduce \textsl{boundaried indices}.
In essence, a \emph{boundaried index} $\iota$ is a function mapping each boundaried subgraph $(F, B)$ with $|B| \leq 3$ to a positive integer.
What is important is that the number of distinct indices assigned to by $\iota$ is \textsl{finite}.
For this we define the \emph{capacity} of $\iota$ as the maximum of its images. 
The question that naturally appears is the following:
\begin{quote}
\textsl{Can we extract some additional information regarding the different indices associated by some boundaried index $\iota$ to the cells controlled by a tangle $\Tcal$?}
\end{quote}

\paragraph{An indexed extension of the LST.}

In this work we prove a strengthened version of the LST, which roughly states that the indices associated to the cells of $\delta$ are ``\textsl{bidimensionally}'' dispersed across the embedding associated to $\delta$ in a way that is ``controlled'' by the tangle $\Tcal$.
We formalize this property in terms of yet another witness of ``high linkedness'': \textsl{Grids} (see \cref{fig_grid_wall}).
Grids and walls are closely related in the sense that large grids contain large walls as subgraphs, while large walls contain large grids as minors.

Let $\iota$ be a boundaried index, $G$ be a graph, and $\delta$ be a $\Sigma$-decomposition of $G$.
Given $t \geq 3$, we say that a minor model $\mu$ of a $(t \times t)$-grid in $G$ \emph{represents} $\iota$ in $\delta$ if $\mu(V(\Gamma))$ is disjoint from any non-boundary vertex of the vortices of $\delta$ and, for \textsl{every} vertex $u \in V(\Gamma)$ and  for \textsl{every} index $\alpha$ assigned to some cell of $\delta$, there is an equally indexed  cell $(F, B)$ of $\delta$ (i.e., $\iota(F, B) = \alpha$) such that $F$ is a subgraph of $\mu(u)$.
Notice that the above notion of representation is quite strong.
It demands that representatives of \textsl{all} cells of $\delta$ appear in \textsl{all} branch sets of the $(t \times t)$-grid model $\mu$.

Moreover, if $\Tcal$ is a tangle of $G$, we say that a minor model $\mu$ of a $(t \times t)$-grid $\Gamma$ in $G$ is \emph{controlled} by $\Tcal$ if there do not exist $(A, B) \in \Tcal$ of order $< t$ and $X \subseteq V(\Gamma)$ where $\Gamma$ is the vertex set of a row or column of $\Gamma$ such that $\mu(X) \subseteq A \setminus B$.
 
Our version of the LST extends \cref{prop_whatisknown} as follows.

\begin{theorem}\label{our_result_first}
There exist functions $\lst^{2} \colon \Nbbb^3 \to \Nbbb$ and 
$\mathsf{wall}^{2} \colon \Nbbb^{4} \to \Nbbb$ such that, for every boundaried index $\iota$ of capacity at most $\ell$, every two positive integers $r$ and $t$, every $k$-vertex graph $H$, and every graph $G$ with a $\mathsf{wall}(r, k, t, \ell)$-wall $W$ subgraph, one of the following holds.
\begin{enumerate}
\item there is a minor model of $H$ in $G$ controlled by $\Tcal_{W}$, or
\item there is a surface $\Sigma$ where $H$ does not embed, a set $A \subseteq V(G)$ with $|A| \leq \lst^{2}(k, t, \ell)$, and an $(r + t)$-wall $W'$ that is a subwall of $W$, such that $G - A$ has a $\Sigma$-decomposition $\delta$ of breadth at most $k^{2} + t^{2} \cdot 2^{\Ocal(\ell)}$ and depth at most $\lst^{2}(k, t, \ell)$ and where $W'$ is grounded.

Moreover, there is a minor model of a $(t \times t)$-grid in $G - A$ that is controlled by $\Tcal_{W'}$ and represents the boundaried index $\iota$ in $\delta$.
\end{enumerate}
\end{theorem}

{A theorem very similar to \cref{our_result_first} was originally proven by Robertson and Seymour themselves \cite{RobertsonS1995GMXIII}.
In volume XIII of the Graph Minors Series, Robertson and Seymour introduce their celebrated algorithm for the $k$-\textsc{Disjoint Paths} problem.
The core of this algorithm is the so-called \textsl{irrelevant vertex technique} and one of the two core results here says that in any large enough wall there exists a vertex whose deletion does not influence the outcome of $k$-\textsc{Disjoint Paths}.
In order to find such a vertex, Robertson and Seymour -- besides the introduction of the \textsl{linkage function} \cite{RobertsonS2009GMXXI} -- prove an indexed version of the so-called \emph{Flat Wall Theorem} (see Section 10 in \cite{RobertsonS1995GMXIII}).
In this context, one may understand \cref{our_result_first} as a generalization of this original indexed Flat Wall Theorem, similar to how the LST is a strict generalization of the non-indexed version of the Flat Wall Theorem.}

It should be pointed out that \cref{our_result_first} is a mere corollary which follows from a much more general version of the LST, namely \cref{thm_local_structure} proved in \cref{sec_proving_lst}, which equips the $\Sigma$-decomposition with a lot of ``canonically'' defined wall-like infrastructure that represents the given boundaried index $\iota$.
Due to the technical nature of these definitions, we choose not to present this theorem in this brief introduction.

Let us remark that all our results are constructive, providing polynomial (in the size of $G$) time algorithms.
Moreover, in the full statement of \cref{our_result_first}, namely \cref{cor_bidim_structure}, we give explicit estimations for the functions $\lst$ and $\mathsf{wall}$   as follows:
$$\lst(k,t,\ell) = 2^{(k^{2} + t^{2}) \cdot \log t \cdot 2^{\Ocal(\ell)}} \ \text{ and }\  \mathsf{wall}(r, k, t, \ell) = (r + \lst(k, t, \ell))^{2^{\Ocal(\ell)}} \cdot k^{2^{\Ocal(\ell)}}.$$

\paragraph{Brief outline.}

\cref{thm_local_structure} is the result of our main technical tool \cref{main_lemma} presented in \cref{sec_main_lemma}, and whose proof spans \cref{sec_homogeneous_walloid}, \cref{sec_representation}, and \cref{sec_splitting_extracting}.
The proof of \cref{main_lemma} proceeds in three main steps each corresponding to one of the sections above, an \textsl{homogenization} step (\cref{sec_homogeneous_walloid}), a \textsl{representation} step (\cref{sec_representation}), and a \textsl{coarsening} step (\cref{sec_splitting_extracting}), which start from a $\Sigma$-decomposition $\delta$ as the one produced by \cref{prop_whatisknown}, and progressively produce a \emph{coarser} version $\delta'$ of $\delta$, i.e., such that every cell of $\delta'$ is a cell of $\delta$ and all such cells are ``represented'' by a lot of wall-like infrastructure accompanying the $\Sigma$-decomposition at each step, and such that every cell of $\delta$ whose index is not represented is ``cornered'' within the vortices of $\delta'$.
After the last step, $\delta'$ along with the accompanying wall-like structure yields the second outcome of \cref{our_result_first}.

\paragraph{Future work.}

We believe that the proof techniques used in this paper are unlikely to yield polynomial bounds.
There are two main sources of exponential blow-up in our arguments.
The first arises from the homogenization step, which would require techniques capable of homogenizing polynomially large walls with respect to a finite number of colors.
The second source is the coarsening step, which would demand a deep understanding and careful adaptation of the ideas and techniques developed in \cite{GorskySW2025Polynomial}.

We expect that our local structure theorem will serve as a foundation for numerous applications in the area of algorithmic graph minors. {Such an application has already appeared in \cite{protopapas2025colorfulminors} which considers minors of colorful graphs, where each vertex of a colorful graph may carry some (possibly empty) set of a bounded number of colors.
Using the results of this work as a starting point, a structural theory of ``colorful minors'' is developed, with various combinatorial and algorithmic applications.}
{Moreover, a preliminary version of this theorem has been used as an integral building block in the recently announced full classification of the \textsl{Erd\H{o}s-P{\'o}sa property} in minor-closed graph classes \cite{PaulPTW2024Obstructions}.}
Our results are presented in a form that facilitates the transition from the local to a global structure theorem, which can be adapted depending on the intended application.

\section{Preliminaries}\label{sec_preliminaries}

In this first preliminary section we introduce notions around graphs and their minors, as well as separations in graphs, walls, and tangles.

\medskip
We use $\Nbbb$ to denote the set of non-negative integers.
Given a positive integer $c \in \Nbbb$, we denote by $\Nbbb_{\geq c}$ the set $\{ x \in \Nbbb \mid c \leq x\}$.
Given two integers $a, b \in \Nbbb$, we denote by $[a, b]$ the set $\{ x \in \Nbbb \mid a \leq x \leq b\}$.
Notice that $[a, b]$ is empty whenever $b < a$.
Given a positive integer $c \in \Nbbb$, we denote by $[c]$ the set $[1, c]$.

\paragraph{Minors.}

Given a graph $G$ and an edge $e = uv$, the \emph{contraction} of $e$ is the operation which yields a new graph in which $u$ and $v$ are identified and any loops or parallel edges that arise are deleted.
We say that $H$ is a \emph{minor} of $G$ or $G$ has an \emph{$H$-minor} if a graph isomorphic to $G$ can be obtained from a subgraph of $G$ by a sequence of edge contractions.

\paragraph{Separations.}

A \emph{separation} in a graph $G$ is a pair $(A, B)$ such that $A \cup B = V(G)$ and there is no edge in $G$ with one endpoint in $A \setminus B$ and the other in $B \setminus A$.
The \emph{order} of $(A, B)$ is $|A \cap B|$.

\paragraph{Paths and linkages.}

A \emph{linkage} $\Lcal$ in a graph $G$ is a set of pairwise vertex-disjoint paths.
In a slight abuse of notation, we use $V(\mathcal{L})$ and $E(\mathcal{L})$ to denote the sets $\bigcup_{L \in \mathcal{L}} V(L)$ and $\bigcup_{L \in \mathcal{L}} E(L)$ respectively.
We say that a path $P$ in $G$ is \emph{internally disjoint} from a set $X \subseteq V(G)$ if $V(P) \cap X$ does not contain any vertex of $P$ that is not an endpoint vertex.
Given a graph $G$ and two subsets $A, B \subseteq V(G)$, an \emph{$A$-path} in $G$ is a path with both endpoints in $A$ and internally disjoint from $A$, and an \emph{$A$-$B$-path} is a path with one endpoint in $A$, the other in $B$, and internally disjoint from $A \cup B.$
An \emph{$A$-$B$ linkage} in $G$ is a linkage consisting of $A$-$B$ paths.
If $H$ is a subgraph of $G$, an \emph{$H$-path} is a $V(H)$-path of length at least one with no edge in $E(H)$.

More often that not we are interested in minors which are \textsl{subcubic}\footnote{A graph of maximum degree at most $3$ is called \emph{subcubic}.} graphs.
In this case we shall avoid speaking of minors directly as we may equivalently find a \textsl{subdivision} of it as a subgraph.
We say that a graph $G'$ is a \emph{subdivision} of a graph $G$ if it can be obtained from $G$ by replacing any edge $uv \in E(G)$ by a $u$-$v$ path whose internal vertices are disjoint from $G$.

\paragraph{Extended graphs.}

To model certain structures in our graphs we employ an intermediate notion between a graph and a hypergraph which we call an \textsl{extended} graph.

\medskip
An extended graph is a triple $H = (V, E, \mathcal{E})$ where $(V, E)$ is a graph and $\mathcal{E}$ is a set of subsets of $V$.
We refer to the elements of $E$ as the \emph{edges} of $H$ and to the elements of $\Ecal$ as the \emph{hyperedges} of $H$.
Note that any set of two vertices which constitute an edge of $G$ can be present in $\mathcal{E}$ as well.
Also, any graph can be seen as an extended graph where the set of its hyperedges is empty.

\subsection{Grids and walls}\label{sec_walls}

Let us now introduce two natural candidates which witness ``well linkedness'' in a graph.

\medskip
We first define grids.
An \emph{$(n \times m)$-grid} is the graph with vertex set $[n] \times [m]$ and edge set
$$\{ (i,j)(i,j+1) \mid i \in [n], j \in [m-1] \} \cup \{ (i,j)(i+1,j) \mid i \in [n-1], j \in [m] \}.$$
We call the path where vertices appear as $(i,1), (i,2), \ldots, (i,m)$ the \emph{$i$-th row} and the path where vertices appear as $(1,j), (2,j), \ldots, (n,j)$ the \emph{$j$-th column} of the grid.
See \cref{fig_grid_wall} for an illustration of the $(5 \times 10)$ grid.

\medskip
We now move on to the definition of walls.
Let  $t, z \in \Nbbb_{\geq 3}.$
An \emph{elementary $(t \times z)$-wall} is obtained from the $(t \times 2z)$-grid by removing a matching $M$ which contains all odd numbered edges of its odd numbered columns and all even numbered edges of its even numbered columns, and then deleting all vertices of degree one.
A \emph{$(t \times z)$-wall} is a subdivision of an elementary $(t \times z)$-wall.
See \cref{fig_grid_wall} for an illustration of an elementary $(5 \times 5)$-wall and how it is obtained from a $(5 \times 10)$-grid.

\begin{figure}[h]
\centering
\includegraphics{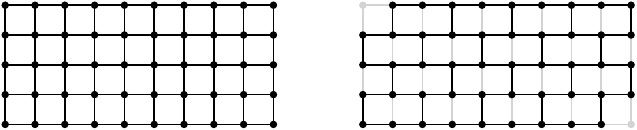}
\caption{\label{fig_grid_wall}A $(5 \times 10)$-grid on the left and an elementary $(5 \times 5)$-wall on the right.}
\end{figure}

Notice that large walls still contain large grids as minors.
Moreover, since they are subcubic, they are easier to handle as it suffices to find the as subgraphs instead of minors.

\subsection{Tangles and the minors they control}

We continue with the more abstract notion we use to model ``well linkedness'', that of a tangle.

\paragraph{Tangles.}
Let $G$ be a graph and $k \in \Nbbb_{\geq 1}$.
We denote by $\Scal_{k}(G)$ the set of all separations of order $< k$ in $G$.
An \emph{orientation} of $\Scal_{k}(G)$ is a set $\Ocal$ such that for all $(A, B) \in \Scal_{k}(G)$ exactly one of $(A, B)$ and $(B, A)$ belongs to $\Ocal$.
A \emph{tangle} of order $k$ in $G$ is an orientation $\Tcal$ of $\Scal_{k}(G)$ such that for all $(A_{1}, B_{1}), (A_{2}, B_{2}), (A_{3}, B_{3}) \in \Tcal$ it holds that $G[A_{1}] \cup G[A_{2}] \cup G[A_{3}] \neq G$.
Given a tangle $\Tcal$ and $(A, B) \in \Tcal$, we call $A$ the \emph{small side} and $B$ the \emph{big side} of $(A, B)$.

Let $G$ be a graph and $\Tcal$ and $\Dcal$ be tangles of $G$.
We say that $\Dcal$ is a \emph{truncation} of $\Tcal$ if $\Dcal \subseteq \Tcal$.

\paragraph{Minors controlled by tangles.}

Let $t, z \in \Nbbb_{\geq 3}$, $G$ be a graph, and $W$ be a $(t \times z)$-wall in $G$.
Let $\Tcal_{W}$ be the orientation of $\Scal_{\min\{ t, z \}}$ such that for every $(A, B) \in \Tcal_{W}$, the set $B \setminus A$ contains the vertex set of both a horizontal and a vertical path of $W$.
Then $\Tcal_{W}$ is the tangle \emph{induced} by $W$.
If $\Tcal$ is a tangle in $G$, we say that $W$ is \emph{controlled} by $\Tcal$ if $\Tcal_{W}$ is a truncation of $\Tcal$.

To formalize how an arbitrary minor $H$ of a graph $G$ is controlled by some tangle $\Tcal$ of $G$ we first need a concrete object to model $H$ as a minor of $G$.
For this we introduce \textsl{minor models}.
A function $\mu \colon V(H) \to 2^{V(G)}$ is called a \emph{(minor) model} of $H$ in $G$ if $\mu(V(H))$ is a set of pairwise disjoint vertex subsets, each inducing a connected subgraph of $G$, and for each edge $uv \in E(H)$ there exists an edge $xy \in E(G)$ with $x \in \mu(u)$ and $b \in \mu(v)$.
Observe that $H$ is a minor of $G$ if and only if there exists a minor model of $H$ in $G$.

If $\Tcal$ is a tangle of $G$, we say that a minor model $\mu$ of $H$ in $G$ is \emph{controlled} by $\Tcal$ if there do not exist $(A, B) \in \Tcal$ of order $< |V(H)|$ and $x \in V(H)$ such that $\mu(x) \subseteq A \setminus B$.

In the special case that $H$ is the $(t \times t)$-grid, for some $t \in \Nbbb_{\geq 3}$, we say that a minor model $\mu$ of $H$ in $G$ is \emph{controlled} by $\Tcal$ if there do not exist $(A, B) \in \Tcal$ of order $< t$ and $X \subseteq V(H)$ where $X$ is the vertex set of a row or column of $H$ such that $\mu(X) \subseteq A \setminus B$.

\section{The Graph Minors toolbox}

In our proofs we closely follow the framework introduced by KTW \cite{KawarabayashiTW2021Quickly} for their proof of the GMST.
In this second preliminary section we provide most of the technical definitions that we require.

\subsection{Paintings, embeddings, and $\Sigma$-decompositions}

We first introduce definitions around almost embedding graphs in surfaces.

\paragraph{Surfaces.}

By \emph{surface} we mean a two-dimensional manifold, possibly with boundary.

Given a pair $(\mathsf{h},\mathsf{c}) \in \mathbb{N} \times [0,2]$, we define $\Sigma^{(\mathsf{h}, \mathsf{c})}$ to be the surface without boundary created from the sphere by adding $\mathsf{h}$ handles and $\mathsf{c}$ crosscaps (for more on the topology of 2-dimensional surfaces, see~\cite{MoharT01Graphs}).
If $\mathsf{c} = 0,$ the surface $\Sigma^{(\mathsf{h},\mathsf{c})}$ is \emph{orientable}, otherwise it is \emph{non-orientable}.
By Dyck's theorem \cite{Dyck1888Beitrage,Francis99ConwayZIP}, two crosscaps are equivalent to a handle in the presence of a (third) crosscap.
This implies that the notation $\Sigma^{(\mathsf{h}, \mathsf{c})}$ is sufficient to denote all surfaces without boundary.

A \emph{closed} (resp. \emph{open}) \emph{disk} of a surface $\Sigma$ is any closed (resp. open) disk that is a subset of $\Sigma.$
For brevity, whenever we use the term \emph{disk} we mean a closed disk.

\paragraph{Painting an extended graph in a surface.}

Let $\Sigma$ be a surface.
A \emph{$\Sigma$-painting} is a triple $\Gamma = (U, V, E)$ with a partition of $E$ into two sets $E_{1}$ and $E_{2}$ such that
\begin{enumerate}
\item $V$ and $E$ are finite,
\item $V \subseteq U \subseteq \Sigma$ and $\bd(\Sigma) \cap U \subseteq V,$
\item $V \cup \bigcup_{e \in E_{1}} e = U,$ $V \cap (\bigcup_{e \in E_{1}} e) = \emptyset,$ and $U \cap (\bigcup_{d \in E_{2}} \bd(d)) \subseteq V,$
\item for every $e \in E_{1},$ $e = h((0,1)),$ where $h \colon [0,1]_{\mathbb{R}} \to U$ is a homeomorphism onto its image with $h(0), h(1) \in V,$
\item for distinct $e, e' \in E_{1},$ $|e \cap e'|$ is finite, and
\item for every $d \in E_{2},$ $d$ is an open disk of $\Sigma$ where $d \cap (\bigcup_{d' \in E_{2} \setminus \{ d \}} d') = \emptyset.$
\end{enumerate}

We call the elements of $V$ the \emph{points}, the elements of $E_{1}$ the \emph{arcs}, and the elements of $E_{2}$ the \emph{disks} of $\Gamma$ respectively.
If $H$ is an extended graph and $\Gamma = (U, V, E)$ is a $\Sigma$-painting whose points, arcs, and disks correspond to the vertices, edges, and hyperedges of $H$ respectively, we say that $\Gamma$ is a \emph{$\Sigma$-painting} of $H$.
Two distinct arcs in $E$ \emph{cross} if they have a common point.
If no two arcs in $E$ cross, we say that $\Gamma$ is a \emph{$\Sigma$-embedding} of $H$.
In that case, the connected components of $\Sigma \setminus U,$ are the \emph{faces} of $\Gamma.$
Given a graph $G,$ we refer to a $\Sigma$-painting of $(V(G), E(G), \emptyset)$ as a \emph{$\Sigma$-drawing} of $G$.
   
\paragraph{$\Sigma$-decompositions.}

Let $\Sigma$ be a surface.
A \emph{$\Sigma$-decomposition} of a graph $G$ is a pair $\delta = (\Gamma, \mathcal{D}),$ where $\Gamma = (U, V, E)$ is a $\Sigma$-drawing of $G$ and $\mathcal{D}$ is a collection of disks of $\Sigma$ such that
\begin{enumerate}
\item the disks in $\mathcal{D}$ have pairwise disjoint interiors,
\item for every disk $\Delta \in \Dcal,$ $\bd(\Delta) \cap U \subseteq V,$ i.e., the boundary of $\Sigma$ and of each disk in $\mathcal{D}$ intersects $U$ only in vertices,
\item if $\Delta_1, \Delta_2 \in \mathcal{D}$ are distinct, then $\Delta_1 \cap \Delta_2 \subseteq V,$ and
\item every edge of $\Gamma$ is a subset of the interior of one of the disks in $\mathcal{D}.$
\end{enumerate}

Let $N$ be the set of all points of $\Gamma$ that do not belong to the interior of any disk in $\mathcal{D}.$
We refer to the elements of $N$ as the \emph{nodes} of $\delta.$
If $\Delta \in \mathcal{D},$ then we refer to the set $\Delta - N$ as a \emph{cell} of $\delta.$
We denote the set of nodes of $\delta$ by $N(\delta)$ and the set of cells of $\delta$ by $C(\delta).$
Given a cell $c \in C(\delta),$ we define the \emph{disk of $c$} as the disk $\Delta_{c} \coloneqq \bd(c) \cup c.$
For a cell $c \in C(\delta),$ the set of nodes of $\delta$ that belong to $\Delta_{c}$ is denoted by $\tilde{c}$.
Moreover, we define the graph $\sigma_{\delta}(c)$ to be the subgraph of $G$ consisting of all vertices and edges that are drawn in $\Delta_{c}.$
We define $\pi_{\delta} \colon N(\delta) \to V(G)$ to be the mapping that assigns to every node in $N(\delta)$ the corresponding vertex of $G.$
We call every vertex in $\pi_{\delta}(N(\delta))$ a \emph{ground vertex} in $\delta$.
A cell $c \in C(\delta)$ is a \emph{vortex cell} if $|\tilde{c}| \geq 4$, otherwise it is a \emph{simple cell}.
We partition $C(\delta)$ into the sets $C_{\mathsf{s}}(\delta)$ and $C_{\mathsf{v}}(\delta),$ containing the simple and vortex cells of $\delta$ respectively.
We say that $\delta$ is \emph{vortex-free} if $C_{\mathsf{v}}(\delta) = \emptyset,$ i.e., if no cell in $C(\delta)$ is a vortex.

Given two $\Sigma$-decompositions $\delta$ and $\delta'$ of a graph $G$, we write
$\delta' \sqsubseteq \delta$ whenever every simple cell of $\delta'$ is a simple cell of $\delta$. When $\delta' \sqsubseteq \delta$, we say that $\delta'$ is a \emph{coarsening} of $\delta$. 
Notice that when $\delta'$ is a coarsening of $\delta$, it means that the subgraphs of $G$ that are drawn in the interior of cells of $\delta$ that are not cells of $\delta'$ are drawn in the interior of vortices of $\delta'$.

\subsection{Societies and more}

We next introduce societies which will allow us to ``localize'' arguments in our proofs, whenever we want to speak about a part of our $\Sigma$-decomposition that is drawn in the interior of some disk of $\Sigma$.

\paragraph{Societies.}

Let $\Omega$ be a cyclic ordering of the elements of some set, which we denote by $V(\Omega)$.
A \emph{society} is a pair $(G, \Omega)$ where $G$ is a graph and $\Omega$ is a cyclic ordering with $V(\Omega) \subseteq V(G)$.

A \emph{cross} in a society $(G, \Omega)$ is a pair $(P_{1}, P_{2})$ of vertex-disjoint paths in $G$ such that $P_{i}$ has endpoints $s_{i}, t_{i} \in V(\Omega)$ and is otherwise disjoint from $V(\Omega)$, and the vertices $s_{1}, s_{2}, t_{1}, t_{2}$ occur in $\Omega$ in the order listed.

Let $(G, \Omega)$ be a society.
A \emph{segment} of $\Omega$ is a set $S \subseteq V(\Omega)$ such that there do not exist $s_1, s_2 \in S$ and $t_1, t_2 \in V(\Omega) \setminus S$ such that $s_1, t_1, s_2, t_2$ occur in $\Omega$ in that order.
A vertex $s \in S$ is an \emph{endpoint} of the segment $S$ if there is a vertex $t \in V(\Omega) \setminus S$ which immediately precedes or immediately succeeds $s$ in $\Omega$.
For vertices $x, y \in V(\Omega)$, we denote by $x \Omega y$ the uniquely determined segment of $\Omega$ with first vertex $x$ and last vertex $y$.

\paragraph{Renditions.}

Let $(G, \Omega)$ be a society and $\Sigma$ be a surface with one boundary component $B$ homeomorphic to the unit circle.
A \emph{rendition} of $(G, \Omega)$ in $\Sigma$ is a $\Sigma$-decomposition $\rho$ of $G$ such that the image under $\pi_{\rho}$ of $N(\rho) \cap B = V(\Omega)$ and $\Omega$ is one of the two cyclic orderings of $V(\Omega)$ defined by the way the points of $\pi_{\delta}(V(\Omega))$ are arranged on the boundary $B$.
A \emph{cylindrical rendition} $(\Gamma, \Dcal, c_{0})$ of a society $(G, \Omega)$ around a cell $c_{0}$ is a rendition $\rho$ of $G$ in a disk such that all cells of $\rho$ that are different from $c_{0}$ are simple cells.

\medskip
The next proposition is a restatement of the celebrated ``\textsl{Two Paths Theorem}'' that gives a characterization of a cross-free society in terms of vortex-free rendition in a disk.
Several proofs of this result have appeared, see \cite{Jung70Verallgemeinerung,RobertsonS1990GraphMinorsIX,shiloach1980polynomial,thomassen1980linked}.
For a recent proof see \cite[Theorem 1.3]{KawarabayashiTW18anew}.

\begin{proposition}
A society has no cross if and only if it has a vortex-free rendition in a disk.
\end{proposition}

\subsection{Transactions and depth}

Due to their large boundary, vortices represent those parts of the $\Sigma$-decomposition that cannot be easily separated from the rest as well as introduce crossings that affect almost embeddability.
However, in the LST vortices are still sufficiently tame as the number of pairwise-disjoint paths that are allowed to traverse through them is bounded.
In our proofs we need tools that allow us to process societies -- eventually turned vortices -- in terms of such sets of paths.

\paragraph{Transactions.}

Let $(G, \Omega)$ be a society. 
A \emph{transaction} in $(G, \Omega)$ is an $A$-$B$-linkage for disjoint segments $A, B$ of $\Omega$ consisting of $V(\Omega)$-paths. 
We define the \emph{depth} of $(G, \Omega)$ as the maximum order of a transaction in $(G, \Omega)$.

Let $\mathcal{T}$ be a transaction in a society $(G, \Omega)$. 
We say that $\mathcal{T}$ is \emph{planar} if no two members of $\mathcal{T}$ form a cross in $(G, \Omega)$.

\paragraph{Aligned disks.}

Let $\delta = (\Gamma = (U, V, E), \mathcal{D})$ be a $\Sigma$-decomposition of a graph $G.$
We call a disk $\Delta$ of $\Sigma$ $\delta$-\emph{aligned} if $U \cap \bd(\Delta) \subseteq N(\delta).$
Clearly, for each $c \in C(\delta),$ $\Delta_c$ is a $\delta$-aligned disk.
Given an arcwise connected set $\Delta \subseteq \Sigma$ such that $U \cap \Delta \subseteq N,$ we use $G \cap \Delta$ to denote the subgraph of $G$ consisting of the vertices and edges of $G$ that are drawn in $\Delta.$
Note that, in particular, the above definition applies when $\Delta$ is a $\delta$-aligned disk of $\Sigma.$
We also define $\Gamma \cap \Delta \coloneqq (U \cap \Delta, V \cap \Delta, \{ e \in E \mid e \subseteq \Delta \})$ and observe that $\Gamma \cap \Delta$ is a $\Delta$-drawing of $G \cap \Delta$.
Given a $\delta$-aligned disk $\Delta$ we define a cyclic ordering $\Omega(\Delta)$ such that $V(\Omega(\Delta)) = \pi_{\delta}(\bd(\Delta) \cap N(\delta)))$, by traversing along the boundary of $\Delta$ in one of the two directions.
Moreover, for a vortex cell $c$ of $\delta$, the \emph{vortex society} of $c$ in $\delta$ is the society $(G \cap \Delta_{c}, \Omega(\Delta_{c}))$.

Given a $\Sigma$-decomposition $\delta = (\Gamma, \Dcal)$ of a graph $G$ and a $\delta$-aligned disk $\Delta$ of $\Sigma,$ we denote by $\delta \cap \Delta$ the rendition $(\Gamma \cap \Delta, \{ \Delta_{c} \in \mathcal{D} \mid c \subseteq \Delta \})$ of the society $(G \cap \Delta, \Omega(\Delta))$ in $\Delta.$

\paragraph{Breadth and depth of a $\Sigma$-decomposition.}

Let $\delta$ be a $\Sigma$-decomposition of a graph $G$ in a surface $\Sigma$.
We define the \emph{breadth} of $\delta$ to be the number of vortex cells of $\delta$ and the \emph{depth} of $\delta$ to be the maximum depth over all vortex societies of $\delta$.

\paragraph{Grounded graphs.}

Let $\delta$ be a $\Sigma$-decomposition of a graph $G.$ 
Let $Q \subseteq G$ be either a cycle or a path that uses no edge of $\cupall \{ \sigma_{\delta}(c) \mid c \in C_{\mathsf{v}}(\delta) \}.$
We say that $Q$ is $\delta$-\emph{grounded} if either $Q$ is a non-trivial path with both endpoints in $\pi_{\delta}(N(\delta))$ or $Q$ is a cycle that contains edges of $\sigma_{\delta}(c_{1})$ and $\sigma_{\delta}(c_{2})$ for at least two distinct cells $c_{1}, c_{2} \in C(\delta).$
A $2$-connected subgraph $H$ of $G$ is said to be \emph{$\delta$-grounded} if every cycle in $H$ is grounded in $\delta.$

\paragraph{Traces.}

Let $\delta$ be a $\Sigma$-decomposition of a graph $G$ in a surface $\Sigma.$ 
For every cell $c \in C(\delta)$ with $|\tilde{c}| = 2$ we select one of the components of $\bd(c) - \tilde{c}.$
This selection is called a \emph{tie-breaker} in $\delta,$ and we assume every $\Sigma$-decomposition to come equipped with a tie-breaker.
If $Q$ is $\delta$-grounded, we define the \emph{trace} of $Q$ as follows.
Let $P_{1}, \dots, P_{t}$ be distinct maximal subpaths of $Q$ such that $P_{i}$ is a subgraph of $\sigma(c)$ for some cell $c.$
Fix an index $i.$
The maximality of $P_{i}$ implies that its endpoints are $\pi_\delta(n_{1})$ and $\pi_\delta(n_{2})$ for distinct nodes $n_{1}, n_{2} \in N(\delta).$
If $|\tilde{c}| = 2,$ define $L_{i}$ to be the component of $\bd(c) - \{ n_{1}, n_{2} \}$ selected by the tie-breaker, and if $|\tilde{c}| = 3,$ define $L_{i}$ to be the component of $\bd(c) - \{ n_{1}, n_{2} \}$ that is disjoint from $\tilde{c}.$
Finally, we define $L'_{i}$ by slightly pushing $L_{i}$ to make it disjoint from all cells in $C(\delta).$
We define such a curve $L'_{i}$ for all $i$ while ensuring that the curves intersect only at a common endpoint.
The \emph{trace} of $Q,$ denoted by $\trace_{\delta}(Q)$ or simply $\trace(Q)$ when $\delta$ is clear from the context, is defined to be $\bigcup_{i \in [t]} L'_{i}.$
So the trace of a cycle is the homeomorphic image of the unit circle and the trace of a path is an arc in $\Sigma$ with both endpoints in $N(\delta).$

We say that a cycle $Q$ of $G$ is \emph{$\delta$-contractible} if $\trace_{\delta}(Q)$ bounds a disk in $\Sigma$.
Given a $\delta$-contractible cycle $C$ of $G$ we define $\mathsf{nodes}_{\delta}(C) \coloneqq \trace_{\delta}(C) \cap N(\delta)$ to be the set of nodes of $\delta$ intersected by $\trace_{\delta}(C).$
Moreover, given a set $B \subseteq N(\delta)$ such that $\mathsf{nodes}_{\delta}(C) \cap B = \emptyset$, we call the disk of $\Sigma$ that is bounded by $\trace_{\delta}(C)$ and is disjoint from $B,$ the \emph{$B$-avoiding disk} of $C$.
Note that the $B$-avoiding disk of $C$ does not always exist, but if it exists, it is unique.

\paragraph{Nests and linkages.}

Let $\rho$ be a rendition of a society $(G, \Omega)$ in a disk $\Delta$ and let $\Delta^* \subseteq \Delta$ be an arcwise connected set.
A \emph{nest} in $\rho$ around $\Delta^*$ of \emph{order} $s \in \Nbbb_{\geq 1}$ is a set $\mathcal{C} = \{C_1, \dots, C_s \}$ of pairwise disjoint $\rho$-grounded cycles where, for $i \in [s],$ the trace of $C_i$ bounds a disk $\Delta_i$ in such a way that $\Delta^* \subseteq \Delta_1 \subseteq \Delta_2 \subseteq \dots \subseteq \Delta_s \subseteq \Delta,$ $\bd(\Delta_1) \cap \Delta^* = \emptyset,$ and $\bd(\Delta_s) \cap \Delta = \emptyset.$
The last two conditions and the fact that cycles are pairwise disjoint imply that the boundaries of all the disks are also pairwise disjoint.

Let $\rho$ be a rendition of a society $(G, \Omega)$ in a disk $\Delta,$ let $\Delta^* \subseteq \Delta$ be an arcwise connected set and let $\mathcal{C} = \{C_1, \dots, C_s\}$ be a nest in $\rho$ around $\Delta^*.$
We say that a family of pairwise vertex-disjoint paths $\Pcal$ in $G$ is a \emph{radial linkage} if each path in $\mathcal{P}$ has one endpoint in $V(\Omega)$, the other drawn in $\Delta^*$, and is otherwise disjoint from $V(\Omega).$
Moreover, we say that $\mathcal{P}$ is \emph{orthogonal} to $\mathcal{C}$ if for every $P \in \mathcal{P}$ and every $i \in [s],$ $C_i \cap P$ consists of a single component.
Similarly, if $\mathcal{P}$ is a transaction in $(G, \Omega)$ then $\mathcal{P}$ is said to be \emph{orthogonal} to $\mathcal{C}$ if for every $i \in [s]$ and every $P \in \mathcal{P},$ $C_i \cap P$ consists of exactly two components.

\section{$\Sigma$-schemata}

In this final preliminary section, we introduce the concept of a $\Sigma$-schema, which will serve as the primary object of reference throughout our proofs.
The purpose of $\Sigma$-schemas is to provide a unified framework that maintains a $\Sigma$-decomposition of a graph while incorporating additional properties required in our arguments.
It also incorporates a generalized wall-like structure designed to streamline and make our arguments more cohesive and structured.

\subsection{Segments and walloids}

We first revisit the definition of walls in order to define a slightly more general notion that will be one of the building blocks towards the generalized structure that can appropriately model the infrastructure we need for our $\Sigma$-decompositions.

\paragraph{Wall segments.}

Let $r, t \in \Nbbb_{\geq 3}$.
An \emph{elementary $(r \times t)$-wall segment} is obtained from the $(r \times 2t)$-grid by removing a matching $M$ which contains all odd numbered edges of its odd numbered columns and all even numbered edges of its even numbered columns.
An elementary $(r \times t)$-wall segment $W_{1}$ can be seen as the union of $r$ horizontal paths and $t$ vertical paths that are defined as follows: The $r$ \emph{horizontal paths} $P_{1}, \dots, P_{r},$ ordered (and visualized) from top to bottom are the $r$ (subdivided) rows of the original $(r \times 2t)$-grid.
We refer to $P_{1}$ as the \emph{top} and to $P_{t}$ as the bottom horizontal path of $W_{1}$ respectively.
The \emph{top} and \emph{bottom boundary vertices} of $W$ are vertices of the original $(r \times 2t)$-grid that are incident to edges of $M$ and belong to the top and bottom path of $W_{1}$ respectively.
The $t$ \emph{vertical paths} $Q_{1}, \dots, Q_{t},$ ordered (and visualized) from left to right are the pairwise disjoint paths whose union meets all vertices of the original $(r \times 2t)$-grid, and where each $Q_{i}$ has one endpoint that is a top boundary vertex and the other endpoints is a bottom boundary vertex of $W_{1}$.
For $i \in [t]$, the $i$-th \emph{top boundary vertex} of $W_{0}$ is the single vertex in $P_{1} \cap Q_{i}$ that belongs to $M$ and the $i$-th \emph{bottom boundary vertex} of $W_{1}$ is the single vertex in $P_{r} \cap Q_{i}$ that belongs to $M$.
For $i \in [r]$, the $i$-th \emph{left boundary vertex} of $W_{1}$ is the vertex $(i, 1)$ and the $i$-th \emph{right boundary vertex} of $W_{1}$ is the vertex $(i, 2t)$.
An \emph{$(r \times t)$-wall segment} is a subdivision of an elementary $(r \times t)$-wall segment.

\begin{figure}[h]
\centering
\includegraphics{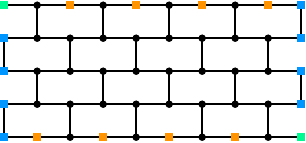}
\caption{\label{fig_wall_segment}An elementary $(5 \times 5)$-wall segment $W_{1}$. The top/bottom boundary vertices of $W_{1}$ are depicted in either green or orange while the left/right boundary vertices of $W_{2}$ are depicted in either green or blue. Note that vertices in green are both top/bottom and left/right boundary vertices.}
\end{figure}

Notice that, by definition a $(r \times t)$-wall $W$ is a $(r \times t)$-wall segment $W_{1}$ with all vertices of degree one removed.
A horizontal or vertical path of $W$ is defined as a horizontal or vertical path of $W_{1}$ after removing any vertex non in $W$.
The cycle of $W$ defined as the union of the top and bottom path of $W$ as well as the leftmost and rightmost vertical path of $W$ is called the \emph{perimeter} of $W$.
Let $W$ be a wall and $W'$ be a subgraph of $W$ that is a wall.
We say that $W'$ is a \emph{subwall} of $W$ if every horizontal path of $W'$ is a subpath of a horizontal path of $W$ and every vertical path of $W'$ is a subpath of a vertical path of $W$.

\begin{figure}[h]
\centering
\includegraphics{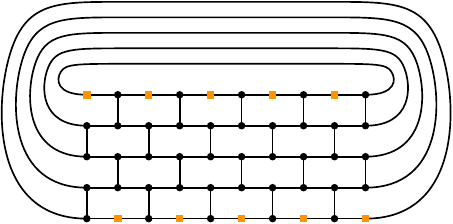}
\caption{\label{fig_annulus_wall}An elementary $(5 \times 5)$-annulus wall.}
\end{figure}

\paragraph{Annulus walls.}

Let $r, t \in \Nbbb_{\geq 3}$.
An \emph{elementary $(r \times t)$-annulus wall} $W$ is obtained from an elementary $(r \times t)$-wall segment $W_{1}$ by adding an edge incident with the $i$-th left and the $i$-th right boundary vertex of $W_{1}$, for each $i \in [r]$.
Notice that the top and bottom boundary vertices of $W$ are the top and bottom boundary vertices of $W_{1}$.
Moreover, the $t$ horizontal paths $P_{1}, \ldots, P_{r}$ of $W_{1}$ are completed into $r$ cycles $C_{1}, \ldots, C_{r}$ of $W$ such that $V(P_{i}) = V(C_{i})$, for each $i \in [r]$.
We call these cycles the \emph{base cycles} of $W$.
We call the cycle $C_{1}$ the \emph{inner cycle} and the cycle $C_{r}$ the \emph{outer cycle} of $W$ respectively and they are visualized as such.
An \emph{$(r \times t)$-annulus wall} is a subdivision of an elementary $(r \times t)$-annulus wall.

Let $W$ be a annulus wall and $W'$ be a subgraph of $W$ that is a wall.
We say that $W'$ is a \emph{subwall} of $W$ if every horizontal path of $W'$ is a subgraph of a base cycle of $W$ and every vertical path of $W'$ is a subpath of a vertical path of $W$.

\paragraph{Handle segments.}

Let $r, t \in \Nbbb_{\geq 3}$.
An \emph{elementary $(r \times t)$-handle segment} is obtained from an elementary $(r, 4t)$-wall segment $W_{1}$ with top boundary vertices $v_{1}, \ldots, v_{t}$, $v'_{1}, \ldots, v_{t}'$, $u_{1}, \ldots, u_{t}$, $u'_{1}, \ldots, u_{t}'$ in left to right order by adding the edges in $\{ v_i u_{t-i+1}, v'_i u'_{t-i+1} \mid i \in [t] \}$.
An \emph{$(r, t)$-handle segment} $W$ is a subdivision of an elementary $(r \times t)$-handle segment.

\begin{figure}[h]
\centering
\includegraphics{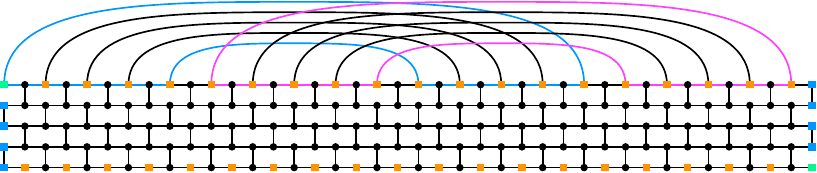}
\caption{\label{fig_handle_segment}An elementary $(5, 5)$-handle segment $W$. The left enclosure of $W$ bounds the left rainbow of $W$ and is depicted in blue while the right enclosure of $W$ bounds the right rainbow of $W$ and is depicted in magenta.}
\end{figure}

Notice that the added edges can be partitioned into a $\{v_{1}, \ldots, v_{t}\}$-$\{ u_{1},\ldots, u_{t}\}$-linkage in $W$, which we call the \emph{left rainbow} of $W$, and a $\{v_{1}', \ldots, v_{t}'\}$-$\{u_{1}', \ldots, u_{t}\}'$-linkage in $W$ respectively, which we call the \emph{right rainbow} of $W$.
Moreover, we define the \emph{left enclosure} of $W$ to be the cycle of $W$ whose degree-$3$ vertices appear in cyclic ordering as $v_{1}, \ldots, v_{t}, u_{1}, \ldots, u_{t}, v_{1}$ and the \emph{right enclosure} of $W$ to be the cycle of $W$ whose degree-$3$ vertices appear in cyclic order as $v_{1}', \ldots, v_{t}', u_{1}', \ldots, u_{t}' v_{1}'$ respectively.

\paragraph{Crosscap segments.}

Let $r, t \in \Nbbb_{\geq 3}$.
An \emph{elementary $(r \times t)$-crosscap segment} is obtained from an elementary $(r \times 4t)$-wall segment $W_{1}$ with top boundary vertices $v_{1}, \ldots, v_{2t}$, $u_{1}, \ldots, u_{2t}$ in left to right order by adding the edges in $\{v_i u_{2t+i} \mid i \in [2t] \}$.
An \emph{$(r \times t)$-crosscap segment} $W$ is a subdivision of an elementary $(r \times t)$-crosscap segment.

\begin{figure}[h]
\centering
\includegraphics{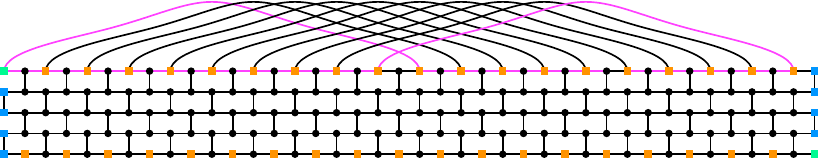}
\caption{\label{fig_crosscap_segment}An elementary $(5 \times 5)$-crosscap segment $W$. The enclosure of $W$ bounds the rainbow of $W$ and is depicted in magenta.}
\end{figure}

Notice that the added edges form a $\{v_{1}, \ldots, v_{2t}\}$-$\{u_{1}, \ldots, u_{2t}\}$-linkage in $W$ which we refer to as the \emph{rainbow} of $W$.
Moreover, we define the \emph{enclosure} of $W$ to be the cycle of $W$ whose degree-$3$ vertices appear in cyclic ordering as $v_{1}, \ldots, v_{2t}, u_{2t}, \ldots, u_{1}, v_{1}$.

\paragraph{Flap segments.}

Let $r, t \in \Nbbb_{\geq 3}$.
An \emph{elementary $(r \times t)$-flap segment} of \emph{arity} $q \in [3]$ is an extended graph obtained from an elementary $(r \times (2t + q))$-wall segment $W_{1}$ with top boundary vertices $v_{1}, \ldots, v_{t}$, $t_{1}, \ldots, t_{q}$, $u_{1}, \ldots, u_{t}$ in left to right order by adding the edges in $\{ v_i u_{t-i+1} \}$, adding a fresh hyperedge with endpoints $z_{1}, \ldots, z_{q}$, and adding the edges in $\{ t_{i} z_{i} \mid i \in [q] \}$.
An \emph{$(r, t)$-flap segment} $W$ is a subdivision of an elementary $(r, t)$-flap segment.

\begin{figure}[h]
\centering
\includegraphics{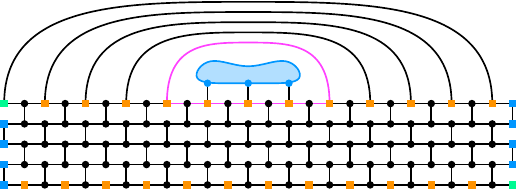}
\caption{\label{fig_flap_segment}An elementary $(5, 5)$-flap segment $W$ of arity $3$. The inner cycle of $W$ is depicted in magenta.}
\end{figure}

Notice that the added edges form a $\{v_{1}, \ldots, v_{t}\}$-$\{u_{1}, \ldots, u_{t}\}$-linkage in $W$ which we call the \emph{rainbow} of $W$.
Moreover, we define the \emph{inner cycle} of $W$ to be the cycle of $W$ that consists of the $v_t$-$u_1$-path that arises after subdividing the edge $v_{t} u_{1}$, along with the $v_t$-$u_1$-path that is a subpath of the top path of $W_{1}$.

\paragraph{Vortex segments.}

Let $r \in \Nbbb_{\geq 3}$, $t \in \Nbbb_{\geq 4}$, and $s \in \Nbbb_{\geq 1}$.
An \emph{elementary $(r \times t, s)$-vortex segment} $W$ is an extended graph obtained from the disjoint union of an elementary $(r \times t)$-wall segment $W_{1}$ with top boundary vertices $v_{1}, \ldots, v_{t}$ in left to right order, a set $\Ccal = \{ C_{0}, \ldots, C_{s + 1} \}$ of $s + 2$ pairwise disjoint cycles, and if $V(C_{0}) = \{ w_{1}, \ldots, w_{t} \}$, a $\{ v_{1}, \ldots, v_{t} \}$-$\{ w_{1}, \ldots, w_{t}\}$-linkage $\Rcal$ that is orthogonal to $\Ccal$, such that there are no vertices of degree $\leq 2$ that are not vertices of $W_{1}$, and then by adding a hyperedge $F = V(C_{0})$.
An \emph{$(r \times t, s)$-vortex segment} $W$ is a subdivision of an elementary $(r \times t, s)$-vortex segment.
Note that in a (non-elementary) $(r \times t, s)$-vortex segment it holds that $\{ w_{1}, \ldots, w_{t}\} \subseteq F \subseteq V(C_{0})$.

\begin{figure}[h]
\centering
\includegraphics{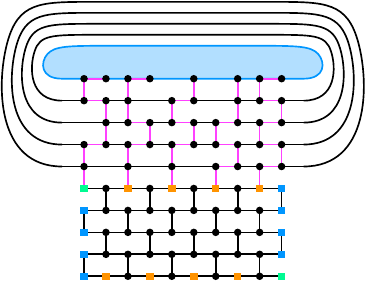}
\caption{\label{fig_vortex_segment}An elementary $(5 \times 5, 3)$-vortex segment $W$. The radial linkage of $W$ is depicted in magenta. The inner cycle of $W$ as well as its hyperedge is depicted in blue.}
\end{figure}

We refer to $\Ccal$ as the \emph{nest} of $W$, to $C_{0}$ as the \emph{inner cycle} of $W$, to $C_{s+1}$ as the \emph{outer cycle} of $W$, and to $\Rcal$ as the \emph{radial linkage} of $W$.

We collectively refer to (elementary) wall, handle, crosscap, and vortex segments as \emph{(elementary) $(r \times t)$-segments} \emph{segments}.
In any segment, we refer to the starting wall segment $W_{1}$ as its \emph{base}.

\paragraph{Concatenation and cylindrical closure of segments.}

Let $r, t \in \Nbbb_{\geq 3}$ and $\ell \in \Nbbb_{\geq 1}$.
Given a sequence $W_{1}, \ldots, W_{\ell}$ of $(r \times t)$-segments, their \emph{concatenation} is obtained from the disjoint union of $W_{1}, \ldots, W_{\ell}$ by adding an edge between the $j$-th right boundary vertex of $W_{i}$ and the $j$-th left boundary vertex of $W_{i+1}$, for $i \in [\ell - 1]$ and $j \in [t]$.
Moreover, their \emph{cylindrical closure} is obtained from their concatenation by adding an edge between $j$-th left boundary vertex of $W_{1}$ and the $j$-th right boundary vertex of $W_{\ell}$, for $j \in [t]$.
Notice that $W$ contains as a subgraph both the concatenation as well as the cylindrical closure of the bases of all these segments.
We refer to this annulus wall as the \emph{base annulus} of $W$ and we denote it as $\widetilde{W}$.
We extend these definitions to non-elementary segments in the natural way.

\paragraph{Surface walls.}

Let $t \in \Nbbb_{\geq 3}$ and $r, \mathsf{h}, \mathsf{c} \in \Nbbb$.
An \emph{elementary $(r, t, \mathsf{h}, \mathsf{c})$-surface wall} is the concatenation of $\mathsf{h} + \mathsf{c} + 1$ many segments such that exactly one is an elementary $((r + t) \times (r + t))$-wall segment, $\mathsf{h}$ are $((r + t) \times t)$-handle segments and $\mathsf{c}$ are $((r + t) \times t)$-crosscap segments.
An \emph{$(r, t, \mathsf{h}, \mathsf{c})$-surface wall} is a subdivision of an elementary $(r, t, \mathsf{h}, \mathsf{c})$-surface wall.

\paragraph{Walloids.}

Let $t \in \Nbbb_{\geq 5}$, $s \in \Nbbb_{\geq 1}$, $r, \mathsf{h}, \mathsf{c}, \ell, \ell' \in \Nbbb$ and $\Sigma$ be a surface of Euler genus $2 \mathsf{h} + \mathsf{c}$.
An \emph{elementary $(r, t, s)$-$\Sigma$-walloid} is the cylindrical closure of $W_{0}, \ldots, W_{\ell + \ell}$, where $W_{0}$ is an elementary $(r, t, \mathsf{h}, \mathsf{c})$-surface wall, $W_{1}, \ldots, W_{\ell}$ are elementary $((r + t) \times t)$-flap segments, and $W_{\ell + 1, \ldots, W_{\ell + \ell'}}$ are elementary $((r + t) \times t, s)$-vortex segments.
An \emph{$(r, t, s)$-$\Sigma$-walloid} is a subdivision of an elementary $(r, t, s)$-$\Sigma$-walloid.

We refer to the these segments as the \emph{segments} of $W$.
We say that a set of hyperedges of $W$ is \emph{consecutive} in $W$ if they belong to segments that appear consecutively in the above order.
For $i \in [\ell]$, we refer to $W_{i}$ as the \emph{$i$-th flap segment} of $W$ and for $i \in [\ell']$, we refer to $W_{\ell + i}$ as the \emph{$i$-th vortex segment} of $W$ respectively.
In the case that a $(r, t)$-$\Sigma$-walloid does not have any flap or vortex segments, we call it a \emph{$(r, t)$-$\Sigma$-annulus wall}.
Moreover, when the parameters $r$, $t$, and $s$ are irrelevant to make a statement about an $(r, t, s)$-$\Sigma$-walloid we may omit specifying any one of them.

\paragraph{Facial cycles of a walloid.}

Let $\mathsf{h}, \mathsf{c}, \ell, \ell' \in \Nbbb$, $\Sigma$ be a surface of Euler genus $2\mathsf{h} + \mathsf{c}$ and $W$ be a $\Sigma$-walloid that is the cylindrical closure of an $(\text{-}, \text{-}, \mathsf{h}, \mathsf{c})$-surface wall, $\ell$ many flap segments and $\ell'$ many vortex segments.
By definition, $W$ is the subdivision of a 3-regular $\Sigma$-embeddable extended graph.
Consider a $\Sigma$-embedding $\Gamma$ of $W$.\footnote{Assuming that $t$ is larger than some modest function depending on the Euler genus of $\Sigma$ (which is, by far, the case in all our results) we can deduce from the classic results of \cite{Thomassen1990Embeddings,Mohar95Uniqueness,SeymourT1996Uniqueness,RobertsonV1990Embeddings} that this embedding is unique.}
Moreover, let $\overline{W}$ to be the graph obtained from $W$ by removing its hyperedges and consequently removing all degree one vertices that arise, and $\Gamma'$ to be the obvious $\Sigma$-embedding of $\overline{W}$ obtained from $\Gamma$ by deleting its disks (which correspond to the hyperedges of $W$) and its nodes and arcs that correspond to the deleted degree one vertices (and their incident edges).

A \emph{brick} of $W$ is any facial cycle of $\overline{W}$ in $\Gamma'$ that contains at most $6$ degree-$3$ vertices of $\overline{W}$.
Among the facial cycles of $\overline{W}$ in $\Gamma'$ that are not bricks, $2 \mathsf{h}$ many of them correspond to the enclosures of the handle segments of $W$, $\mathsf{c}$ of them to the enclosures of the crosscap segments of $W$, $\ell$ many of them correspond to the inner cycles of the flap segments of $W$, and finally $\ell'$ of them correspond to the inner cycles of the vortex segments of $W$.

Two facial cycles of $\overline{W}$ in $\Gamma'$ remain unclassified.
One of the two corresponds to the outer cycle of the base annulus of $W$ which we call the \emph{simple cycle} of $W$, denoted by $C^{\mathsf{si}}(W)$, while the only remaining unclassified facial cycle, we call the \emph{exceptional cycle} of $W$, denoted by $C^{\mathsf{ex}}(W)$, which is the only facial cycle that contains the endpoints of all top paths of the bases of the segments of $W$.

A cycle of $\overline{W}$ is an \emph{enclosure} of $W$ if it is an enclosure of one of its handle or crosscap segments or if it is the inner cycle of the base annulus of $W$, which we call the \emph{big enclosure} of $W$. 
By definition, $W$ has $2\mathsf{h} + \mathsf{c} + 1$ enclosures.

A cycle $F$ of $W$ is a \emph{fence} if it is either a facial cycle of $\overline{W}$ in $\Gamma'$ or if there exists an edge separation $(E_1, E_2)$ of $\overline{W}$, such that
\begin{itemize}
\item $V(F)$ is the set of vertices that are endpoints of edges in both $E_{1}$ and $E_{2}$,  
\item $E_2 \setminus E_1 \neq \emptyset$, $E_1 \setminus E_2 \neq \emptyset$, and  
\item $E(C^{\mathsf{ex}}(W)) \subseteq E_2$.
\end{itemize}
Notice that for every fence $F$ such an edge separation $(E_{1}, E_{2})$ is uniquely defined.
We say that a fence $F'$ of $W$ is \emph{inside} the fence $F$ if $E(F') \subseteq E_{1}$.
A brick of $W$ is a \emph{brick} of a fence $F$ if it is inside $F$.

The $\delta$-\emph{influence} of a fence $F$ of $W$, denoted by $\delta$-$\mathsf{influence}(F)$, is the set of all simple cells of $\delta$ that contain at least one edge of $F$ or whose disk is a subset of the $\mathsf{nodes}_{\delta}(C^{\mathsf{si}}(W))$-avoiding disk of the trace of $F$.
Note that a cell can belong to the influence of at most three of the fences of $W$.
Moreover, disjoint fences have disjoint influences.

\subsection{$\Sigma$-schemata and their cell-colorings}\label{sec_schemata}

In our proofs we require our $\Sigma$-decompositions to satisfy some additional connectivity properties which we introduce here.

\paragraph{Tight $\Sigma$-decompositions.}

Let $G$ be a graph, $\Sigma$ be a surface, and $\delta = (\Gamma, \mathcal{D})$ be a $\Sigma$-decomposition of $G$.
We define the \emph{$\delta$-torsoid} of $G$ as the graph $T$ with vertex set being the union of the ground vertices of $\delta$ along with a fresh vertex $u_{c}$ for every vortex cell $c \in C(\delta).$
As for the edge set of $T,$ for every cell $c \in C(\delta),$ if $(x_{1}, \ldots, x_{\ell}, x_{\ell + 1} = x_{1})$ is the cyclic ordering of the vertices in $\pi_{\delta}(\tilde{c})$ induced by traversing along the boundary of $\Delta_{c}$ in counter-clockwise direction starting from an arbitrarily chosen vertex $x_{1} \in \pi_{\delta}(\tilde{c}),$ then if $\ell = 2,$ $x_{1}x_{2} \in E(T),$ and if $\ell \geq 3,$ for every $i \in [\ell],$ $x_{i}x_{i+1} \in E(T).$
Moreover, if $c$ is a vortex cell of $\delta,$ then $x_{i}u_{c} \in E(T),$ that is every vortex cell induces a wheel subgraph of $T$ with center $u_{c}.$

With this definition at hand, we say that $\delta$ is \emph{tight} if the following two conditions are met:
\begin{enumerate}
\item For every simple cell $c \in C(\delta),$ there is an $x$-$y$-path $P$ in $\sigma_{\delta}(c)$ between any two distinct vertices $x, y \in \pi_{\delta}(\tilde{c}).$
In case $|\tilde{c}| = 3,$ then $P$ is required to avoid $z \in \pi_{\delta}(\tilde{c})$ where $x \neq z \neq y.$
\item The $\delta$-torsoid of $G$ is 3-connected.
\end{enumerate}

Since the definition of tightness is a relatively standard notion when working with $\Sigma$-decompositions and renditions (see e.g.~\cite{BasteST19HittingMinors}), we only briefly sketch how every $\Sigma$-decomposition can be transformed into a tight $\Sigma$-decomposition in linear time.

\begin{observation}\label{lem:tight_renditions} For every non-negative integer $g$ there exists a non-negative integer $c \geq 3$ such that the following holds.

For every graph $G$ that admits a $\Sigma$-decomposition $\delta$ in a surface $\Sigma$ of Euler-genus $g$ with a $\delta$-grounded $(r, c)$-$\Sigma$-annulus wall $W \subseteq G,$ for some $r \in \Nbbb,$ there exists a tight $\Sigma$-decomposition $\delta'$ of $G$ such that $W$ is $\delta'$-grounded.

Moreover, there is an algorithm that computes $\delta'$ in time $\Ocal(|V(G)| + |E(G)|).$
\end{observation}
\begin{proof}[Proof sketch] First note that we may assume that every ground vertex of $\delta$ belongs to the same connected component of $G$ and that every other connected component (if any) is contained in a special cell of $\delta$ with boundary size $0.$

Meeting the first condition is relatively straightforward and can be accomplished by locally repairing cells that do not satisfy it as follows: Let $c \in C(\delta)$ be such a cell.
Now look at the connected components of $\sigma_{\delta}(c)$ which must be at least $2$, and by our former assumption rooted at distinct vertices of $\pi_{\delta}(\tilde{c}).$
For each such connected component $C,$ we define a disk $\Delta_{C} \subseteq \Delta_{c}$ that contains $C$ and whose boundary intersects only the vertices in $\pi_{\delta}(\tilde{c}) \cap V(C).$
Note that we may choose these disk so that they are pairwise disjoint.
Now, we define $\delta'$ by replacing $\Delta_{c}$ with the union of all disks $\Delta_{C}$ as above.
It is now straightforward to verify that $\delta$ satisfies the first condition.

Meeting the second condition is slightly more involved.
First note that our former assumption implies that the $\delta'$-torsoid $T$ of $G$ is connected.
Moreover, $T$ is a $\Sigma$-embeddable graph and we consider the obvious embedding $\Gamma'$ induced by $\delta'.$
Additionally, since $W$ is grounded in $\delta,$ by property i), and the definition of $T,$ there is an $(r, c)$-$\Sigma$-annulus wall\footnote{$\hat{W}$ is in fact is mesh-like due to the replacement of cells with triangles in the definition of $\delta$-torsoid. This means that horizontal and vertical paths of $\hat{W}$ orthogonally cross but their intersection may be a single degree-$4$ vertex.} subgraph $\hat{W}$ of $T$ with the same ground vertices as $W.$

It is relatively easy to show that $\hat{W}$ has a $\Sigma$-embedding with \emph{representativity} at least $c.$
Moreover, if we choose $c$ to be some modest function of the Euler-genus $g$ of $\Sigma,$ it follows by the classic results of \cite{Thomassen1990Embeddings,Mohar95Uniqueness,SeymourT1996Uniqueness,RobertsonV1990Embeddings} that this embedding is unique.
Therefore, we may assume that the representativity of $\Gamma'$ is at least $c,$ i.e., there is no non-contractible curve on $\Sigma$ that intersects $T$ only in vertices and in at most $c \geq 3$ of them.
Therefore, we may assume that any minimal separator $S$ of $T$ with $|S| \leq 2,$ induces a contractible curve on $\Sigma$ that intersects $\Gamma'$ only in the points associated to the drawing of the vertices in $S.$

Now assume that $T$ is not $2$-connected.
We may compute the set $\mathcal{B}$ of blocks ($2$-connected components) of $T$ using the classic linear-time algorithm of Tarjan~\cite{Tarjan1971Depth}.
Since $\hat{W}$ is a subdivision of a $3$-connected graph it follows that there exists a unique block $B \in \mathcal{B}$ that contains $\hat{W}.$
Let $u$ be a cut-vertex of $T$ that belongs to $B$ and $(X, Y)$ be a separation of $T$ such that $X \cap Y = \{ u \},$ $X \setminus \{ u \}$ contains the vertex set of the component of $T - u$ that contains $V(B) \setminus \{ u \}$ while $Y \setminus \{ u \}$ contains the rest.
Let $J$ be a contractible closed curve that intersects $\Gamma'$ only at $u,$ bounds a a $\delta'$-aligned disk $\Delta_{J}$ internally disjoint from $X$ and such that every vertex of $Y$ is drawn in $\Delta_{J}.$
We may define $\delta''$ by replacing all cells contained in $\Delta_{J}$ by $\Delta_{J}$ for all such cut-vertices $u$ and contractible curves $J$ defined as above.

We may now assume that the $\delta''$-torsoid $T'$ is $2$-connected but not $3$-connected.
We essentially can repeat the same process by computing the set of $3$-connected components $\mathcal{C}$ of $T'$ using the classic linear-time algorithm of Hopcroft and Tarjan~\cite{HopcroftT1973Dividing}.
Similarly to before, we identify the unique $3$-connected component $C \in \mathcal{C}$ that contains all degree-$(\geq 3)$ vertices of $\hat{W}$ (note that paths of $\hat{W}$ may invade other parts of $T$ via the $2$-separators of $T$ incident to $C$).
Any minimal separator $S$ of size $2$ attaching other parts of $T'$ to $C$ corresponds again, by representativity, to a contractible closed curve bounding a disk internally disjoint from $C$ and containing all other components attaching to $C$ via $S.$
We absorb the cells contained in such disks into single cells as before.
Iterating this procedure yields a $\Sigma$-decomposition $\delta'',$ whose torsoid is $3$-connected.

It is important to note that any vortex cell of $\delta'''$ is one of the original vortices of $\delta$.
In the torsoid, vortices are replaced by wheels, which are $3$-connected.
Hence they cannot be separated by the $1$- or $2$-separators considered above and are therefore unaffected.
\end{proof}

We may proceed with the formal definition of $\Sigma$-schemas.

\medskip
Let $G$ be a graph, $\Sigma$ be a surface, and $\delta = (\Gamma, \mathcal{D})$ be a $\Sigma$-decomposition of $G$.
Also, let $r, x, y \in \Nbbb$, $t \in \mathbb{N}_{\geq 3}$, $s \in \Nbbb_{\geq 1}$, and $W$ be an $(r, t, s)$-$\Sigma$-walloid whose vertices and edges are vertices and edges of $G$.
We say that the triple $(G, \delta, W)$ is an \emph{$(r, t, s, x, y)$-$\Sigma$-schema} if
\begin{itemize}
    \item[$\mathbf{S}_{1}$:] $\delta$ is tight;
    \item[$\mathbf{S}_{2}$:] $\overline{W}$ is grounded in $\delta$;
    \item[$\mathbf{S}_{3}$:] $\delta$ has breadth $b \leq x$ and depth at most $y$; and moreover
\end{itemize}
there is a $\Sigma$-embedding $\Gamma'$ of $W$ such that
\begin{itemize}
\item[$\mathbf{S_{4}}$:] every arc of $\Gamma'$ is an arc of $\Gamma$,
\item[$\mathbf{S_{5}}$:] the closure of every disk of $\Gamma'$ (which corresponds to a hyperedge of a flap or vortex segment of $W$) is identical to the disk of a simple cell of $\delta$ in the case of a flap segment and to a vortex cell of $\delta$ otherwise, and
\item[$\mathbf{S_{6}}$:] if $W$ has no vortex segments, every vortex of $\delta$ is contained in the interior of the $\mathsf{nodes}_{\delta}(C^{\mathsf{si}}(W))$-avoiding disk of $\mathsf{trace}_{\delta}(C^{\mathsf{ex}}(W))$. Otherwise, $W$ has $b$ vortex segments.
\end{itemize}

Whenever the parameters $r$, $t$, $s$, $x$, or $y$ are irrelevant to make a statement about an $(r, t, s, x, y)$-$\Sigma$-schema we may omit specifying any one of them to simplify notation.
E.g. we may write $(\text{-}, t, \text{-}, x, y)$-$\Sigma$-schema if the statement does not concern values $r$ and $s$.

\medskip
In a similar manner as to walls we can define the tangle induced by a walloid.
Let $r \in \Nbbb$, $t \in \Nbbb_{\geq 3}$, and $(G, \delta, W)$ be an $(r, t, \text{-}, \text{-}, \text{-})$-$\Sigma$-schema.
Let $\Tcal_{W}$ be the orientation of $\Scal_{t + r}$ such that for every $(A, B) \in \Tcal_{W}$, the set $B \setminus A$ contains the vertex set of both a cycle and a vertical path of the base annulus of $W$.
Then $\Tcal_{W}$ is the tangle \emph{induced} by $W$.
If $\Tcal$ is a tangle in $G$, we say that $(G, \delta, W)$ is \emph{controlled} by $\Tcal$ if $W$ is controlled by $\Tcal$.

\medskip
For the proof of \cref{main_lemma} we do not yet require boundaried indices to abstractly capture finite index properties of boundaried subgraphs.
Instead, we equip the simple cells of our $\Sigma$-decomposition with \textsl{colors} from a finite set.

\paragraph{Cell-colorings.}

Let $\delta$ be a $\Sigma$-decomposition of a graph $G$ in a surface $\Sigma$.
A \emph{cell-coloring} of $\delta$ is a function $\chi$ mapping simple cells of $\delta$ to positive integers.
We define $\chi$-$\mathsf{col}(\delta) \coloneqq \{ \chi(c) \mid c \in C_{\mathsf{s}}(\delta) \}$.
The maximum number in $\chi$-$\mathsf{col}(\delta)$ is the \emph{capacity} of $\chi$.

Let $(G, \delta', W')$ be a $\Sigma$-schema such that $\delta' \sqsubseteq \delta$.
Let $S \subseteq \chi\text{-}\mathsf{col}(\delta')$ and $b \in \mathbb{N}_{\geq_{1}}$.
We say that (the corresponding $\Sigma$-embedding $\Gamma'$ of) $W'$ \emph{$b$-represents $S$ in $\delta'$} if for every $\alpha \in S$ there are $b$ simple cells $c_{1}, \ldots, c_{b}$ of $\delta'$ whose disks are identical to the closures of $b$ many disks of $\Gamma'$, which correspond to consecutive hyperedges of flap segments of $W'$, and such that for every $i \in [b]$, $\chi(c_i) = \alpha$.
We also say that $W'$ \emph{$b$-represents $\chi$ in $\delta'$} if for every $\alpha \in \chi\text{-}\mathsf{col}(\delta')$, $W'$ $b$-represents $\alpha$ in $\delta'$.

Given a fence $F$ of $W'$, we define the \emph{$\chi$-coloring of $F$ in $\delta'$} as the set $\{ \chi(c) \mid c \in \delta'\text{-}\mathsf{influence}(F) \}$.
We say that $F$ is \emph{$\chi$-homogeneous in $\delta'$} if all bricks of $W'$ in $F$ have the same $\chi$-coloring in $\delta'$ (which is equal to the $\chi$-coloring of $F$ in $\delta'$).
We also say that $W'$ is \emph{$\chi$-homogeneous in $\delta'$} if every enclosure of $W'$ is $\chi$-homogeneous in $\delta'$.

\subsection{The local structure theorem for graph minors}

As we have discussed \cref{thm_local_structure} builds on the LST of Robertson and Seymour \cite{RobertsonS03GraphMinorsXVI}.
We present here an abbreviated version of the recent result of Gorsky, Seweryn, and Wiederrecht \cite{GorskySW2025Polynomial} which will serve as our starting point.

\begin{proposition}[Local Structure Theorem \cite{GorskySW2025Polynomial}, Theorem 15.1] \label{prop_lst}
There exist functions $\mathsf{wall}_{\ref{prop_lst}} \colon \Nbbb^{3} \to \Nbbb$,
$\mathsf{apex}_{\ref{prop_lst}}$, $\mathsf{depth}_{\ref{prop_lst}} \colon \Nbbb^{2} \to \Nbbb$, such that for every $k, t \in \Nbbb_{\geq 5}$, $r \in \Nbbb$, every graph $H$ on $k$ vertices, every graph $G$, and every $\mathsf{wall}_{\ref{prop_lst}}(k, t, r)$-wall $W \subseteq G$, one of the following holds. 
\begin{enumerate}
\item there is a minor model of $H$ in $G$ controlled by $\Tcal_{W}$, or
\item there is a set $A \subseteq V(G)$ and an $(r, t, \text{-}, \nicefrac{1}{2}(k - 3)(k - 4), \mathsf{depth}_{\ref{prop_lst}}(k, t))$-$\Sigma$-schema $(G - A, \delta, W')$ controlled by $\Tcal_{W}$ such that $|A| \leq \mathsf{apex}_{\ref{prop_lst}}(k, t)$, $\Sigma$ is a surface where $H$ does not embed, and $W'$ is an $(r, t)$-$\Sigma$-annulus wall.
\end{enumerate}

Moreover, it holds that
$$\mathsf{apex}_{\ref{prop_lst}}(k, t)\text{, }\mathsf{depth}_{\ref{prop_lst}}(k, t) \in \Ocal((k + t)^{112})\text{, and }\mathsf{wall}_{\ref{prop_lst}}(k, t, r) \in \Ocal((k + t)^{115} \cdot r).$$
There also exists an algorithm that, given $k, r, t$, a graph $H$, a graph $G$, and a wall $W$ as above as input, finds one of these outcomes in time $\poly(k + t + r) \cdot |E(G)||V(G)|^{2}$.
\end{proposition}

We stress that the tightness of $\delta$ (property $\mathbf{S}_{1}$ in the definition of schemas) is not guaranteed in the the outcome of Theorem 15.1 in \cite{GorskySW2025Polynomial}.
However, this can easily be mended by applying \cref{lem:tight_renditions}.

Moreover, Theorem 15.1 in \cite{GorskySW2025Polynomial} only gives a $(0, t)$-$\Sigma$-annulus wall $W'$ as an outcome and a separate wall $W''$ of order $r$ that is grounded in $\delta$.
In their proof however, both, the base annulus of $W''$ and the wall $W'$ are found within a much larger wall $W'''$ that is grounded in $\delta$ (see the inductive construction in the proof of Theorem 15.4 in \cite{GorskySW2025Polynomial}).
This allows to easily extend the $(0,t)$-$\Sigma$-annulus wall to an $(r,t)$-$\Sigma$-annulus wall by using the infrastructure of $W'''$ as well as the wall $W'$.

\section{Obtaining a homogeneous walloid}\label{sec_homogeneous_walloid}

In this section we prove a lemma which marks the first step towards the final $\Sigma$-schema.
We prove that, assuming that the walloid of the $\Sigma$-schema is large enough, we can find a still large walloid, which can be divided into territories (enclosures) that all have the property of being homogeneous with respect to some pre-specified cell-coloring of the $\Sigma$-decomposition of the $\Sigma$-schema.

\subsection{Finding a homogeneous subwall}\label{subsec_homogeneity}

We first show how given a large enough wall whose bricks come equipped with a subset of a fixed set of colors, we can always find a still large subwall of it where all bricks are accompanied by the same color set.
We moreover observe how the same arguments give us a mono-dimensional variant applied to ladders.

\medskip
Let $W$ be a wall.
Notice that every cycle $C$ of $W$ naturally defines a separation $(A, B)$ of $W$ such that $A \cap B = V(C)$ and the vertex set of the perimeter of $W$ belongs to $A$.
Given another cycle $C'$ of $W$ we say that $C'$ is \emph{inside} $C$ if $V(C') \subseteq B$ and \emph{outside} otherwise.
Moreover, notice that given a cycle $C$ of $W$, every brick $B$ of $W$ is either inside of outside $C$.

We define a \emph{brick-coloring} of $W$ as a function $\chi$ mapping each brick of $W$ to some finite non-empty subset of $\mathbb{N}_{\geq 1}$.
The maximum number contained in any of these subsets is the \emph{capacity} of $\chi.$
Consider a subwall $W'$ of $W$.
Given a brick-coloring $\chi$ of $W$, we say that $\chi'$ is the brick-coloring of $W'$ \emph{induced by $\chi$} if for every brick $B$ of $W',$ $\chi'(B) = \bigcup_{B' \in \Bcal} \chi(B')$, where $\Bcal$ is the set of bricks of $W$ inside $B$.
Notice that the capacity of $\chi'$ is at most the capacity of $\chi$.

\begin{lemma}\label{lem_homogeneity_2d} There exists a function $f_{\ref{lem_homogeneity_2d}} \colon \Nbbb^{2} \to \Nbbb$ and an algorithm that, given $\ell \in \Nbbb_{\geq 1}$,  $t, z \in \mathbb{N}_{\geq 3}$, a $(f_{\ref{lem_homogeneity_2d}}(t, \ell) \times f_{\ref{lem_homogeneity_2d}}(z, \ell))$-wall $W$, and a brick-coloring $\chi$ of $W$ of capacity $\ell$, outputs a $(t \times z)$-subwall $W'$ of $W$ such that for every two bricks $B_{1}$ and $B_{2}$ of $W'$, $\chi'(B_{1}) = \chi'(B_{2})$, where $\chi'$ is the brick-coloring of $W'$ induced by $\chi$, in time $(t + z)^{\Ocal(2^{\ell})}.$
Moreover $f_{\ref{lem_homogeneity_2d}}(x, \ell) = x^{2^{\ell}}.$
\end{lemma}
\begin{proof} Let $\alpha \colon \Nbbb \to \Nbbb$ be such that $\alpha(\ell) \coloneqq 2^{\ell}.$
We define $f_{\ref{lem_homogeneity_2d}}(x, \ell) \coloneqq x^{\alpha(\ell)}.$

Let $W$ be a $(t^{\alpha(\ell)} \times z^{\alpha(\ell)})$-wall.
Let $\Pcal_{1}, \ldots, \Pcal_{t^{\alpha(\ell-1)}}$ be a grouping of the horizontal paths of $W$ in consecutive bundles of $t^{\alpha(\ell-1)}$ paths ordered from left to right, where the last path in $\Pcal_{i}$ is the first path of $\Pcal_{i+1}.$
Also let $\Qcal_{1}, \ldots, \Qcal_{z^{\alpha(\ell-1)}}$ be a grouping of the vertical paths of $W$ in consecutive bundles of $z^{\alpha(\ell-1)}$ paths ordered from top to bottom, where the last path in $\Qcal_{i}$ is the first path of $\Qcal_{i+1}.$

For every $i \in [t^{\alpha(\ell-1)}]$ and $j \in [z^{\alpha(\ell-1)}],$ let $W_{i, j}$ be the graph $\cupall \Pcal'_{i} \cup \Qcal'_{j}$ which is a $(t^{\alpha(\ell-1)} \times z^{\alpha(\ell-1)})$-subwall of $W$ (after removing all vertices of degree one), where $\Pcal'_{i}$ and $\Qcal'_{j}$ are defined as follows.
Let $P_{i}$ (resp. $P'_{i}$) be the topmost (resp. bottommost) path of $W$ in $\Pcal_{i}$ and $Q_{j}$ (resp. $Q'_{j}$) be the leftmost (resp. rightmost) path of $W$ in $\Qcal_{j}$.
Let $\Pcal'_{i}$ (resp. $\Qcal'_{j}$) contain the paths in $\Pcal_{i}$ (resp. $\Qcal_{i}$) cropped to contain only vertices of vertical (resp. horizontal) paths between $Q_{i}$ (resp. $P_{i}$) and $Q'_{i}$ (resp. $P'_{i}$).
There are $t^{\alpha(\ell-1)} \cdot z^{\alpha(\ell-1)}$ such subwalls depending on the choices of $i$ and $j$.

We also consider $P$ (resp. $P'$) to be the top (resp. bottom) horizontal path of $W$, and $Q$ (resp. $Q'$) to be the leftmost (resp. rightmost) vertical path of $W$.
Observe that the graph
$$P \cup P' \cup Q \cup Q' \cup \bigcup_{i \in [t^{\alpha(\zeta-1)} - 1]} (\Pcal_{i} \cap \Pcal_{i+1}) \cup \bigcup_{j \in [z^{\alpha(\zeta-1)} - 1]} (\Qcal_{j} \cap \Qcal_{j+1})$$ also forms a $(t^{\alpha(\ell-1)} \times z^{\alpha(\ell-1)})$-subwall $W'$ of $W$.

Now, we assume inductively that for every $(t^{\alpha(\ell-1)} \times z^{\alpha(\ell-1)})$-subwall $W''$ of $W,$ there exists a $(z \times t)$-subwall of $W$ such that for every two bricks $B_{1}$ and $B_{2}$ of $W''$, $\chi''(B_{1}) \cap [\ell - 1] = \chi''(B_{2}) \cap [\ell - 1]$, where $\chi''$ is the brick-coloring of $W''$ induced by $\chi$.
Next, we show that there exists $W'' \in \{ W_{i, j} \mid i \in [t^{\alpha(\ell-1)}], j \in [z^{\alpha(\ell-1)}]\} \cup \{ W' \}$ such that, either for every brick $B''$ of $W''$, $\ell \in \chi''(B''),$ or $\chi''$ has capacity strictly less than $\ell$, where $\chi''$ is the brick-coloring of $W''$ induced by $\chi$.
Notice that from the induction hypothesis this suffices to prove our claim.
To conclude, observe that by definition, either for every brick $B'$ of $W'$, $\ell \in \chi'(B'),$ where $\chi'$ is the brick-coloring of $W'$ induced by $\chi$, or there exists $W'' \in \{ W_{i, j} \mid i \in [t^{\alpha(\ell-1)}], j \in [z^{\alpha(\ell-1)}]\}$ such that $\chi''$ has capacity strictly less than $\ell$, where $\chi''$ is the brick-coloring of $W''$ induced by $\chi.$
\end{proof}

\paragraph{Ladders.}

Let $L$ be a subdivision of the $(t \times 2)$-grid, where $t \in \Nbbb_{\geq 2}.$
We call $L$ a $t$-\emph{ladder}.
Notice that $L$ can be seen as the union of $t$ horizontal paths $P_{1}, \ldots, P_{t}$ and two vertical paths $Q_{1}, Q_{2}$ which correspond to the subdivided rows and columns respectively of the original $(t \times 2)$-grid.
We say that a ladder $L'$ is a \emph{subladder} of $L$ if $L'$ is obtained from $L$ by removing the edges of some (maybe none) of its horizontal paths.
Consider a plane-embedding of $L$ obtained from the natural plane-embedding of the $(t \times 2)$-grid.
We call all facial cycles of $L$, except the cycle that bounds the outer face, the \emph{bricks} of $L$.
It is straightforward to define brick-colorings for ladders as well. In fact, one may see a the bricks of a $t$-ladder as the bricks of a $(t\times 2)$-wall.

\medskip
We get the following ``mono-dimensional'' corollary of \cref{lem_homogeneity_2d}, by applying the same inductive argument only in the dimension of the horizontal paths of a ladder.

\begin{corollary}\label{cor_homogeneity_1d} There is an algorithm that, given $\ell \in \Nbbb_{\geq 1},$ $t \in \mathbb{N}_{\geq 2},$ an $f_{\ref{lem_homogeneity_2d}}(t, \ell)$-ladder $L,$ a brick-coloring $\chi$ of $L$ of capacity $\ell,$ outputs a $t$-subladder $L'$ of $L$ such that for every two bricks $B_{1} = B_{2}$ of $L',$ $\chi'(B_{1}) = \chi'(B_{2}),$ where $\chi'$ is the brick-coloring of $L'$ induced by $\chi$, in time $\Ocal(t^{2^{\ell}}).$
\end{corollary}

\subsection{The homogenization lemma}

We are now in the position to present the main result of this section.

\begin{lemma}[Homogenization lemma]\label{lem:homogenization}
There exists functions $f^1_{\ref{lem:homogenization}} \colon \Nbbb^3 \to \Nbbb$ and $f^2_{\ref{lem:homogenization}} \colon \Nbbb^4 \to \Nbbb$ such that for every $r \in \Nbbb$, $t \in \Nbbb_{\geq 3}$, every $\ell \in \Nbbb_{\geq 1}$, every $\mathsf{h}, \mathsf{c} \in \mathbb{N}$, every $(r' \coloneqq f^2_{\ref{lem:homogenization}}(r, t, \ell, \mathsf{h} + \mathsf{c}), t' \coloneqq f^1_{\ref{lem:homogenization}}(t, \ell, \mathsf{h} + \mathsf{c}), \text{-}, \text{-}, \text{-})$-$\Sigma$-schema $(G, \delta, W)$ where $\Sigma$ has Euler genus $2\mathsf{h} + \mathsf{c}$ and $W$ is a $(r', t')$-$\Sigma$-annulus wall, and every cell-coloring $\chi$ of $\delta$ of capacity at most $\ell$, there is an $(r, t, \text{-}, \text{-}, \text{-}, \text{-})$-$\Sigma$-schema $(G, \delta, W')$ controlled by $\Tcal_{W}$ such that $W'$ is an $(r, t, \text{-})$-$\Sigma$-annulus wall that is $\chi$-homogeneous in $\delta$.

Moreover, it holds that
$$f^1_{\ref{lem:homogenization}}(t, \ell, \mathsf{h} + \mathsf{c}) \in (t(\mathsf{h} + \mathsf{c} + 1))^{2^{\Ocal(\ell)}} \ \text{ and } \ f^2_{\ref{lem:homogenization}}(r, t, \ell, \mathsf{h} + \mathsf{c}) \in ((r + t)(\mathsf{h} + \mathsf{c} + 1))^{2^{\Ocal(\ell)}}.$$
There also exists an algorithm that finds the outcome above in time $\poly(r + t + \mathsf{h} + \mathsf{c}) \cdot |E(G)||V(G)|$.
\end{lemma}
\begin{proof} For the remainder of this proof, let $g = \mathsf{h} + \mathsf{c}$.
We define the two functions of the lemma as follows:
\begin{align*}
f^1_{\ref{lem:homogenization}}(t, \ell, \mathsf{h} + \mathsf{c}) \ &\coloneqq \ f_{\ref{lem_homogeneity_2d}}(2t, \ell)\text{ and}\\
f^2_{\ref{lem:homogenization}}(r, t, \ell, \mathsf{h} + \mathsf{c}) \ &\coloneqq \ 4 \cdot f_{\ref{lem_homogeneity_2d}}(2(r + t), \ell) \cdot g + f_{\ref{lem_homogeneity_2d}}(3(r + t) + 4 \cdot f_{\ref{lem_homogeneity_2d}}(2t, \ell) \cdot g, \ell).
\end{align*}

Let $z' = r' + t'$ and $z = r + t$.
By definition, $W$ is the cylindrical closure of an $(r', t', \mathsf{h}, \mathsf{c})$-surface wall which itself is the concatenation of a sequence $W_{0}, W_{1}, \ldots, W_{g}$, where $W_{0}$ is a $(z' \times z')$-wall segment and for each $i \in [g]$, $W_{i}$ is either a $(z' \times t')$-handle or $(z' \times t')$-crosscap segment.
We show how to utilize this infrastructure to define an $(r, t, \text{-})$-$\Sigma$-annulus wall that is a subgraph of $W$ with the property that all its $2\mathsf{h} + \mathsf{c} + 1$ enclosures are $\chi$-homogeneous in $\delta$.
The proof proceeds in the following four steps.
See \cref{fig_homogeneous_walloid} for an illustration.

\begin{figure}[h]
\centering
\includegraphics{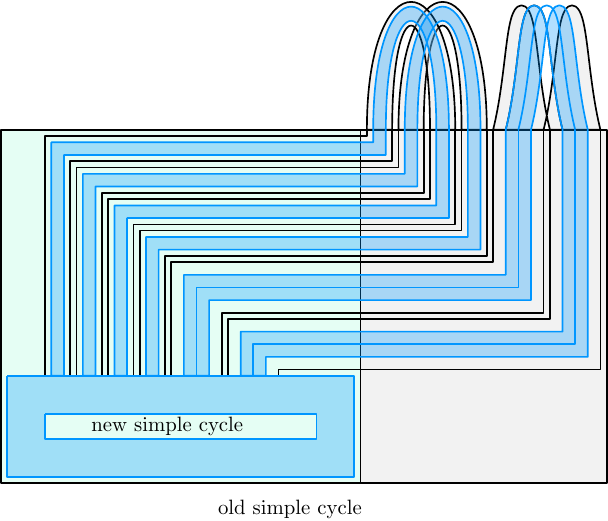}
\caption{\label{fig_homogeneous_walloid}Finding an $(r, t)$-$\Sigma$-annulus wall with homogeneous enclosures (depicted in blue) within an $(r', t')$-$\Sigma$-annulus wall in the proof of \cref{lem:homogenization}.}
\end{figure}

\smallskip
\textbf{Step 1: Finding a homogeneous subwall.} 
Let $W'_{0}$ be the $(z' \times z')$-wall obtained from $W_{0}$ by deleting all vertices of degree $1$.
By definition of $z'$, there is an $(f_{\ref{lem_homogeneity_2d}}(2z, \ell) \times f_{\ref{lem_homogeneity_2d}}(3z + 4 \cdot f_{\ref{lem_homogeneity_2d}}(2t, \ell) \cdot g, \ell))$-subwall $W''_{0}$ of $W'_{0}$ that is disjoint from the first $4 \cdot f_{\ref{lem_homogeneity_2d}}(2t, \ell) \cdot g$ horizontal paths of $W'_{0}$.

Now consider the brick-coloring $\chi_{0}$ of $W'_{0}$, where for every brick $B$ of $W'_{0},$ $\chi_{0}(B)$ is defined as the $\chi$-coloring of $B$ in $\delta$.
Next, by applying \cref{lem_homogeneity_2d} on $\chi_{0}$ and $W'_{0}$, we obtain a $(2z \times (3z + 4 \cdot f_{\ref{lem_homogeneity_2d}}(2t, \ell) \cdot g))$-subwall $W'_{1}$ of $W'_{0}$ such that for every two bricks $B_{1}$ and $B_{2}$ of $W'_{1}$, $\chi_{1}(B_{1}) = \chi_{1}(B_{2})$, where $\chi_{1}$ is the brick-coloring of $W'_{1}$ induced by $\chi_{0}$.
As a result, all bricks of $W'_{1}$ have the same $\chi$-coloring in $\delta$, since for every brick $B$ of $W'_{1}$, $\chi_{1}(B)$ is identical to the $\chi$-coloring of $B$ in $\delta$.

\smallskip
\textbf{Step 2: Defining the homogeneous base annulus.} Let $a \coloneqq 2z$ and $b \coloneqq z + 4 \cdot f_{\ref{lem_homogeneity_2d}}(2t, \ell) \cdot g + a$.
Let $P^{2}_{1}, \ldots, P^{2}_{a}$ be the horizontal paths of $W'_{1}$ and $Q^{2}_{1}, \ldots, Q^{2}_{b}$ be the vertical paths of $W'_{1}$.
We define $z$ pairwise disjoint cycles $\Ccal \coloneqq \{ C_{1}, \ldots, C_{z} \}$ as follows.
For every $i \in [z]$, define $C_{i} \coloneqq P^{2}_{i} \cup P^{2}_{a - i + 1} \cup Q^{2}_{i} \cup Q^{2}_{b - i + 1}$ (after removing all vertices of degree one).
Next, we define $b' \coloneqq b - a$ pairwise disjoint paths $\Pcal \coloneqq \{ P_{1}, \ldots, P_{b'} \}$ that are subpaths of vertical paths in $W'_{1}$ as follows.
For every $i \in [b']$, define $P_{i}$ to be the maximal subpath of $Q^{2}_{z + i}$ with one endpoint in $P^{2}_{1}$ and the other in $P^{2}_{z}$.
Observe that the graph $W'_{2} \coloneqq \cupall (\Pcal \cup \Qcal)$ is a $(z \times (z + 4 \cdot f_{\ref{lem_homogeneity_2d}}(2t, \ell) \cdot g))$-annulus wall that is a subgraph of $W'_{2}$ with the following additional property.
Every brick of $W'_{2}$ contains a set of bricks of $W'_{1}$, and therefore, all bricks of $W'_{2}$ have the same $\chi$-coloring in $\delta$ which is equal to the $\chi$-coloring of the perimeter of $W'_{2}$.

\smallskip
\textbf{Step 3: Attaching the handle and crosscap linkages.} Let $d \coloneqq f_{\ref{lem_homogeneity_2d}}(2t, \ell)$.
Notice that $W'_{2}$ can be seen as the cylindrical closure of a sequence $W''_{1}, \ldots, W''_{g}$ of $(z \times 4d)$-wall segments and a $(z \times z)$-wall segment which we assume to appear last in this ordering.
We can moreover define this ordering so that the top boundary vertices of $W''_{1}$ appear first in the left to right ordering of the top boundary vertices of $W'_{1}$.

Now, fix $i \in [g]$.
Assume that $W_{i}$ is a handle segment.
Let $\Lcal_{i} \coloneqq \{ L_{i, 1}, \ldots, L_{i, d} \}$ be the first $d$ paths in left to right order of the left rainbow of $W_{i}$.
Let $u_{i, j}$ and $v_{i,j}$ be the top boundary vertices of the base of $W_{i}$ in left to right order that are endpoints of $L_{i, j},$ $j \in [d]$.
Also, let $\Rcal_{i} \coloneqq \{ R_{i, 1}, \ldots R_{i, d} \}$ be the first $d$ paths in left to right order of the left rainbow of $W_{i}$.
Let $w_{i, j}$ and $z_{i,j}$ be the top boundary vertices of the base of $W_{i}$ in left to right order that are endpoints of $R_{i, j},$ $j \in [d]$.
Additionally, let $u'_{i,1}, \ldots, u'_{i,d}, w'_{i, 1}, \ldots, w'_{i, d}, v'_{i, 1}, \ldots v'_{i, d}, z'_{i, 1}, \ldots, z'_{i, d}$ be the top boundary vertices of $W''_{i}$ in left to right order.
Our target is to identify a set
$$\Ycal_{i} = \{ U_{i, 1}, \ldots, U_{i, d}, V_{i, 1}, \ldots, V_{i, d}, W_{i, 1}, \ldots, W_{i, d}, Z_{i, 1}, \ldots, Z_{i, d} \}$$
of $4d$ many paths such that for every $j \in [d]$, $U_{i,j}$ is a $u_{i,j}$-$u'_{i,j}$-path, $V_{i,j}$ is a $v_{i,j}$-$v'_{i,j}$-path, $W_{i,j}$ is a $w_{i,j}$-$w'_{i,j}$-path, and $Z_{i,j}$ is a $z_{i,j}$-$z'_{i,j}$-path.

On the other hand assume that $W_{i}$ is a crosscap segment.
Let $\Tcal_{i} \coloneqq \{ T_{i, 1}, \ldots, T_{i, 2d} \}$ be the first $2d$ paths in left to right order of the rainbow of $W_{i}$.
Let $u_{i, j}$ and $v_{i,j}$ be the top boundary vertices of the base of $W_{i}$ in left to right order that are endpoints of $T_{i, j},$ $j \in [d]$.
Additionally, let $u'_{i,1}, \ldots, u'_{i,2d}, v'_{i, 1}, \ldots v'_{i, 2d}$ be the top boundary vertices of $W''_{i}$ in left to right order.
Our target is to identify a set
$$\Ycal_{i} = \{ U_{i, 1}, \ldots, U_{i, 2d}, V_{i, 1}, \ldots, V_{i, 2d} \}$$
of $4d$ many paths such that for every $j \in [2d],$ $U_{i, j}$ is a $u_{i, j}$-$u'_{i, j}$-path and $V_{i, j}$ is a $v_{i, j}$-$v'_{i, j}$ path, in the case $W_{i}$ is a crosscap segment.
We also want the set $\bigcup_{i \in [g]} \{ \Ycal_{i} \} = \{ Y_{1}, \ldots, Y_{4dg} \}$ to be a linkage and every path to be internally disjoint from $W'_{2}$.

We show how to define the $x$-$y$-path $Y_{i},$ $i \in [4dg],$ where $x$ and $y$ are the endpoints of $Y_{i}$ as previously defined.
Consider $C_{j}$ to be the $j$-th cycle from top to bottom and $Q_{j}$ to be the $j$-th vertical path from left to right of the base annulus of $W$, where $Q_{1}$ is the leftmost vertical path of $W_{0}$.
Also let $s$ be minimum such that $C_{s}$ intersects $W'_{2}$.
Note that by construction, all top boundary vertices of $W'_{2}$ belong to $C_{s}$ and that $s > 4dg$.
Now, assume w.l.o.g. that $x$ is the topmost vertex of $Q_{i_{x}}$ in $C_{1} \cap Q_{i_{x}}$ and that $y$ is the topmost vertex of $Q_{i_{y}}$ in $C_{s} \cap Q_{i_{y}},$ for some $i_{x}, i_{y} \in \Nbbb$.
Since $W'_{2}$ is contained in $W_{0}$, observe that $i_{x} < i_{y}$.
Then, we define $Y_{i}$ as the union of the subpath of $Q_{i_{x}}$ whose endpoints are $x$ and the topmost vertex of $Q_{i_{x}}$ in $Q_{i_{x}} \cap C_{s - i_{x} + 1},$ the subpath of $C_{s - i_{x} + 1}$ whose endpoints are the topmost vertex of $Q_{i_{x}}$ in $Q_{i_{x}} \cap C_{s - i_{x} + 1}$ and the topmost vertex of $Q_{i_{y}}$ in $Q_{i_{y}} \cap C_{s - i_{x} + 1}$ and whose internal vertices that belong to vertical paths are vertices of $Q_{j},$ $j \in (i_{x}, i_{y}),$ and the subpath of $Q_{i_{y}}$ whose endpoints are $y$ and the topmost vertex of $Q_{i_{y}}$ in $Q_{i_{y}} \cap C_{s - i_{x} + 1}$.

It is not difficult to see that $\Ycal$ is a linkage whose paths are internally disjoint from $W'_{2}$.
Moreover by their definition, the graph $W'_{3} \coloneqq \cupall \{ W'_{2} \} \cup \Ycal$ is a subgraph of $W$, obtained as the cylindrical closure of a $(z \times z)$-wall segment $W'''_{0}$ and a sequence $W'''_{1}, \ldots, W'''_{g}$, where $W'''_{i}$ is the $(z \times d)$-handle segment obtained by the union of $W''_{i}$, the two rainbows of $W_{i}$, and the linkage $\Ycal_{i}$, in the case that $W_{i}$ is a handle segment, and where $W'''_{i}$ is the $(z \times d)$-crosscap segment obtained by the union of $W''_{i}$, the rainbow of $W_{i}$, and the linkage $\Ycal_{i}$, in the case $W_{i}$ is a crosscap segment.
Moreover, the big enclosure of $W'_{3}$ is $\chi$-homogeneous in $\delta$, since its base annulus is the previously defined annulus wall $W'_{2}$.

\smallskip
\textbf{Step 4: Making the handle and crosscap enclosures homogeneous.} Fix $i \in [g]$.
If $W_{i}$ is a handle segment, we define the graph $L_{i}$ as the union of the left enclosure of $W'''_{i},$ the left rainbow of $W_{i},$ the linkage $\{ U_{i, 1}, \ldots, U_{i, d} \}$, and the linkage $\{ W_{i, 1}, \ldots, W_{i, d} \}$.
Symmetrically, we define the graph $R_{i}$ as the union of the left enclosure of $W'''_{i},$ the left rainbow of $W_{i},$ the linkage $\{ V_{i, 1}, \ldots, V_{i, d} \}$, and the linkage $\{ Z_{i, 1}, \ldots, Z_{i, d} \}$.
Similarly, if $W_{i}$ is a crosscap segment, we define the graph $S_{i}$ as the union of the enclosure of $W'''_{i},$ the rainbow of $W_{i},$ the linkage $\{ U_{i, 1}, \ldots, U_{i, 2d} \}$, and the linkage $\{ V_{i, 1}, \ldots, V_{i, 2d} \}$.
Observe that the graphs $L_{i}$ and $R_{i}$ are $d$-ladders where each of their $d$ many horizontal paths contains a common vertex with two vertical paths of $W_{3}.$
Also the graphs $S_{i}$ are $2d$-ladders where each of their $2d$ many horizontal paths contains a common vertex with two vertical paths of $W'_{2}$.
Observe that every brick of each of the graphs $L_{i}, R_{i}, S_{i}$ is a fence of $W$ and as such its $\chi$-coloring in $\delta$ is well-defined.
Now, similarly to Step 1, we define a brick-coloring for each of $L_{i}, R_{i}$ via the $\chi$-coloring of their bricks in $\delta$ and apply \cref{cor_homogeneity_1d}, to obtain $t$-ladders $L'_{i}, R'_{i}$ and $2t$-ladders $S'_{i}$ whose bricks individually all have have the same $\chi$-coloring in $\delta$.
Finally, we obtain our desired $(r, t)$-$\Sigma$-annulus wall $W'_{4}$ as the subgraph of $W'_{2}$ that contains all cycles of $W'_{2}$, all vertical paths of $W'''_{0}$, and from the remaining vertical paths only those that contain an endpoint of some horizontal path of some $L'_{i}, R'_{i}, S'_{i}$, union all $L'_{i}, R'_{i}, S'_{i}$.
It is not difficult to observe that all enclosures of $W'_{4}$ are now $\chi$-homogeneous in $\delta$.
Moreover, the fact that $W'_{4}$ is controlled by $\Tcal_{W}$ follows since $W'_{4}$ is a obtained as a subgraph of $W'_{4}$.
\end{proof}

\section{Representing colors from the homogeneous area} \label{sec_representation}

In this section, we obtain the second step towards the final $\Sigma$-schema.
We prove that, assuming that the walloid in our schema is large enough and homogeneous with respect to a cell coloring of the decomposition of the schema, we can find a still large walloid, which is still homogeneous, and moreover sufficiently represents the coloring of every one of its homogeneous enclosures.
Our main goal is to prove \cref{lem:representation}.

\subsection{Fences and routing cells through the walloid}\label{subsec:routing_cells}

We first introduce the notion of a fence at some distance from some other fence in a walloid, which will be useful for several arguments concerning the routing of paths in our walloid in the following proofs.
We then prove a series of auxiliary results that will serve as intermediate steps towards the proof of \cref{lem:representation}.

\medskip
Let $(G, \delta, W)$ be a $\Sigma$-schema of a graph $G$ in a surface $\Sigma$.
Let $F$ be a fence of $W$.
We say that $F_{0} \coloneqq F$ is the fence \emph{at distance $0$ from $F$ in $W$}.
Given $d \in \Nbbb_{\geq 1}$, we recursively define the \emph{fence at distance $d$ from $F$ in $W$} as follows:
If $F_{d-1}$ is disjoint from $C^{\mathsf{si}}(W)$ and $C^{\mathsf{ex}}(W)$, then $F_{d}$ is the unique fence of $W$ whose bricks are exactly the bricks of $F_{d-1}$ union all bricks of $W$ that share a vertex with some brick of $F_{d-1}$.

We call a fence $F$ of $W$ \emph{$d$-internal} if the fence at distance $d$ from $F$ is well-defined.
Given an $1$-internal fence $F$ of $W$, we define its \emph{internal pegs} (resp. \emph{external pegs}) as its vertices that are incident to edges of $W$ that are not edges of $F$ and that are edges of bricks of $F$ (resp. of bricks that are not bricks of $F$).

\medskip
In what follows we extensively utilize property $\mathbf{S}_{1}$ of the definition of schemas.
We may observe the following for which we provide a short proof sketch.

\begin{observation} Let $(G, \delta, W)$ be a $\Sigma$-schema and $\Delta$ be a $\delta$-aligned disk such that $\delta \cap \Delta$ contains no vortex cell of $\delta$ and there exists a ground vertex of $\delta$ not contained in $G \cap \Delta.$
Then for every cell $c \in C(\delta \cap \Delta)$ there exists a $\pi_{\delta}(\tilde{c})$-$V(\Omega(\Delta))$-linkage of order $|\tilde{c}|$ in $G \cap \Delta.$
\end{observation}
\begin{proof}[Proof sketch] Consider the $\delta$-torsoid $T$ of $G.$
Let $T^{+}$ be the graph obtained from $T$ by adding a fresh vertex $u$ to $T$ adjacent to every vertex in $\pi_{\delta}(\tilde{c}).$
Since by the second property of the tightness of $\delta,$ $T$ is $3$-connected, Menger's Theorem tells us that there exists a set $\mathcal{P}$ of $|\tilde{c}|$ internally disjoint $u$-$v$ paths in $T^{+},$ where $v$ is any ground vertex of $\delta$ not contained in $G \cap \Delta,$ which exists by assumption.

Clearly, every path $P \in \mathcal{P}$ must intersect $V(\Omega(\Delta))$ to reach $v.$
We may also observe that the unique minimal $\pi_{\delta}(\tilde{c})$-$V(\Omega(\Delta))$-subpath $P'$ of $P,$ can be lifted to a $\pi_{\delta}(\tilde{c})$-$V(\Omega(\Delta))$-path $P''$ of $G \cap \Delta.$
This follows by utilizing the first property of the tightness of $\delta.$
Therefore, collecting all paths $P''$ defined like that, we obtain the desired linkage.
\end{proof}

This observation allows us to use the following proposition which is a restatement of \cite[Lemma 22]{BasteST19HittingMinors} in our terminology.

\begin{proposition}[Lemma 22, \cite{BasteST19HittingMinors}]\label{prop:cell_paths_to_pegs}
Let $(G, \delta, W)$ be a $\Sigma$-schema.
For every $2$-internal brick $B$ of $W$ and every simple cell $c$ in the $\delta$-influence of $B,$ $W$ contains a linkage of order $|\tilde{c}|$ between $\pi_{\delta}(\tilde{c})$ and the external pegs of the fence $F$ at distance $1$ from $B$ in $W.$
Moreover these paths are drawn inside of the $\mathsf{nodes}_{\delta}(C^{\mathsf{si}}(W))$-avoiding disk of the trace of the fence at distance $2$ from $B$ in $W.$
\end{proposition}

The following lemma will be useful in order to extract the flap segments needed to obtain the desired schema in \cref{lem:representation}.

\begin{lemma}\label{lem:paths_to_bottom}
Let $(G, \delta, W)$ be a $\Sigma$-schema.
For every $5$-internal brick $B$ of $W$ and every simple cell $c$ in the $\delta$-influence of $B,$ $G$ contains a linkage of order $|\tilde{c}|$ between $\pi_{\delta}(\tilde{c})$ and the bottom boundary vertices of the $(12 \times 12)$-wall segment $W'$ that is a subgraph of $W$ such that every brick of $W'$ is a brick of $W$ and every brick of the fence at distance $5$ from $B$ is a brick of $W'$.
Moreover, these paths are contained in $W'$ union the graph drawn in the $\mathsf{nodes}_{\delta}(C^{\mathsf{si}}(W))$-avoiding disk of the trace of the fence at distance $1$ from $B$.
\end{lemma}
\begin{proof}

\begin{figure}[h]
\centering
\includegraphics{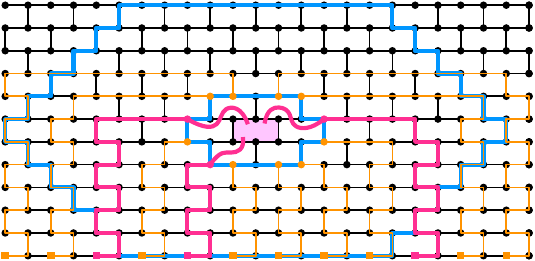}
\caption{\label{fig_brick_to_bottom_boundary}$B$ is illustrated in magenta. The inner blue cycle, say $F$, is the fence at distance $1$ from $B$ and the outer blue cycle is the fence at distance $5$ from $B$. The wiggly part of the pink linkage is implied by applying \cref{prop:cell_paths_to_pegs} on $B$ and $c$. We extend this linkage all the way to the bottom boundary paths of $W'$ by allocating a distinct orange path joining each external peg of $F$ to a bottom boundary vertex of $W'$.}
\end{figure}

The proof is fairly straightforward but quite tedious to formally present.
Instead we rely on a proof by picture.
See \cref{fig_brick_to_bottom_boundary}.
Let $B$ be a $5$-internal brick of $W$ and $c$ be a simple cell in the $\delta$-influence of $B$.
Moreover, let $W'$ be the $(12 \times 12)$-wall segment that is a subgraph of $W$ such that every brick of $W'$ is a brick of $W$ and every brick of the fence at distance $5$ is a brick of $W'$.
\cref{fig_brick_to_bottom_boundary} shows how to find the desired linkage when $|\tilde{c}| = 3$.
The other cases follow in the same manner.
\end{proof}

\subsection{Extracting $t$-flap segments representing colors}

We require one more auxiliary lemma that shows how to further progress with the construction our desired schema by extracting $b$ consecutive flap segments so that they all represent a color of a cell coloring of our $\Sigma$-decomposition which is not yet represented in our current schema.

\begin{lemma}\label{lem:representation_step}
There exists a function $f_{\ref{lem:representation_step}} \colon \Nbbb^{2} \to \Nbbb$ such that for every $r \in \Nbbb$, $t \in \Nbbb_{\geq 3}$, every $b \in \Nbbb_{\geq 1}$, every $(r, t' \coloneqq f_{\ref{lem:homogenization}}(t, b), \text{-}, \text{-}, \text{-})$-$\Sigma$-schema $(G, \delta, W)$ where $W$ is an $(r, t')$-$\Sigma$-walloid with no vortex segments, every cell-coloring $\chi$ of $\delta$ such that $W$ is $\chi$-homogeneous in $\delta$, and every $\alpha \in \mathbb{N}_{\geq 1}$ in the $\chi$-coloring of an enclosure of $W$ where $\alpha$ is not in the inclusion-maximum subset $S$ of $\chi$-$\mathsf{col}(\delta)$ which $W$ $b$-represents in $\delta$, there is an $(r, t, \text{-}, \text{-}, \text{-})$-$\Sigma$-schema $(G, \delta, W')$ controlled by $\Tcal_{W}$ such that $W'$ is an $(r, t, \text{-})$-$\Sigma$-walloid with no vortex segments which is $\chi$-homogeneous in $\delta$ and $b$-represents $S \cup \{ \alpha \}$ in $\delta$.

Moreover, it holds that $f_{\ref{lem:representation_step}}(t, b) \in \Ocal(t b).$
There also exists an algorithm that finds the outcome above in time $\poly(t + b) \cdot |E(G)||V(G)|$.
\end{lemma}
\begin{proof}

We define $f_{\ref{lem:representation_step}}(t, b) \coloneqq t + 2(6 + b \cdot (2t + 12)).$

First, note that since $W$ is $\chi$-homogeneous in $\delta$, the $\delta$-influence of every brick of an enclosure of $W$ whose $\chi$-coloring contains $\alpha$, contains a simple cell of $\delta$ whose color is $\alpha$.
The proof proceeds by demonstrating how to use the infrastructure provided by $W$ to enhance (a smaller version of) it with $b$ additional consecutive $(r, t)$-flap segments, such that the drawing of their hyperedges, corresponds to the disks of $b$ simple cells of $\delta$, hailing from $b$ distinct bricks of an enclosure of $W$ whose $\chi$-coloring contains $\alpha$, effectively defining a new walloid that $b$-represents $S \cup \{ \alpha \}$ in $\delta$.

Our second observation is the following.
Let $F_{1} \coloneqq C^{\mathsf{ex}}(W).$
For every fence $F_{d},$ $d \leq 5 + b \cdot (2t + 10)$, at distance $d$ from $C^{\mathsf{ex}}(W),$ we can naturally define an $(r, t' - 2d)$-$\Sigma$-walloid $W_{t' - 2d}$ by downscaling the order of each segment of $W$.
Indeed, let $W_{t' - 2d}$ be the graph obtained by the union of any $r + t' - 2d$ base cycles of $W$, the $t' - 2d$ paths of the two rainbows of each handle segment, any $2(t' - 2d)$ paths of the rainbow of each crosscap segments, and any $t' - d$ paths of the rainbow of each flap segment, all disjoint from $F_{1}, \ldots, F_{d}$, all hyperedges of $W$ (along with the edges connecting them to top boundary vertices of the base of their respective flap segment), the vertical paths of $W$ intersected by the previously selected paths and edges, and any $r + t' - 2d$ vertical paths from the wall segment of $W$, after removing all degree one vertices that may appear.

Observe that $W_{t'-2d}$ is well-defined since $F_{1}, \ldots, F_{d}$ only intersect the base cycles $C_{1}, \ldots, C_{d}$ of $W,$ the first and last $d$ paths of both rainbows of each handle segment, the first and last $d$ paths of the rainbow of each crosscap segment, and the first $d$ paths of the rainbow of each flap segment.

Now, let $W^{\alpha}$ be a segment of $W$ with an enclosure whose $\chi$-coloring contains $\alpha$.
We proceed by distinguishing cases based on the type of segment that $W^{\alpha}$ is.

\smallskip
\textbf{Extracting from a handle or crosscap enclosure:} Assume that $W^{\alpha}$ is a crosscap segment.
Let $p \coloneqq 2t + 12$ and $q \coloneqq bp.$
Consider the fences $\Fcal \coloneqq \{ F_{1} \ldots F_{q + 6} \}$ at distance $0$ up to $q + 5$ from $F_{1}$.
Let $P_{1}, \ldots, P_{t'}$ be the horizontal and $Q_{1}, \ldots, Q_{4t'}$ be the vertical paths of the base of $W^{\alpha}$ in top to bottom and left to right order respectively.
Moreover, let $\Rcal = \{ R_{1}, \ldots, R_{2t'} \}$ be the rainbow of $W^{\alpha}$ in left to right order.
Observe that $P_{i},$ $i \in [q + 6]$ is a subpath of $F_{i},$ and that $R_{i}$ and $R_{2t' - i + 1}$ are also subpaths of $F_{i}.$

Let $\langle B_{1}, \ldots, B_{b} \rangle$ be a sequence of $b$ bricks of $W$, ordered from left to right (according to the ordering of $\Rcal$) obtained as follows.
We select $B_{i},$ $i \in [b]$ as the brick of the enclosure of $W^{\alpha}$ that intersects both $R_{6 + (i - 1) \cdot p + 6}$ and $B_{i} \cap R_{6 + (i - 1) \cdot p + 7}$.
Moreover, for every $B_{i},$ $i \in [b],$ we reserve a bundle of $p$ many consecutive cycles $\Fcal_{i} \subseteq \Fcal$ defined as $\Fcal_{i} \coloneqq \{ F_{6 + (i - 1) \cdot p + 1}, \ldots, F_{6 + i \cdot p} \}$ (see \cref{fig:extract_crosscap}).

\begin{figure}[h]
\centering
\includegraphics{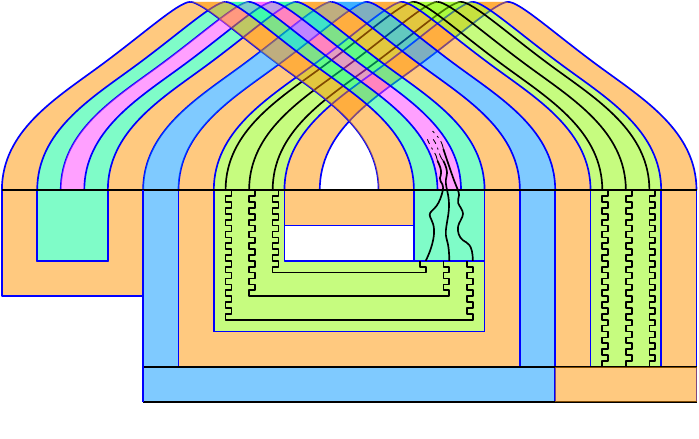}
\caption{\label{fig:extract_crosscap}Extracting a crosscap segment $W_{1}$ and a flap segment $W_{2}$ from a crosscap segment $W$. The magenta area represents a brick $B$ in the enclosure of $W$. The green area represents the perimeter of the $(12 \times 12)$-wall segment $W'$ containing $B$. The wiggly paths represent the linkage between the vertices on the boundary of the chosen cell from the influence of $B$ and bottom boundary vertices of $W'$. The blue area represents the rainbow and base of $W_{1}$. The orange area represents the rainbow and base of $W_{2}$.}
\end{figure}

Now for every $i \in [b]$, let $W_{i}$ be the $(12 \times 12)$-wall segment subgraph of $W$ such that every brick of $W_{i}$ is a brick of $W$, every brick of the fence at distance $5$ from $B_{i}$ is a brick of $W_{i}$, $W_{i}$ is disjoint from $W_{j},$ $j \in [b] \setminus \{ i \}$, $W_{i}$ is disjoint from the first and last $t$ fences in $\Fcal_{i}$, and exactly $2$ subpaths of each of the horizontal paths $P_{j}$, $j \in [6]$, are horizontal paths of $W_{i}$.
We assume that the bottom horizontal path of $W_{i}$ is the rightmost such subpath of $P_{6}.$
Moreover, for every $i \in [b]$, consider a cell $c_{i}$ in the $\delta$-influence of $B_{i}$ whose color is $\alpha$ (as we observed such a cell is bound to exist).
With the wall $W_{i}$ at hand we may now call upon \cref{lem:paths_to_bottom} for $B_{i}$ and $c_{i}$, thus obtaining a linkage $\Pcal_{i}$ of order $|\tilde{c_{i}}|$ between $\pi_{\delta}(\tilde{c_{i}})$ and the bottom boundary vertices of $W_{i}$ with each path in $\Pcal_{i}$ contained in $W_{i}$ union the graph drawn in the $\mathsf{nodes}_{\delta}(C^{\mathsf{si}}(W))$-avoiding disk of the trace of the fence at distance $1$ from $B_{i}$.

Next for every $i \in [b]$, we define a linkage $\Lcal_{i}$ of order $t$ that will serve as the rainbow of the $i$-th flap segment we are trying to define.
Let $\Fcal_{i}'$ be a set consisting of $2t$ paths, each a subpath of one of the first or last $t$ fences in $\Fcal_{i}$ with one endpoint being the leftmost top boundary vertex of the base of $W^{\alpha}$ visited by the respective cycle and with the other endpoint being the rightmost top boundary vertex of the base of the rightmost segment of $W$ that is not its wall segment.
Notice that the paths in $\Fcal_{i}'$ are made up of $2$ bundles of $t$ consecutive paths.
We define $\Lcal_{i}$ by connecting the leftmost endpoint of each path of one bundle to the leftmost endpoint of a path in the other bundle, using the $2t$ vertical paths of the base of $W^{\alpha}$ that intersect the paths of the two bundles and the $t$ horizontal paths $P_{7}, \ldots, P_{t + 6}$ of the base $W^{\alpha},$ in a way that we obtain a linkage of order $t$ disjoint from all others.

Additionally, notice that for every $i \in [b],$ each endpoint of a path in $\Pcal_{i}$ not in $\pi_{\delta}(\tilde{c_{i}})$ is a vertex of a distinct fence in $\Fcal_{i}.$
Using this observation, we define a new set of paths $\Tcal_{i}$ as an extension of each path $P \in \Pcal_{i}$ by considering the union of $P$ with the subpath of the corresponding fence that connects the endpoint of $P$ not in $\pi_{\delta}(\tilde{c_{i}})$ to the rightmost top boundary vertex of the base of the rightmost segment of $W$ that is not its wall segment visited by the fence.
Note that the set $$\bigcup_{i \in [b]} \{ \Tcal_{i} \} \cup \bigcup_{i \in [b]} \{ \Lcal_{i} \}$$ is a linkage, by construction.

To conclude with this case, observe that $t' - 2|\Fcal| = t.$
We obtain our desired $(r, t)$-$\Sigma$-walloid $W'$ by considering the $(r, t)$-$ \Sigma$-walloid $W_{t' - 2|\Fcal|}$, adding the linkages $\Lcal_{i}$ and $\Tcal_{i}$, $i \in [b]$, adding all vertical paths of $W$ intersected by the former linkages, and last but not least adding $\pi_{\delta}(\tilde{c_{i}})$ as a hyperedge of (the $i$-th flap segment of) $W'$, $i \in [b]$.

As for the case that $W^{\alpha}$ is a handle segment the proof is very similar and we omit it in favour of brevity.

\smallskip
\textbf{Extracting from the big enclosure:} The only remaining case is when $W^{\alpha}$ is the $((r + t') \times (r + t'))$-wall segment of $W$.
Observe that the bricks of the big enclosure of $W$ are the bricks of the base annulus $\widetilde{W}$ of $W$ and $W^{\alpha}$ is clearly a subgraph of $\widetilde{W}$.
Let $P_{1}, \ldots, P_{r + t'}$ be the horizontal and $Q_{1}, \ldots, Q_{r + t'}$ be the vertical paths of $W^{\alpha}$.

Let $\langle B_{1}, \ldots, B_{b} \rangle$ be a sequence of $b$ bricks of $W^{\alpha}$, ordered from left to right (respecting the ordering of the vertical paths of $W^{\alpha}$) obtained as follows.
We select $B_{i}$, $i \in [b]$, as the brick bounded by $P_{t + 6}$, $P_ {t + 7}$, $Q_{2(q + 6) + (i - 1) \cdot p + 6}$, and $Q_{2(q + 6) + (i - 1) \cdot p + 7}$.
Moreover, for every $i \in [b],$ we reserve a bundle of $p$ consecutive vertical paths of $W,$ from $Q_{2(q + 6) + (i - 1) \cdot p + 1}$ up to $Q_{2(q + 6) + i \cdot p}.$

Now for every $i \in [b]$, let $W_{i}$ be the $(12 \times 12)$-wall segment subgraph of $W$ such that every brick of $W_{i}$ is a brick of $W$, every brick of the fence at distance $5$ from $B_{i}$ is a brick of $W_{i}$, $W_{i}$ is disjoint from $W_{j},$ $j \in [b] \setminus \{ i \}$, $W_{i}$ is disjoint from $P_{1}, \ldots, P_{t}$, and $W_{i}$ is disjoint from the first and last $t$ paths of the bundle of vertical paths we reserved for $B_{i}$.
Moreover, for every $i \in [b]$, consider a cell $c_{i}$ in the $\delta$-influence of $B_{i}$ whose color is $\alpha$ (as we observed such a cell is bound to exist).
With the wall $W_{i}$ at hand we may now call upon \cref{lem:paths_to_bottom} for $B_{i}$ and $c_{i}$, thus obtaining a linkage $\Pcal_{i}$ of order $|\tilde{c_{i}}|$ between $\pi_{\delta}(\tilde{c_{i}})$ and the bottom boundary vertices of $W_{i}$ with each path in $\Pcal_{i}$ contained in $W_{i}$ union the graph drawn in the $\mathsf{nodes}_{\delta}(C^{\mathsf{si}}(W))$-avoiding disk of the trace of the fence at distance $1$ from $B_{i}$.

We immediately extend each path $P \in \Pcal_{i}$ to reach all the way to a bottom boundary vertex of $W^{\alpha}$ via the vertical path of $W^{\alpha}$ intersected by $P$.

Also for every $i \in [b],$ we define a linkage $\Lcal_{i}$ of order $t$ by taking the first and last $t$ vertical paths reserved for $B_{i}$ and using the intersection between these vertical paths and the paths $P_{1}, \ldots, P_{t}$.
Notice that we can define each $\Lcal_{i}$ to be disjoint from any $\Lcal_{j},$ $j \in [b] \setminus \{ i \}.$

To conclude with this case, we obtain our desired $(r, t)$-$\Sigma$-walloid $W'$ by considering the $(r, t)$-$\Sigma$-walloid $W_{t' - 2(q + 6)}$ (in this case we forcibly remove the first $2(q + 6)$ vertical paths of the wall segment of $W$), adding the linkages $\Lcal_{i}$ and $\Pcal_{i}$, $i \in [b]$, and last but not least adding $\pi_{\delta}(\tilde{c_{i}})$ as a hyperedge of (the $i$-th flap segment of) $W'$, $i \in [b]$.

Finally, we may observe that by construction, the simple cycle of $W$ is equal to that of $W'$ and the big enclosure of $W'$ is a fence of $W$ contained in the big enclosure of $W$ such that every one of its bricks is a brick of $W$ and as such remains $\chi$-homogeneous in $\delta$.
Moreover the remaining $2\mathsf{h} + \mathsf{c}$ enclosures of $W'$ that correspond to the enclosures of its handle and crosscap segments can be seen to also be $\chi$-homogeneous in $\delta$ as their bricks originate from bricks of the corresponding handle or crosscap enclosures of $W$, after subsuming bricks from the big enclosure of $W$, and all such bricks have the same $\chi$-coloring in $\delta$.
Also, again by construction, $W'$ can be seen to $b$-represent $S \cup \{ \alpha \}$ in $\delta$, which concludes the proof.
Lastly, $W'$ is controlled by $\Tcal_{W}$, since the base annulus of $W'$ is a subgraph of the base annulus of $W$.
\end{proof}

\subsection{The representation lemma}

We are now in the position to prove the main result of this section, which by an iterative application of \cref{lem:representation_step}, will give us a schema that sufficiently represents every color of the given cell coloring that belongs to the coloring of some enclosure of the walloid.

\begin{lemma}[Representation lemma]\label{lem:representation}
There exists a function $f_{\ref{lem:representation}} \colon \Nbbb^{3} \to \Nbbb$ such that for every $r \in \Nbbb$, every $t \in \Nbbb_{\geq 3}$, every $\ell, b \in \Nbbb_{\geq 1}$, every $(r, t' \coloneqq f_{\ref{lem:representation}}(t, \ell, b), \text{-}, \text{-}, \text{-})$-$\Sigma$-schema $(G, \delta, W)$ where $W$ is an $(r, t')$-$\Sigma$-annulus wall, and every cell-coloring $\chi$ of $\delta$ of capacity at most $\ell$ such that $W$ is $\chi$-homogeneous in $\delta$, there is an $(r, t, \text{-}, \text{-}, \text{-})$-$\Sigma$-schema $(G, \delta, W')$ controlled by $\Tcal_{W}$ such that $W'$ is an $(r, t, \text{-})$-$\Sigma$-walloid with no vortex segments which is $\chi$-homogeneous in $\delta$ and $b$-represents the $\chi$-coloring of every enclosure of $W'$ in $\delta$.

Moreover, it holds that $f_{\ref{lem:representation}}(t, \ell, b) \in \Ocal(t \ell b).$
There also exists an algorithm that finds the outcome above in time $\poly(t + \ell + b) \cdot |E(G)||V(G)|$.
\end{lemma}
\begin{proof}

We define $f_{\ref{lem:representation}}(t, \ell, b) \coloneqq f_{\ref{lem:representation_step}}(t, b) \cdot \ell$.

Let $(G, \delta, W_{0}) = (G, \delta, W)$ and $i \coloneqq 0.$
Iteratively apply \cref{lem:representation_step} on $(G, \delta, W_{i})$ with a color $\alpha$ that belongs to the $\chi$-coloring of an enclosure of $W_{i}$ that is not $b$-represented by $W_{i}$ in $\delta$.
Let $(G, \delta_{i+1}, W_{i+1})$ be the result.
Set $i \coloneqq i + 1,$ and repeat.
\end{proof}

\section{Splitting vortex segments and extracting flap segments}
\label{sec_splitting_extracting}

In this section, we show that given a $\Sigma$-schema equipped with a walloid that is large enough, homogeneous with respect to a cell coloring of the decomposition of the schema, and sufficiently represents the coloring of every one of its homogeneous enclosures, we can find a coarser $\Sigma$-schema to one with a bounded number of vortex segments each of bounded depth, such that all colors not yet represented, may be cornered in the interior of the vortices within our vortex segments.
Our goal is to prove \cref{lem:coarsening}.
This step marks the last step towards our goal.
The proof of \cref{lem:coarsening} will be an inductive one based on a sequence of lemmas which we derive in the following subsections.

\paragraph{Schemas with a single vortex segment.}

We first observe that we may view any $(r, (t + s + 2), \text{-}, \text{-})$-$\Sigma$-walloid as an $(r, t, s)$-$\Sigma$-walloid with a single $((r + t) \times t, s)$-vortex segment by sacrificing the first $s + 2$ many fences at distance $0$ up to $s + 1$ from its exceptional cycle.
Moreover, any properties that the original walloid satisfied in its schema will also hold true in the new schema. The next observation helps us to extract the fist vortex segment from a walloid without vortex segments as the one provided  by \cref{lem:representation}.
As we see in the proof, this vortex segment is extracted so that it contains the whole exceptional face and enough many layers around it.

\begin{observation}\label{obs:single_vortex_segment_schema}
For every $t \in \Nbbb_{\geq 4}$, every $s, b \in \Nbbb_{\geq 1}$, every $r, x, y \in \Nbbb$, every $(r, (2t + s + 2), \text{-}, x, y)$-$\Sigma$-schema $(G, \delta, W)$ where $W$ is an $(r, (2t + s + 2), \text{-})$-$\Sigma$-walloid with no vortex segments, and every cell-coloring $\chi$ of $\delta$ such that $W$ is $\chi$-homogeneous in $\delta$ and $b$-represents the $\chi$-coloring of every enclosure of $W$ in $\delta$, there is an $(r, t, s, 1, \text{-})$-$\Sigma$-schema $(G, \delta', W')$ controlled by $\Tcal_{W}$ such that $\delta' \sqsubseteq \delta$, $W'$ has a single $((r + t) \times t, s)$-vortex segment and $\delta'$ has a single vortex cell $c^{*}$, $W'$ is $\chi$-homogeneous in $\delta'$ and $b$-represents the $\chi$-coloring of every enclosure of $W'$ in $\delta'$, and $\delta \cap \Delta_{c^{*}}$ is a rendition of the vortex society of $c^{*}$ in $\Delta_{c^{*}}$ of breadth at most $x$ and depth at most $y.$

Moreover there exists an algorithm that finds the outcome above in time $\poly(t + s + b) \cdot |E(G)|$.
\end{observation}
\begin{proof}
Consider the fences $C^{\mathsf{ex}}(W) = F_{0}, \ldots, F_{s + 1}$ of distance $0$ up to $s + 1$ from the exceptional cycle $C^{\mathsf{ex}}(W)$ of $W$, let $\Delta_{0}$ be the $\mathsf{nodes}_{\delta}(C^{\mathsf{si}}(W))$-avoiding disk bounded by the trace of $C^{\mathsf{ex}}(W)$, and define $c^{*} \coloneqq \Delta_{0} \setminus \bd(\Delta_{0})$.
Let $\delta = (\Gamma, \Dcal)$.
We define $\delta' = (\Gamma, \Dcal')$ so that $$\Dcal' = (\Dcal \setminus \{ c \in C(\delta) \mid c \subseteq \Delta_{0} \}) \cup c^{*}.$$
Observe that $\delta \cap \Delta_{0}$ defines the desired rendition of the vortex society of $c^{*}$ in $\Delta_{0}.$
We define $W'$ by only keeping the last $r + t$ cycles of the base annulus of $W$, the appropriate amount of vertical paths cropped within the new base annulus of $W$ that are not vertical paths of its wall segment, in order to obtain the proper order of all segments except its wall segment.
And lastly for its wall segment, we keep the last $r + t$ vertical paths cropped within the new base annulus, and for each of the first $t$, we take a minimal subpath that has one endpoint in $\pi_{\delta}(\mathsf{nodes}_{\delta}(F_{0}))$.
These paths constitute the radial linkage of the desired $((r + t) \times t, s)$-vortex segment and are clearly orthogonal to its nest $\{ F_{0}, \ldots, F_{s + 1}\}$.
It is straightforward to verify that all other properties are satisfied.
\end{proof}

\subsection{Vortex segment societies and proper transactions}

Before we begin, we need some preliminary results on transactions in societies.
Our goal in the sequel is to make progress within the vortices in the interior of vortex segments starting from a schema as the one in the conclusion of \cref{obs:single_vortex_segment_schema}.
Towards this, we first define the \textsl{vortex segment societies} of a schema.

\paragraph{Vortex segment societies.}

Let $s \in \Nbbb_{\geq 1}$, and $(G, \delta, W)$ be a $(\text{-}, \text{-}, s, \text{-}, \text{-})$-$\Sigma$-schema.
Moreover, let $\{ C_{0}, \ldots, C_{s + 1} \}$ be the nest of a vortex segment of $W$.
A \emph{vortex segment society} of $(G, \delta, W)$ is a society $(H, \Omega)$ where $\Omega$ is a cyclic ordering of $\pi_{\delta}(\mathsf{nodes}_{\delta}(C_{s + 1}))$, and $H$ is the subgraph of $G$ drawn in the $\mathsf{nodes}_{\delta}(C^{\mathsf{si}}(W))$-avoiding disk of $\mathsf{trace}_{\delta}(C_{s+1})$.
We define $\Delta_{H}$ to be the $\mathsf{nodes}_{\delta}(C^{\mathsf{si}}(W))$-avoiding disk of $\trace_{\delta}(C_{s+1})$ and we shall refer to the cylindrical rendition $\delta \cap \Delta_{H} = (\Gamma, \Dcal, c_{0})$ of $(H, \Omega)$ as the \emph{$\delta$-canonical} rendition of $(H, \Omega)$ in $\Delta_{H}$.
Moreover, we say that a rendition $\rho$ of $(H, \Omega)$ in $\Delta_{H}$ is \emph{$\delta$-compatible} if $\delta \cap \Delta_{H} \sqsubseteq \rho$ and $C_{0}$ is $\rho$-grounded.

Note that in a $\Sigma$-schema $(G, \delta, W)$ as the one in the conclusion of \cref{obs:single_vortex_segment_schema}, the vortex segment society $(H, \Omega)$ corresponding to the single vortex segment of $\delta$ trivially has a $\delta$-compatible rendition in $\Delta_{H}$ with breadth at most $x$ and depth at most $y$ by combining $\delta \cap \Delta_{H}$ with the rendition $\rho$ of \cref{obs:single_vortex_segment_schema}.

\paragraph{Proper transactions.}

The previous observation implies that over the next few subsections, it suffices to deal with vortex segment societies for which we already know that there exists a rendition in a disk with at most $x$ vortices, each of depth at most $y.$
This immediately allows us to deduce that any large enough transaction in a vortex segment society contains a still large subtransaction which is planar and resides in a vortex-free part of our rendition.

\medskip
Let $(G, \Omega)$ be a society and $\rho$ be a rendition of $(G, \Omega)$ in a disk $\Delta$.
Moreover, let $\Pcal$ be a planar transaction of order at least $2$ in $(G, \Omega)$.
Let $\Pcal = \{ P_{1}, \ldots, P_{\ell} \}$ be ordered such that for each $i \in [\ell]$, $P_{i}$ has the endpoints $a_{i}$ and $b_{i}$ and $a_{1}, \ldots, a_{\ell}, b_{\ell}, \ldots, b_{1}$ appear in $\Omega$ in the order listed.
We call $P_{1}$ and $P_{\ell}$ the \emph{boundary paths} of $\Pcal$ in $(G, \Omega)$.
Also, let $L_{1}$ and $L_{2}$ be the two minimal arcs of $\bd(\Delta)$ which contain the nodes corresponding to $V(a_{1} \Omega a_{\ell})$ and $V(b_{\ell} \Omega b_{1})$ respectively.

We say that $\Pcal$ is \emph{$\rho$-proper} if $P_{1}$ and $P_{\ell}$ are $\rho$-grounded and the rendition $\rho \cap \Delta_{\Pcal}$ is vortex-free, where $\Delta_{\Pcal} \subseteq \Delta$ is the $\rho$-aligned disk bounded by $\trace_{\rho}(P_{1}) \cup \trace_{\rho}(P_{\ell}) \cup L_{1} \cup L_{2}$ that intersects the drawing of every path of $\Pcal$.
We call $\rho \cap \Delta_{\Pcal}$ the \emph{strip rendition} of $\rho$ and $\Pcal$.

\begin{lemma}\label{lem:proper_transaction}
Let $x \in \Nbbb_{\geq 1}$ and $p, y \in \Nbbb_{\geq 2}$.
Moreover, let $(G, \delta, W)$ be a $\Sigma$-schema and $(H, \Omega)$ be a vortex segment society of $(G, \delta, W)$ that has a $\delta$-compatible rendition $\rho$ in $\Delta_{H}$ with breadth at most $x$ and depth at most $y$.
Then every transaction in $(H, \Omega)$ of order at least $(x+1)(2xy+p)$ contains a transaction of order $p$ that is $\rho$-proper.
\end{lemma}
\begin{proof}
Let $\mathcal{L}$ be a transaction of order at least $(x+1)(2y+p)$ in $(H, \Omega)$.
Then there exist two disjoint segments of $\Omega$, say $I$ and $J$ such that every path in $\mathcal{L}$ has one endpoint in $I$ and the other in $J.$
Let us assume that $I$ and $J$ are chosen minimally with this property and let $\mathcal{L}=\{ L_1,L_2,\ldots,L_{\ell}\}$ be ordered with respect to the occurrence of the endpoint $u_i$ of $L_i$ in $I.$

Suppose there exists a transaction $\mathcal{L}'\subseteq\mathcal{L}$ such that there exists a path $P$ in $\mathcal{L}'$ which crosses all other paths in $\mathcal{L}'$.
If $|\mathcal{L}'| > xy$, then there must exist a transaction $\mathcal{L}'' \subseteq \mathcal{L}'$ of order at least $y+1$ and a vortex $c$ of $\rho$ such that $\mathcal{L}''$ induces a transaction of order at least $y+1$ in the vortex society of $c$.
Since $c$ is of depth at most $y$ this is impossible and thus $|\mathcal{L}'| \leq xy$.

We partition $I$ into $x+1$ consecutive segments $I_1, I_2, \ldots, I_{x+1}$ each containing at least $2xy+p$ endpoints of members of $\mathcal{L}.$
For each $i\in[x+1]$ there exists a family $\mathcal{L}_i$ containing $p$ consecutive paths such that $I_i$ has at least $xy$ other endpoints on either side.
Notice that, by the observation above, it is not possible that a path from $\mathcal{L}_i$ crosses a path from $\mathcal{L}$ which does not have an endpoint in $I_i.$
In particular, this means that for any choice of $i\neq j\in[x+1]$ there does not exist a vortex of $\rho$ which contains an edge of a path in $\mathcal{L}_i$ and an edge of a path of $\mathcal{L}_j.$
Therefore, there must exist $i\in[x+1]$ such that no path in $\mathcal{L}_i$ has an edge or a vertex that belongs to a vortex of $\rho.$
It follows that $\mathcal{L}_i$ is $\rho$-proper.
\end{proof}

\subsection{Exposed transactions and pruning}

\paragraph{Exposed transactions.}

To ensure we are actually making progress with our constructions we require a large part of a ``large enough'' transactions should ``traverse'' the central region of our vortex segment.  

\medskip
Let $(G, \Omega)$ be a society and $\rho$ be a rendition $\rho$ of $(G, \Omega)$ in the disk with a nest $\{C_{1}, \ldots, C_{s} \}$ for some $s \in \Nbbb_{\geq 1}$.
A transaction $\mathcal{P}$ in $(G, \Omega)$ is said to be \emph{$\rho$-exposed} if an edge of every path in $\mathcal{P}$ is drawn in the $\pi^{-1}_{\rho}(V(\Omega))$-avoiding disk of $\trace_{\rho}(C_{1})$.

\medskip
Please note that our definition of exposed transactions differs slightly from the definition in \cite{KawarabayashiTW2021Quickly}.
This is witnessed by the fact that we do not require the presence of an actual vortex cell for a transaction to be exposed.
For technical reasons, in our setting it suffices that the transaction crosses through the disk defined by the inner cycle of the vortex segment's nest.

\paragraph{Pruning a schema.} Our target is to utilize exposed transactions traversing through our vortex segment societies to ``split'' them and make progress.
Towards this, we utilize a technique originating from \cite{ThilikosW2024Killing}, to prove a lemma that allows us to to find in every large transaction, either a large exposed transaction, or an slightly different schema that reduces the part of the graph that belongs to the vortex segment society.

\medskip
Let $s \in \Nbbb_{\geq 1}$, $(G, \delta, W)$ be a $(\text{-}, \text{-}, s, \text{-}, \text{-})$-$\Sigma$-schema, and $(H, \Omega)$ be the vortex segment society of the $j$-th vortex segment of $W$ with nest $\{ C_{0}, \ldots, C_{s + 1} \}$, $j \in [\ell]$, where $\ell \in \Nbbb_{\geq 1}$ is the number of vortex segments of $W$.
Also, let $\rho$ be a $\delta$-compatible rendition of $(H, \Omega)$ in $\Delta_{H}$.
We define $\delta_{H, \rho}$ to be the $\Sigma$-decomposition of $G$ obtained by combining $\delta$ and $\rho$.
An $(\text{-}, \text{-}, s, \text{-}, \text{-})$-$\Sigma$-schema $(G, \delta', W')$ is said to be a \emph{$\rho$-pruning} of $(G, \delta, W)$ if $\delta' \sqsubseteq \delta_{H, \rho}$ and
\begin{itemize}
    \item all segments of $W'$ except its vortex segments are segments of $W$,
    \item $W'$ has $\ell$ vortex segments and for every $i \in [\ell] \setminus \{ j \}$, the $i$-th vortex segment of $W'$ is the $i$-th vortex segment of $W$, and
    \item if $(H', \Omega')$ is the vortex segment society of the $j$-th vortex segment of $W'$ with nest $\mathcal{C} = \{ C'_0, C'_1, \ldots, C'_{s+1} \}$ and $\delta'$-canonical rendition $\rho'$ in $\Delta_{H'}$, then
    \begin{itemize}
    \item either $\Delta_{H'}$ is properly contained in $\Delta_{H}$ or
    \item there exists $i \in [s]$ such that if $H_{i} \coloneqq G \cap \Delta_{i}$, where $\Delta_i$ is the $\mathsf{nodes}_{\rho}(C_{s+1})$-avoiding disk of $\mathsf{trace}_{\rho}(C_i),$ and $H'_{i} \coloneqq G \cap \Delta'_{i}$, where $\Delta_i'$ is the $\mathsf{nodes}_{\rho'}(C'_{s+1})$-avoiding disk of $\mathsf{trace}_{\rho'}(C'_i),$ then either $H'_i-V(C_i')\subsetneq H_i-V(C_i)$ or $E(H'_i)\subsetneq E(H_i)$.
    \end{itemize}
\end{itemize}
Moreover, if $\chi$ is a cell-coloring of $\delta_{H,\rho}$, we say that $(G, \delta', W')$ is a \emph{$\chi$-respectful} $\rho$-pruning of $(G, \delta, W)$ if $\chi$-$\mathsf{col}(\delta') \subseteq \chi$-$\mathsf{col}(\delta)$.

\medskip
Thus, the pruning of a schema is essentially accomplished by either slightly pushing one of the cycles within the nest of a vortex segment closer to the untamed area or by extending the nest into this untamed area.
This process effectively tames a portion of the vortex within the segment, while ensuring that no simple cells containing new colors are displaced.

\begin{lemma}\label{lem:exposed_transaction}
Let $s, p \in \Nbbb_{\geq 1}$ and $(G, \delta, W)$ be a $(\text{-}, \text{-}, s, \text{-}, \text{-})$-$\Sigma$-schema.
Let $(H, \Omega)$ be a vortex segment society of $(G, \delta, W)$ that has a $\delta$-compatible rendition $\rho$ in $\Delta_{H}$, $\chi$ be a cell-coloring of $\delta_{H, \rho}$, and $\mathcal{P}$ be a transaction of order at least $2s + p + 2$ in $(H, \Omega)$ that is $\rho$-proper.
Then there exists
\begin{itemize}
    \item a $\rho$-exposed transaction $\mathcal{Q} \subseteq \mathcal{P}$ of order $p$ that is $\rho$-proper, or
    \item a $\chi$-respectful $\rho$-pruning of $(G, \delta, W).$
\end{itemize}
Moreover, there exists an algorithm that finds one of the two outcomes in time $\mathsf{poly}(s + p) \cdot |E(G)|$.
\end{lemma}
\begin{proof}

Let $\{ C_{0}, \ldots, C_{s+1} \}$ be the nest of the vortex segment that corresponds to $(H, \Omega)$ and $\mathcal{C} = \{ C_{0}, \ldots, C_{s} \}.$
If $\mathcal{P}$ contains a $\rho$-exposed transaction of order $p$ we are immediately done.
Moreover, we can check in linear time for the existence of such a transaction by checking for each path in $\mathcal{P}$ individually if it is $\rho$-exposed.

\smallskip
\textbf{Step 0:} First, observe that we may assume that every path of $\mathcal{P}$ is intersected by the $\mathsf{nodes}_{\delta}(C_{s+1})$-avoiding disk of $\trace_{\rho}(C_{s})$.
Indeed, if there is a path $P$ which violates this condition then we can define a $\delta$-grounded cycle $C'$ by replacing a subpath of $C_{s+1}$ whose endpoints are the same as the endpoints of a minimal $V(C_{s+1})$-$V(C_{s+1})$-subpath of $P$ in a way that defines a cycle whose trace in $\delta$ bounds a disk $\Delta'$ which contains $\trace_{\rho}(C_{0}).$
It is then easy to see that $\Delta'$ is properly contained in $\Delta$.
However, before we replace our working vortex segment with a vortex segment whose society is $(G \cap \Delta', \Omega(\Delta'))$ to obtain a $\chi$-respectful $\rho$-pruning of $(G, \delta, W)$, we have to verify that the paths in $\mathcal{R}$ remain orthogonal to the new cycle $C'$.
For this, skip to \textbf{Step 2}.

\smallskip
\textbf{Step 1:} Hence, we may assume that there exists a linkage $\mathcal{Q} \subseteq \mathcal{P}$ of order $2(s + 1) + 1$ such that \textsl{no} path of $\mathcal{Q}$ is $\rho$-exposed and every path is intersected by the $\mathsf{nodes}_{\delta}(C_{s+1})$-avoiding disk of $\trace_{\rho}(C_{s})$.
Let $\Delta^{*}$ denote the $\mathsf{nodes}_{\delta}(C_{s+1})$-avoiding disk of $\trace_{\rho}(C_{0})$.
It follows that no path in $\mathcal{Q}$ intersects the interior of $\Delta^{*}$.
Notice that each member $Q$ of $\mathcal{Q}$ naturally separates $\Delta$ into two disks, exactly one of which contains $\Delta^*$.
Let us call the other disk the \emph{small side of $Q$}.
Given two members of $\mathcal{Q}$ then either the small side of one is contained in the small side of the other, or their small sides are disjoint.
It is straightforward to see that if we say that two members are \emph{equivalent} if their small sides intersect, this indeed defines an equivalence relation on $\mathcal{Q}.$
Moreover, there are exactly two equivalence classes, one of which, call it $\mathcal{Q}'$, must contain at least $s + 2$ members.
Since $|\mathcal{C}| = s + 1$ there must exist some $i \in [s]$ and some subpath $L$ of some path in $\mathcal{Q}'$ such that
\begin{enumerate}
\item both endpoints of $L$ belong to $V(C_i)$,
\item $L$ is internally disjoint from $\bigcup_{i\in[0, s]} V(C_i)$,
\item $L$ contains at least one edge that does not belong to $C_i$, and
\item $L$ is drawn in the $\mathsf{nodes}_{\delta}(C_{s+1})$-avoiding disk of $\trace_{\rho}(C_i)$.
\end{enumerate}

Observe that $i \neq 0$ since this would imply that the corresponding path in $\Qcal'$ is $\rho$-exposed.

Notice that $C_i \cup L$ contains a unique cycle $C'$ different from $C_i$ whose trace separates $\Delta^*$ from the nodes corresponding to $V(\Omega)$.
Moreover, the disk $\Delta''$ bounded by $\mathsf{trace}_{\rho}(C')$ which contains $\Delta^{*}$ is properly contained in the disk bounded by $\mathsf{trace}_{\rho}(C_{i})$ which contains $\Delta^{*}$.
In particular, there exists an edge of $C_i$ which does not belong to $C_i',$ and there exists an edge that belongs to $C_i'$ but not to $C_i$ and therefore, the graph drawn on $\Delta''$ after deleting the vertices of $C_i'$ is properly contained in the graph drawn on $\Delta'$ after deleting the vertices of $C_i.$

\smallskip
\textbf{Step 2:} We would now like to define the new nest $\mathcal{C}'$ by replacing the edited cycle with $C'$.
First, we are going to change the cycle $C'$ once more.
Consider $C''$ to be the cycle obtained from $C'$ after iteratively replacing a subpath of $C',$ in the same way as above, with any $C'$-subpath of a path in $\mathcal{R}$ that satisfies the same four properties as $L$ above.
This step guarantees that the drawing of any $C''$-subpath of a path in $\Rcal$ does not intersect the \emph{interior} of $C''$, i.e., the interior of the disk bounded by $\mathsf{trace}_{\rho}(C'')$ that contains $\Delta^{*}$.
Call this property *.

The next and final step is to edit the paths in the radial linkage $\Rcal$ in order to obtain a new radial linkage $\Rcal'$ that is orthogonal to our new nest $\Ccal''$ obtained by replacing $C'$ with $C''.$
First, observe that the paths in $\Rcal$ may not be orthogonal to $\Ccal''$ only because they are not orthogonal to $C''$.
Now, for each path $R \in \Rcal,$ let $R'$ be any minimal $C''$-subpath of $R$ and $C''_{1}$ be the subpath of $C''$ with the same endpoints as $R'$ such that $R' \cup C''_{1}$ is the unique cycle whose trace in $\delta$ bounds a disk $\Delta_{R'} \subseteq \Delta$ that does not fully contain neither all nodes associated to $V(\Omega)$ nor $\Delta^{*}$.
By property *, $\Delta_{R'}$ avoids the interior of $C''$.
Moreover, $\Delta_{R'}$ cannot intersect the drawing of any other path in $\mathcal{R}$.
Indeed, this would imply the existence of a subpath of a path in $\mathcal{R}$ satisfying properties 1) to 4) which cannot exist.
Therefore, we may safely update $R$ by replacing $R'$ by $C''_{1}$.
By repeating this procedure until we no longer we find any such subpath $R'$, we conclude with a radial linkage $\mathcal{R}'$ orthogonal to $\mathcal{C''}$.

By updating the walloid $W$ to a new walloid $W'$ accordingly, we obtain our desired $\chi$-respectful $\rho$-pruning of $(G, \delta, W)$.
\end{proof}

Notice that whenever we find a pruning, at least one edge of $G$ is being ``pushed'' out of a cycle of a nest of some vortex segment of $(G, \delta, W)$.
As we shall demonstrate, the total number of vortex segments throughout each step as well as the size of the each nest are upper bounded by some constant depending only on the input parameters.
This means that after finding at most $\mathcal{O}(|E(G)|)$ many large enough transactions on the vortex segment societies we have either found a large exposed transaction or may conclude that all vortex segment societies are of bounded depth.

Note that, in the end, we want to guarantee that not only each vortex segment society is of bounded depth, but within every vortex segment is a vortex which is also of bounded depth, and this vortex is surrounded by a large nest.
The tools we develop in the upcoming subsections will allow us to achieve just that.

\subsection{Orthogonal transactions}

So far we know that we can always either find a pruning of a schema, or an exposed planar transaction that is proper with respect to a compatible bounded breadth and depth rendition of a vortex segment society.
An important ingredient that we need to develop the tools in the upcoming subsections is that, in the latter case above, we should also be able to make the exposed planar transaction orthogonally, by possibly slightly altering the nest that comes with the corresponding vortex segment.

\paragraph{Coterminal and orthogonal radial linkages.}

We first require a series of tools that will allow us to orthogonalize a radial linkage with respect to some nest in a society under different assumptions.

\begin{lemma}\label{lem:orthogonal_radial_linkage} For all integers $s, r \geq 1$ the following holds.
Let $\rho$ be a rendition of a society $(G, \Omega)$ in the disk $\Delta$ with a nest $\mathcal{C}$ of order $s$ and a radial linkage $\mathcal{R}$ of order $r$ for $\mathcal{C}$ such that there is no cycle $C \in \mathcal{C}$ and no $C$-subpath $R'$ of a path $R \in \mathcal{R}$ that is
\begin{itemize}
\item disjoint from all other cycles of $\mathcal{C}$ and
\item fully drawn in $\Delta$ minus the interior of the disk bounded by the trace of $C$ in $\rho.$
\end{itemize}
Then there exists a radial linkage $\mathcal{R}'$ of order $r$ for $\mathcal{C}$ that is orthogonal to $\mathcal{C}$ and with the same endpoints as $\mathcal{R}.$

Moreover, there exists an algorithm that finds $\mathcal{R}'$ in time $\mathbf{poly}(s + r) \cdot |E(G)|.$
\end{lemma}
\begin{proof} Given a cycle $C \in \mathcal{C},$ let us call the \emph{exterior} of $C,$ the set $\Delta$ minus the interior of the disk bounded by the trace of $C$ in $\rho.$

Let $C_{1}, \ldots, C_{s}$ be an ordering of $\mathcal{C}$ from innermost to outermost.
It follows by assumption that, any $C_{i}$-subpath $R'$ of any path $R \in \mathcal{R}$ for some $i \in [s-1]$ must either intersect $C_{i + 1}$ or be internally disjoint from the exterior of $C_{i}.$
In the case that $R'$ is a $C_{s}$-path it must either intersect $V(\Omega)$ or be internally disjoint from the exterior of $C_{s}.$
From now on we call such a subpath $R'$ of a path $R \in \mathcal{R}$ a \emph{regression} of $R$.

Let the paths in $\mathcal{R}=\{ R_1,R_2,\ldots,R_r\}$ be ordered according to the occurrence of their endpoints on $\Omega$.
Moreover, assume there exists an $i \in [r]$ and a $j \in [2, s]$ such that $R_i$ has a regression $R'$ that is a $C_{j}$-path.
Notice that there exists a subpath $L_{R'}$ of $C'_j$ that shares its endpoints with $R'$ which is internally disjoint from $R_i.$
This is because otherwise, we could find a subpath of $R_i$ which witnesses that we are in the previous case.
Let $\Delta_{R'}$ be the disk bounded by the trace of $R'\cup L_{R'}$ which avoids the trace of $C_1$ in $\rho.$
We assume that $R'$ is chosen maximally with the property of being a regression.

Suppose there exists some path $Q,$ possibly $Q=R,$ in $\mathcal{R}$ whose drawing intersects the interior of $\Delta_{R'}$ in a node or arc distinct from those of $R'.$
Then, this intersection of $Q$ with $\Delta_{R'}$ must belong to a maximal regression $Q'$ of $Q$ and there must exist $j' \in [s],$ $j' > j,$ such that $Q'$ is a $C_{j'}$-path.
Hence, there must exist some maximal regression $P'$ of some path $P' \in \mathcal{R}$ such that no other part of $\mathcal{R}$ intersects the corresponding disk $\Delta_P'.$
Thus, by replacing $P'$ with $L_{P'},$ we obtain a new radial linkage that is end-coterminal with $\Rcal$ but with strictly less regressions.
This means that, after at most $E(\mathcal{R})$ many such steps, we must have found a radial linkage $\mathcal{R}'$ which is orthogonal to $\mathcal{C}'.$
This can be performed in time $\poly(s + t) \cdot |E(G)|$.
\end{proof}

Recall that every vortex segment comes with a nest $\Ccal$ and a radial linkage $\Rcal$ where $\mathcal{R}$ is orthogonal to $\mathcal{C}.$
Moreover, $\mathcal{R}$ acts as the connection of the vortex segment to the annulus wall of the walloid of our schema.
While refining, and possibly splitting the vortex segment, we wish to maintain the existence of such a radial linkage and we need to make sure that new flap segments that potentially need to be extracted from the vortex segment can be routed through the paths in $\mathcal{R}$.
To this end we prove the following lemma which allows us to root any given orthogonal transaction on our orthogonal radial linkage.

\medskip
Let $(G, \Omega)$ be a society and $\rho$ be a rendition of $(G, \Omega)$ in the disk with a nest.
Moreover, let $\Lcal$ be a radial $X$-$Y$-linkage and $\Rcal$ be a radial $X'$-$Y'$ linkage in $(G, \Omega)$ such that both are of the same order.
Then, assuming that $Y \subseteq V(\Omega)$ and $Y' \subseteq V(\Omega)$, we say that a radial $X$-$Y'$-linkage in $(G, \Omega)$ is \emph{start-coterminal} with $\Lcal$ and \emph{end-coterminal} with $\Rcal$.

\begin{lemma}\label{lem:coterminal_radial_linkage}
Let $s \in \Nbbb_{\geq 1}$, $(G, \Omega)$ be a society, and $\rho$ be a rendition of $(G, \Omega)$ in the disk with a nest $\Ccal$ of order $s$.
Moreover, let $\Lcal$ and $\Rcal$ be two radial linkages such that both are of order $s$ and orthogonal to $\Ccal$.
Then there exists a radial linkage $\Qcal$ of order $s$ that is orthogonal to $\Ccal$, start-coterminal with $\Lcal$, and end-coterminal with $\Rcal$.

Moreover, there exists an algorithm that finds the outcome above in time $\poly(s) \cdot |E(G)|$.
\end{lemma}
\begin{proof} Notice that for every set $S \subseteq V(G)$ such that $|S| < s$, there exists at least one cycle, one path in $\Lcal$, and one path in $\Rcal$ fully contained in $G - S$.
Since $\Lcal$ and $\Rcal$ are orthogonal to $\Ccal$, the union of these three elements is connected.
This implies that by Menger's theorem, there must exist a radial linkage $\mathcal{Q}$ in $G$ of order $s$ that is start-coterminal with $\Lcal$ and end-coterminal with $\Rcal$.
Moreover, since $\mathcal{L}$ and $\mathcal{R}$ are orthogonal to $\mathcal{C}$, we can immediately deduce that there is no $C$-subpath of any path in $\mathcal{Q}$ for any cycle $C \in \mathcal{C}$ that is disjoint from all other cycles of $\mathcal{C}$ and fully drawn in $\Delta$ minus the interior of the disk bounded by the trace of $C$ in $\rho.$
This allows us to apply \cref{lem:orthogonal_radial_linkage} and conclude with a radial linkage $\mathcal{Q}'$ that is orthogonal to $\mathcal{C}$ and with the same endpoints $\mathcal{Q}.$
This completes the proof.
\end{proof}

We also require the following lemma from~\cite{GorskyPW2026Quickly} that allows to orthogonalize a radial linkage with respect to some nest in a society in a respectful manner.
The proof of this lemma is very similar to that of \cref{lem:orthogonal_radial_linkage}, but in favor of brevity we choose to not reprove it.

\begin{proposition}[Lemma 4.11, \cite{GorskyPW2026Quickly}]\label{prop:orthogonal_radial_linkage} For all integers $s, r \geq 1$ the following holds.
Let $\rho$ be a rendition of a society $(G, \Omega)$ with a nest $\mathcal{C}$ of order $s$ and a radial linkage $\mathcal{R}$ of order $r$ for $\mathcal{C}$.
Then there exists a nest $\mathcal{C}'$ of order $s$ and a radial linkage $\mathcal{R}'$ of order $r$ for $\mathcal{C}'$ that is orthogonal to $\mathcal{C}'$ and with the same endpoints on $V(\Omega)$ as $\mathcal{R}$.
Further, the disk bounded by the trace of the outer cycle of $\mathcal{C}'$ in $\rho$ is contained in the disk bounded by the trace of the outer cycle of $\mathcal{C}$ in $\rho$.

Moreover, there exists an algorithm that finds $\mathcal{C}'$ and $\mathcal{R}'$ in time $\mathbf{poly}(s + r) \cdot |E(G)|.$
\end{proposition}

We now proceed with the proof of our orthogonalization tool.
Note that a preliminary version of this result already appeared in~\cite{PaulPTW2024Delineating,PaulPTW2024Obstructions}.

\begin{lemma}\label{lem:orthogonal_transaction}
Let $t \in \Nbbb_{\geq 4}$, $s \in \Nbbb_{\geq 1}$, $p \in \Nbbb_{\geq 2}$, and $(G, \delta, W)$ be a $(\text{-}, t, s, \text{-}, \text{-})$-$\Sigma$-schema.
Moreover, let $(H, \Omega)$ be a vortex segment society of $(G, \delta, W)$ corresponding to a vortex segment with nest $\{ C_{0}, \ldots, C_{s + 1} \}$ and radial linkage $\Rcal$ that has a $\delta$-compatible rendition $\rho$ in $\Delta_{H}$.
Let $\Pcal$ be a $\rho$-exposed transaction of order $(s + 1)(p + 2)$ in $(H, \Omega)$ that is $\rho$-proper.
Then there exists
\begin{itemize}
\item a nest $\Ccal$ in $\rho$ of order $s + 1$ around the $\mathsf{nodes}_{\delta}(C_{s + 1})$-avoiding disk of $\trace_{\rho}(C_{0})$,
\item a radial linkage of order $t$ that is orthogonal to $\mathcal{C}$ and end-coterminal with $\mathcal{R}$, and
\item a $\rho$-exposed transaction $\mathcal{Q}$ of order $p$ in $(H, \Omega)$ that is orthogonal to $\mathcal{C}$ and $\rho$-proper.
\end{itemize}
Moreover, there exists an algorithm that finds the outcome above in time $\poly(s + p + t) \cdot |E(G)|$.
\end{lemma}
\begin{proof}
Let $\rho = (\Gamma, \Dcal)$ and $\Ccal = \{ C_{0}, \ldots, C_{s} \}$.
First note that by definition of $\rho$-proper, the paths in $\Pcal$ never cross one another in $\Gamma$ and this is crucial for the arguments that follow.

Let $q \coloneqq (s + 1)(p + 2)$.
Let $\Lambda_{\Pcal}$ be an ordering of the paths in $\Pcal$ such that while traversing $\Omega$ we first encounter one of the two endpoints for all paths in $\Pcal$ and then all others.
Let $P_{1}, \ldots, P_{q}$ be the paths in $\Pcal$ respecting the ordering $\Lambda_{\Pcal}.$
Moreover, for every cycle $C \in \mathcal{C}$ we define the \emph{interior} (resp. \emph{exterior}) of $C$ as the disk bounded by the trace of $C$ that contains (resp. avoids) the nodes associated to $V(\Omega)$.

Now, for $i \in [0, s]$ and $j \in [q],$ let $Β$ be a non-trivial $V(P_{j})$-$V(P_{j})$-path that is a subpath of $C_{i}.$
Let $P_{B}$ be the subpath of $P_{i}$ that shares its endpoints with $B.$
Let $\Delta_{B}$ be the disk bounded by the trace of the cycle $B \cup P_{B}$ that does not fully contain either the trace of $C_{0}$ or the nodes associated to $V(\Omega)$.
We call $B$ a \emph{bend of $C_{i}$ at $P_{j}$} if $\Delta_{B}$ is non-empty (meaning that $P_{B}$ differs from $B$ in at least one edge) and $\Delta_{B}$ does not intersect the drawing of $C_{i} \setminus B$.
Moreover, we call a bend $B$ \emph{shrinking} if $\Delta_{B}$ is a subset of the interior of $C_{i}$ and \emph{expanding} otherwise.
We call an expanding bend $B$ \emph{loose} if $\Delta_{B}$ does not intersect the drawing of any cycle in $\Ccal$ that is not $C_{i}$ and \emph{tight} if any expanding bend $B'$ of $C_{i}$ where $B'$ is a subpath of $B$ is not loose.

We also define the \emph{height} of a bend $B$ as the non-negative integer $r + 1,$ where $r$ is the maximum value such that $B$ intersects at least one of $P_{j + r}$ or $P_{j - r}$ and $j + r \leq s$ or $j - r \geq 1$ (if $j + r$ exceeds $q$ we consider $P_{j + r} = P_{q}$ and if $j - r$ drops below $1$ we similarly assume that $P_{j - r} = P_{1}$).
Finally, we call any minimal $V(P_{1})$-$V(P_{q})$ path contained in a cycle $C \in \Ccal$ a \emph{pillar} of $C.$

First, notice that if $B$ is any expanding bend of a cycle $C \in \Ccal,$ then the graph obtained by replacing the subpath $B$ of $C$ with the path $P_{B}$ defines a cycle $C'$ and the interior of $C'$ strictly contains the interior of $C,$ hence the name expanding.
In the case of a shrinking bend the opposite happens, particularly we obtain a cycle with strictly smaller interior.
Additionally, notice that for every $C \in \Ccal,$ since $\Pcal$ is a $\rho$-exposed transaction, $C$ contains at least one pillar and since $C$ is a cycle, it must contain an even number of pairwise disjoint pillars.

We first prove the following claim by induction on the number of cycles in $\mathcal{C}$.

\begin{claim} If no cycle in $\Ccal$ contains a loose expanding bend then for every $i \in [0, s],$ then
\begin{enumerate}
\item every expanding bend of $C_{i}$ has height at most $s - i + 1,$
\item $C_{i}$ contains exactly two disjoint pillars $R_{i},$ $i \in [2],$ and
\item every shrinking bend of $C_{i}$ that is a subpath of $R_{i},$ $i \in [2],$ has height at most $s - i + 1.$
\end{enumerate}
\end{claim}
\begin{cproof}
Let $k \in [s]$ and $\Ccal^{k}$ be the subset of $\Ccal$ that consists of its $k$ outermost cycles.
Assume inductively that no cycle in $\mathcal{C}^{k}$ contains a loose expanding bend and as a result the three properties above hold for $\mathcal{C}^{k}$.
We prove that under the same assumption that no cycle in $\mathcal{C}^{k + 1}$ contains a loose expanding bend, then the three properties above also hold for $\Ccal^{k + 1}.$

Assume that no cycle in $\Ccal^{k + 1}$ contains a loose expanding bend.
Let $B$ be any tight expanding bend of $C_{s - k}$ at $P \in \Pcal$ which is the innermost cycle contained in $\Ccal^{k+1}$ and the only cycle of $\Ccal^{k+1}$ not contained in $\Ccal^{k}.$
Since $B$ is tight, the disk $\Delta_{B}$ intersects the drawing of some other cycle of $\mathcal{C}$.
In fact, since $B$ is expanding and therefore $\Delta_{B}$ is not a subset of the interior of $C_{s-k}$, it follows that $\Delta_{B}$ intersects the drawing of cycles only in $\mathcal{C}^{k}$.
Now, let $C' \in \mathcal{C}^{k}$ such that $\Delta_{B}$ intersects the drawing of $C'$.
Consider the maximal $V(P)$-$V(P)$-subpath of $C'$, say $P'$ whose endpoints are drawn in $\Delta_{B}$.
Note that $P'$ may not be a bend of $C'$ at $P$.
Even worse it may be that $P'$ is not fully drawn in $\Delta_{B}$.
However, we can find a set $\{ B'_{1}, \ldots, B'_{l} \}$ of internally disjoint $V(P)$-$V(P)$-subpaths of $P'$ such that any subpath of $P'$ drawn in $\Delta_{B}$ is a subpath of one of the $B'_{i}$'s and each $B'_{i}$ is an expanding bend of $C'$ at $P$.
Indeed this set can be defined as the set of maximal subpaths of $P'$ that define an expanding bend of $C'$ at $P$.
Moreover, applying the inductive hypothesis, implies that each $B'_{i}$ has height at most $k$.

Next, notice that any subpath $B'$ of $B$ whose endpoints are the endpoints of a maximal subpath of a path in $\Pcal$ that is drawn in $\Delta_{B}$ is also an expanding bend of $C_{s - k}$ where $\Delta_{B'} \subseteq \Delta_{B}.$
Since $B$ is tight, any such subpath $B'$ must also be tight.
This implies that the disk $\Delta_{B'}$ for any such subpath $B'$ intersects at least one of the $B'_{i}$'s defined analogously for a cycle $C'$ whose drawing intersects $\Delta_{B'}$.

Now, assume towards contradiction that the height of $B$ is at least $k + 2$.
This implies that there exists a $V(P'')$-$V(P'')$- subpath $B''$ of $B$ that is an expanding bend of $C$ at $P''$, where the index of $P''$ differs by at least $k + 1$ from that of $P$.
However, no $B'_{i}$ defined as before can intersect the disk $\Delta_{B''}$ as the height of $B'_{i}$ is at most $k$.
Therefore, $B''$ must be loose.
This contradicts our assumptions and we conclude that the height of $B'$ is at most $k + 1$.

Moreover, it is straightforward to observe that $C_{s - k}$ contains exactly two disjoint pillars.
Indeed, if there were at least four disjoint pillars in $C_{s - k}$ then it is easy to see that there would exist an expanding bend of $C_{s - k}$ on either $P_{1}$ and/or $P_{q}$ whose height would be larger than $k + 1,$ which would imply that it is a loose expanding bend, which by assumption cannot exist.

Finally, it is also straightforward to observe that the existence of a shrinking bend of $C_{s - k}$ that is a subpath of either of the two pillars of $C_{s - k}$ of height more than $k + 1,$ implies the existence of an expanding bend of the same height, which as we proved cannot exist.
\end{cproof}

We proceed with the description of an algorithm that runs in time $\poly(s + p + t) \cdot |E(G)|^{2}$ that computes the desired nest $\mathcal{C}'$ along with the desired railed linkage $\mathcal{R}'$ that is orthogonal to $\mathcal{C}',$ and moreover computes the desired planar $\rho$-exposed transaction $\mathcal{Q}$ of order $p$ that is orthogonal to the new nest $\mathcal{C}'.$

\smallskip
\textbf{Step 1:} We obtain our new set of cycles $\mathcal{C}' \coloneqq \langle C'_{0}, \ldots, C'_{s} \rangle,$ where $C'_{i},$ $i \in [0, s],$ is obtained from $C_{i}$ by iteratively applying whenever applicable the following two rules.

\begin{itemize}

\item[\textbf{1:}] As long as there exists a path $R \in \mathcal{R},$ some $i \in [0, s]$ and a subpath $R' \subseteq R$ such that $R'$ is a $C_{i}$-path and is otherwise disjoint from all cycles of $\mathcal{C}$ and moreover $R'$ is drawn entirely in the exterior of $C_{i},$ we update $C_{i}$ to the unique cycle $C'$ contained in $C_{i} \cup R'$ different from $C_{i}$ whose trace separates the interior of $C_{0}$ from the nodes associated to $V(\Omega)$.
We call such a subpath $R'$ of $R$ as above, a $C_{i}$-\emph{expanding} path.
Now observe that $C'$ remains $\rho$-grounded.
Moreover, the interior of $C'$ properly contains the interior of $C_{i}$ (before the update).

\item[\textbf{2:}] As long as there exists a loose expanding bend $B$ of a cycle $C \in \mathcal{C}$ we replace $B$ by $P_{B}$ as defined above.
Now observe that the updated cycle remains $\rho$-grounded.
\end{itemize}

Since for both rules above, whenever we update a cycle of our nest, it is guaranteed that the interior of the updated cycle strictly contains the interior of its predecessor, the above procedure will terminate in at most $|E(G)|$ many steps and the resulting nest $\mathcal{C}'$ will satisfy the following two properties.
There is no loose expanding bend of any cycle in $\mathcal{C}',$ which implies that $\mathcal{C}'$ satisfies the three properties above, and there is no $C$-expanding subpath of any path in $\mathcal{R}$ for any cycle $C \in \mathcal{C}'.$

\smallskip
\textbf{Step 2:} Now, update $\mathcal{R}$ to contain the minimal $V(\Omega)$-$\pi_{\rho}(\mathsf{nodes}_{\rho}(C'_{0}))$ subpath of each path in $\mathcal{R}$.
By our previous step, no path in $\mathcal{R}$ contains a $C'_{i}$-expanding path.
This permits us to apply \cref{lem:orthogonal_radial_linkage} to $\mathcal{C}'$ and $\mathcal{R}$ and obtain a radial linkage for $\mathcal{C}'$ that is orthogonal to $\mathcal{C}'$ as desired, in time $\mathsf{poly}(s + t) \cdot |E(G)|^{2}.$

\smallskip
\textbf{Step 3:} We now define a transaction $\Pcal' \subseteq \Pcal$ of order $p$ by skipping the first and last $s + 1$ paths in $\Pcal$ and then choosing the first path from every bundle of $s + 1$ many consecutive paths from the remaining paths of $\Pcal.$
By the arguments above the height of every expanding bend of every cycle in $\Ccal'$ at $P_{1}$ or $P_{s}$ is at most $s + 1$ and therefore by definition no such bend can intersect any of the paths in $\Pcal'$.
It follows that any other bend of any cycle must be a bend that is a subpath of a pillar of the given cycle.
By definition of $\Pcal'$, this implies that any such bend of any cycle can intersect at most one path in $\Pcal'$.
This implies that the transaction $\Pcal'$ which is a planar and $\rho$-exposed transaction in $(H, \Omega)$ is almost orthogonal to our new set of cycles $\Ccal'$.

We need to make one more adjustment to each of the paths in $\Pcal'$ to finally make them orthogonal to $\Ccal'$.
We define the final transaction $\Qcal$ where each path $Q \in \Qcal$ is a path obtained from a path $P \in \Pcal'$ that shares the same endpoints with $P$ and is defined by starting from one of the two endpoints of $P$ and while moving towards the other endpoint of $P$ taking any possible ``shortcut'', that is by greedily following along any bend of any pillar of a cycle in $\Ccal'$ at $P$ that appears, until we reach the other endpoint of $P$.
Since as we already observed any bend of any pillar of a cycle can intersect at most one path in $\Pcal'$, it is implied that $\Qcal$ is a linkage.
Moreover, by construction, it now follows that $\Qcal$ is orthogonal to $\Ccal'$ as desired and that the paths in $\Qcal$ are $\rho$-grounded which also implies that $\Qcal$ is $\rho$-proper.
Finally, since the shortcuts we take are subpaths of pillars it follows that $\Qcal$ contains at least one edge drawn in the interior of $C'_{0}$ and therefore is $\rho$-exposed.
\end{proof}

\subsection{Splitting and homogeneous transactions}

With the tools developed in the previous sections, we know that we can always either find a pruning of our schema, or a planar transaction in some vortex segment society that is proper with respect to a rendition of the society that, moreover, is compatible with the decomposition of our schema.
On top of this, by \cref{lem:orthogonal_transaction}, we may also assume that this transaction is orthogonal to the nest of our vortex segment.
The next step is to apply the ideas of \cref{sec_homogeneous_walloid} to the transaction to make sure that the area covered by it is homogeneous with respect to a cell-coloring of our decomposition which would then allow us to extract additional flap segments from it.

\medskip
Let $(G, \Omega)$ be a society and $\rho$ be a rendition of $(G, \Omega)$ in the disk with a nest $\Ccal = \{ C_{1}, \ldots, C_{s} \}$ for some $s \in \Nbbb_{\geq 1}$.
We say that a transaction in $(G, \Omega)$ is \emph{$\rho$-splitting} if it is $\rho$-proper and orthogonal to $\Ccal$.

Let $\mathcal{P}$ be a $\rho$-splitting transaction in $(G, \Omega)$.
A \emph{parcel} of $\mathcal{P}$ is a cycle $C$ consisting of two vertex-disjoint subpaths of $C_{1}$, say $X_1$ and $X_2,$ and a shortest $\pi_{\rho}(\mathsf{nodes}_{\rho}(C_1))$-$\pi_{\rho}(\mathsf{nodes}_{\rho}(C_1))$ subpath from two distinct paths $P_{1} \neq P_{2} \in \mathcal{P}$, such that $C_1$ contains an $X_1$-$X_2$-subpath of $P_i$ for each $i \in [2]$, and $C_1$ intersects no other member of $\mathcal{P}$.
Notice that all but the boundary paths of a $\rho$-splitting transaction belong to exactly two parcels and the boundary paths belong to exactly one parcel each.

Let $\ell \in \Nbbb_{\geq 1}$ and $\chi$ be a cell-coloring of $\rho$ of capacity at most $\ell$.
The \emph{$\chi$-coloring of a parcel $O$ of $\mathcal{P}$ in $\rho$} is the set $\{ \chi(c) \mid c \in \rho\text{-}\mathsf{influence}(O) \}$.
A transaction $\mathcal{P}$ is said to be \emph{$\chi$-homogeneous in $\rho$} if it is $\rho$-splitting and all of its parcels have the same $\chi$-coloring in $\rho$.

\begin{lemma}\label{lem:homogeneous_transaction}
Let $\ell \in \Nbbb_{\geq 1}$, $p \in \Nbbb_{\geq 2}$, and $(G, \delta, W)$ be a $\Sigma$-schema.
Moreover, let $(H, \Omega)$ be a vortex segment society of $(G, \delta, W)$ that has a $\delta$-compatible rendition $\rho$ in $\Delta_{H}$ and $\chi$ be a cell-coloring of $\rho$ of capacity at most $\ell$.
There exists a function $f_{\ref{lem:homogeneous_transaction}} \colon \Nbbb^{2} \to \Nbbb$ such that for every $\rho$-splitting transaction $\mathcal{P}$ of order $f_{\ref{lem:homogeneous_transaction}}(p, \ell)$ in $(H, \Omega)$ there exists a transaction $\Qcal \subseteq \Pcal$ of order $p$ that is $\chi$-homogeneous in $\rho$.

Moreover, $f_{\ref{lem:homogeneous_transaction}}(p,d)\in p^{2^{\ell}}.$
\end{lemma}
\begin{proof}
Recall that $f_{\ref{lem_homogeneity_2d}}(z_1, z_2)= z_1^{2^{z_1}}.$
We set $f_{\ref{lem:homogeneous_transaction}}(d,p) \coloneqq f_{\ref{lem_homogeneity_2d}}(p, \ell)$ and thus meet the upper bound claimed in the assertion.

Let $\{ C_{0}, \ldots, C_{s+1} \}$ be the nest of the vortex segment of $W$ corresponding to $(H, \Omega)$.
Let $P_1$ and $P_2$ be the two boundary paths of $\mathcal{P}$ and for each $i\in[2]$ let $P_i'$ be the shortest $\pi_{\rho}(\mathsf{nodes}_{\rho}(C_1))$-$\pi_{\rho}(\mathsf{nodes}_{\rho}(C_1))$-subpath of $P_i$ with at least one edge not in $C_1$.
Then there exist two vertex-disjoint subpaths $L_1$ and $L_2$ of $C_1$ such that each $L_i$ has one endpoint in common with $P_1'$ and the other endpoint in common with $P_2',$ and every path $P\in\mathcal{P}$ contains a unique $V(L_1)$-$V(L_2)$-subpath $P'.$
It follows that $J\coloneqq L_1\cup L_2\cup \bigcup_{P\in\mathcal{P}}$ is an $f_{\ref{lem:homogeneous_transaction}}(p, \ell)$-ladder.
Moreover, the bricks of $J$ are in a one-to-one correspondence with the parcels of $\mathcal{P}.$
Now consider the brick-coloring $\chi_{J}$ of $J$, where for every brick $B$ of $J$, $\chi_{J}(B)$ is defined as the $\chi$-coloring of $B$ in $\rho$.
This allows us to call \cref{cor_homogeneity_1d} to obtain a $p$-ladder $J' \subseteq J$ such that for every two bricks $B_{1}$ and $B_{2}$ of $J'$, $\chi_{J'}(B_{1}) = \chi_{J'}(B_{2})$, where $\chi_{J'}$ is the brick-coloring of $J'$ induced by $\chi_{J}$.
As a result, all bricks of $J'$ have the same $\chi$-coloring in $\rho$, since for every brick $B$ of $J'$, $\chi_{J'}(B)$ is identical to the $\chi$-coloring of $B$ in $\rho$.
Moreover, notice that each of the $p$ horizontal paths of $J'$ is, in fact, a subpath $P'$ of some path in $\mathcal{P}.$
Let $\mathcal{Q}$ be the collection of all paths $Q\in\mathcal{P}$ such that $Q'$ is a horizontal path of $J'.$
It follows that $\mathcal{Q}$ is $\chi$-homogeneous in $\rho$ and of order $p$ as desired.
\end{proof}

\subsection{Storage segments}

We are now able to always find large homogeneous transactions.
Next, we need an auxiliary tool that will help us extract extract additional flap segments when processing the vortices within vortex segments.
Toward this goal we define an intermediate structure.

\medskip
Let $p \geq t \geq 3$ be integers.
A \emph{$(t \times p)$-storage segment} is a $((2t + 1) \times p)$-wall segment $W$ that is the union of a $(t \times p)$-wall segment $W_1$, called the \emph{base}, whose $t$ horizontal paths are exactly the first $t$ horizontal paths of $W$, a $((t + 1) \times p)$-wall $W_2$ whose $t$ horizontal paths are exactly the horizontal paths $t+1$ up to $2t+1$ of $W$, and a $(2 \times p)$ wall segment $Q$ whose two horizontal paths are exactly the horizontal paths $t$ and $t+1$ of $W$.
We call the bricks of $Q$ the \emph{plots} of $W$.

Let $G$ be a graph, $\delta$ be a $\Sigma$-decomposition of $G$, $\chi$ be a cell-coloring of $\delta$, and $W$ be a $(t \times p)$-storage segment that is a subgraph of $G$.

The \emph{$\chi$-coloring of a plot $O$ of $W$ in $\delta$} is the set $\{ \chi(c) \mid c \in \delta\text{-}\mathsf{influence}(O) \}$.
We say that $W$ is \emph{$\chi$-homogeneous in $\delta$} if the $\chi$-coloring of all plots of $W$ is the same.

We will use the storage segments as a storage unit for simple cells we have found deep within a segment and which will be turned into flap segments in a second step.
For this reason, storage segments only appear as an intermediate structure within this subsection.

\begin{figure}[h]
\centering
\includegraphics{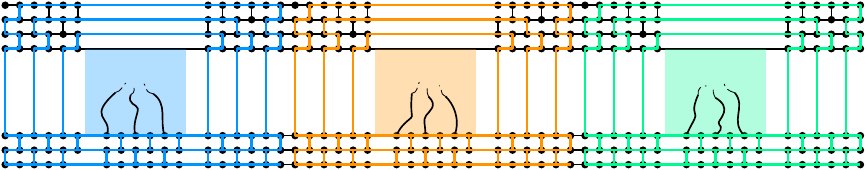}
\caption{\label{fig:extract_storage}A storage segment divided into a sequence of flap segments.}
\end{figure}

\begin{lemma} \label{lem:storage_to_flaps}
Let $t \in \Nbbb_{\geq 3}$, $p, b \in \Nbbb_{\geq 1}$, $G$ be a graph, $\delta$ be a $\Sigma$-decomposition of $G$, $\chi$ be a cell-coloring of $\delta$, $W$ be a $((t + 12) \times p)$-storage segment in $G$ that is $\chi$-homogeneous in $\delta$, and let $X$ be the perimeter of the wall obtained after removing all degree one vertices of $W$.

Assume further that $p \geq \ell^{*} b \cdot (2t + 12)$ where $\ell^*$ is the number of colors in the $\chi$-coloring of the plots of $W.$
Then there exists a sequence $W_1, \ldots, W_{\ell^{*}b}$ of $(t \times t)$-flap segments such that
\begin{itemize}
    \item for every $i \in [\ell^{*}b]$, $W_i$ is contained in $W$ union the graph drawn in the closure of the union all cells in the $\delta$-influence of $X$, the base of $W_i$ is a subgraph of the base of $W$, and its horizontal paths coincide with the horizontal paths of $W$, and
    \item for every color $\alpha$ in the $\chi$-coloring of the plots of $W$, there exists $i \in [d^{*}]$ such that for every $j \in [b]$, $W_{(i - 1)b + j}$ represents $\alpha$ in $\delta$.
\end{itemize}
\end{lemma}
\begin{proof}
Let $\{ \alpha_1, \ldots, \alpha_{\ell^*}\}$ be the collection of all colors that belong to the $\chi$-coloring of the plots of $W.$
Let us begin by partitioning $W$ into a sequence $\langle W'_1, \ldots, W'_{\ell^*} \rangle$ of $((t + 12) \times b(2t + 12))$-wall segments.
Fix some $i \in [\ell^*]$ and now partition $W'_i$ further into a sequence $\langle W^i_1, \ldots, W^i_b \rangle$ of $((t + 12) \times (2t + 12))$-wall segments.
It suffices to show that each $W^i_j$ can be turned into a $t$-flap segment that represents color $\alpha_{i}$.

Notice that each $W^i_j$ individually can be seen as a $((t + 12) \times (2t + 12))$-storage segment.
Since $W$ is $\chi$-homogeneous in $\delta$, it follows that $W^i_j$ is so as well since every plot of $W^i_j$ is a plot of $W.$
We may now find a linkage $\mathcal{P}$ starting at the first vertex of each of the first $t$ vertical paths of $W^i_j,$ then following these vertical paths into the $(t+2)$-nd horizontal path, from there heading to the last $t$ vertical paths of $W^i_j$ and following these down to the bottom horizontal path.
See \cref{fig:extract_storage} for an illustration.

This linkage intersects only $2t$ many plots of $W^i_j$ and will serve as the rainbow of our flap segment.
To finish creating the flap segment observe that there are $12$ remaining vertical paths in $W^{i}_{j}$.
We may now use the same arguments as in the proof of \cref{lem:representation_step} and \cref{lem:paths_to_bottom} to pick a plot of $W^{i}_{j}$ which is the center of a $(12 \times 12)$-subwall of $W^i_j$ that does not intersect $\mathcal{P}$.
Within the influence of this plot, we know we can find a cell $c_{i}$ such that $\chi(c_{i}) = \alpha_{i}$.
Now \cref{lem:paths_to_bottom} allows us to link the boundary vertices of $c_{i}$ to the columns of $W^{i}_{j}$ we are currently handling.
Again, see \cref{fig:extract_storage} for an illustration.
This completes our construction. 
\end{proof}

\subsection{Swapping a vortex and a flap segment}

A last issue we have to deal with before we can jump into the proof of our main tool is the following.
Suppose that the vortex segment society of, at least, the first vortex segment is already of bounded depth.
This means that there is nothing more to be extracted from that vortex segment, and we have to start treating some vortex segment to the right of it.
However, whenever we extract a flap segment by first pulling out a part of some homogeneous transaction in a vortex segment society of the $i$-th vortex segment and then using \cref{lem:storage_to_flaps}, the flap segment will be ``in-between'' vortex segments with respect to the linear ordering of the segments.
As we demand all flap segments of a schema to occur in order without being interrupted by other segments, we need a way to move this flap segment to the ``left'' of all vortex segments.
This will cost us a bit of the infrastructure of our vortex segments, but since we will only extract, and then move, a bounded number of flap segments in the end, this becomes just a matter of large enough numbers.

\begin{lemma} \label{lem:moving_flaps}
Let $r \in \Nbbb$, $t, t_{0} \in \Nbbb_{\geq 3}$ and $s, s_{0}, \ell, b \in \Nbbb_{\geq 1}$ where
$$s \geq s_0 + \ell (2t_0 + 3) \coloneqq d_{1} \ \text{ and } \ t \geq t_0 + 2\ell (2t_0 + 3)\coloneqq d_{2}.$$

Now, let $(G, \delta, W)$ be obtained from an $(r, t, s, \text{-}, \text{-})$-$\Sigma$-schema $(G, \delta^{*}, W^{*})$ by replacing the $i$-th $((r + t) \times t, s)$-vortex segment of $W$ with a sequence $\langle W_1, \ldots, W_{\ell}, W_{\ell+1}, W_{\ell + 2} \rangle$, where $W_i$ is an $((r + t) \times t_{0})$-flap segment, for every $i \in [\ell]$, and both $W_{\ell+1}$ and $W_{\ell+2}$ are $((r + t) \times d_{2}, d_{1})$-vortex segments where we allow either $W_{\ell+2},$ or both $W_{\ell+1}$ and $W_{\ell+2}$ to be empty.
Then there exists an $(r, t_{0}, s_{0}, \text{-}, \text{-})$-$\Sigma$-schema $(G, \delta, W')$ such that the number of vortex segments of $(G, \delta^{*}, W^{*})$ is exactly the number of vortex segments of $(G, \delta, W)$ minus one plus the number of non-empty members of $\{ W_{\ell+1}, W_{\ell+2} \}.$

Moreover, let $\chi$ be a cell-coloring of $\delta$ and assume that for every color $\alpha \in \chi\text{-}\mathsf{col}(\delta)$, either $W^{*}$ $b$-represents $\alpha$ in $\delta^{*}$, or $W_{i}$ represents $\alpha$ in $\delta$ where $i$ belongs to an interval of $\ell$ of length $b$.
Then $W'$ $b$-represents $\chi$ in $\delta$.
\end{lemma}

The proof of \cref{lem:moving_flaps} is mostly technical but not very difficult.
We therefore choose to present a rather informal description of the operation instead of a full formalization.
The core part of the necessary re-routing is depicted in \cref{fig_swap_vortex_flap}.

\begin{figure}[h]
\centering
\includegraphics{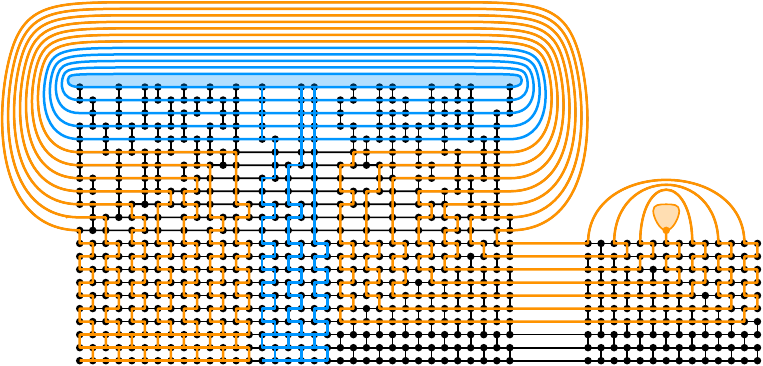}
\caption{\label{fig_swap_vortex_flap}``Moving'' a flap segment from the ``right'' of a vortex segment to its ``left'' in the proof of \cref{lem:moving_flaps}.}
\end{figure}

\begin{proof}

Let $C$ denote the exceptional cycle $C^{\mathsf{ex}}(W)$ of $W.$
For each $i \in [\ell]$, $W_i$ is an $((r + t) \times t_{0})$-flap segment.
Moreover, for each $i \in [\ell]$, the hyperedge of $W_i$ has at most $3$ vertices in common with $C$ while the rainbow of $W_i$ meets $2t_0$ additional vertices of $C$.
Let us call these vertices the \emph{terminals} of $W_i$.
To prove out statement it suffices to show that the flap segments may be ``moved'' sequentially from one side of a single vortex segment to the other.
If this is true, then, by using the same arguments, we can move the flap segment to ``the other side'' of any number of consecutive vortex segments since we only have to sacrifice the infrastructure locally.
Moreover, if this works for a single flap segment, then by induction on $\ell$ it follows that we may repeat the procedure for each of them.

We allocate the following part of the infrastructure to make our routing:
Starting from $C$ we select $2t_0 + 3$ many consecutive cycles from the base cycles of $W.$
From the $(i-1)$-st vortex segment we allocate the first $2t_0 + 3$ paths from its radial linkage together with the last $2t_0 + 3$ paths of its radial linkage.
Moreover, we reserve the outermost $2t_0 + 3$ cycles of its nest.

Let $C'$ denote the new exceptional cycle of $W'$ in case $W'$ is obtained from $W$ by forgetting the first $2t_0 + 3$ consecutive cycles starting from $C.$

Observe that each part of the infrastructure allocated is large enough to host $2t_0 + 3$ disjoint paths, which is at least the number of terminals of our flap segment.
By following along the vertical paths of $W$ and the allocated infrastructure, we can now find a linkage from the terminals of our flap segment to the intersection of the first $2t_0 + 3$ paths of the radial linkage of the $(i-1)$-st vortex segment such that these paths stay within the allocated infrastructure except for their endpoints, and the endpoint of each path belongs to $C'$.
See \cref{fig:extract_storage} for an illustration of how those paths are to be chosen.

Notice that this operation reduces the infrastructure by at most $2t_0 + 3$ many cycles in the vortex segment and $W$ and at most $2(2t_0 + 3)$ many paths from its radial linkage.
Moreover, the operation results, after forgetting the used infrastructure and adjusting the numbers accordingly, in the flap segment to now be on ``the other side'' of the $(i-1)$-st vortex segment.

Since we perform this operation for $\ell$ flap segments in total, the total amount of lost infrastructure is at most $\ell (2t_0 + 3)$ cycles from the nest of each involved vortex segment, at most $2\ell (2t_0 + 3)$ paths from the radial linkage of each involved vortex segment, and at most $\ell (2t_0 + 3)$ cycles of $W$.
By our choice of numbers this implies that the base annulus of $W$ itself has at least $r + t_{0}$ cycles left.
Moreover, each vortex segment has at least $t_0$ paths in its radial linkage that remain untouched and at least $s_0$ cycles in its nest that have not been used.
Thus, our claim follows.
\end{proof}

\subsection{Splitting and extracting}

We are now ready to deal with our most technical tool that will facilitate most of the technicalities of the proof of the coarsening lemma, that is \cref{lem:coarsening} in \cref{sec_coarsening_lemma}.

\medskip
Let $(G, \delta, W)$ be a $\Sigma$-schema.
Moreover, let $(H, \Omega)$ be a vortex segment society of $(G, \delta, W)$ with $\delta$-canonical rendition $\delta \cap \Delta_{H}$ around $c_{0}$ and $\rho$ be a $\delta$-compatible rendition of $(H, \Omega)$ in $\Delta_{H}$.
Also, let $\Pcal$ be a $\rho$-splitting transaction in $(H, \Omega)$.

By combining $\delta \cap \Delta_{H}$ with the strip rendition of $\rho$ and $\Pcal$, we may obtain a $\delta$-compatible rendition $\rho^{*}$ of $(H, \Omega)$ such that $\rho^{*} \sqsubseteq \rho$ and $c_{0}$ is divided into two vortices $c^{*}_{1}$ and $c^{*}_{2}$ in $\rho^{*}$.
We call $\rho^{*}$ the \emph{residual rendition} of $\rho$ and $\Pcal$ and these two vortices the \emph{residual vortices} of $\rho^{*}$.
Moreover, for both $i \in [2]$, $\rho^{*}_{i} \coloneqq \rho \cap \Delta_{c^{*}_{i}}$ is a rendition of the vortex society of $c^{*}_{i}$ in $\Delta_{c^{*}_{i}}$ and $\rho$ can be obtained by combining $\rho^{*},$ $\rho^{*}_{1},$ and $\rho^{*}_{2},$ after removing $c^{*}_{1}$ and $c^{*}_{2}.$

\begin{lemma}\label{lem:splitting_extracting}
Let $r \in \Nbbb$, $t_{0} \in \Nbbb_{\geq 4}$, $s_{0}, x, \ell, b \in \Nbbb_{\geq 1}$, $y \in \Nbbb_{\geq 2}$, and let
\begin{align*}
d &\coloneqq 2 \ell b (2t_{0} + 12),\\
t &\coloneqq 8d + 2t_{0} + 2,\\
s &\coloneqq s_{0} + 2(t_{0} + 14) + 12d + 4t_{0} + 2\text{, and}\\
p &\coloneqq 3t + 2s + 2.
\end{align*}

Let $(G, \delta, W)$ be an $(r, t, s, \text{-}, \text{-})$-$\Sigma$-schema, $(H, \Omega)$ be a vortex segment society of $(G, \delta, W)$, $\rho$ be a $\delta$-compatible rendition of $(H, \Omega)$ in $\Delta_{H}$ with breadth at most $x$ and depth at most $y$, $\chi$ be a cell-coloring of $\delta_{H, \rho}$, and $\Pcal$ be a transaction in $(H, \Omega)$ of order $p$ that is $\chi$-homogeneous in $\rho$.

Let $\delta^{*}$ be the $\Sigma$-decomposition of $G$ obtained by combining $\delta$ and the residual rendition $\rho^{*}$ of $\rho$ and $\Pcal$, and let $c^{*}_{1}, c^{*}_{2}$ be the residual vortices of $\rho^{*}$.

Now, assume that either $\rho$ has a vortex, or that the set $\chi$-$\mathsf{col}(\rho) \setminus \chi$-$\mathsf{col}(\delta)$ is non-empty.
Let $\Ocal$ be the maximal subset of $\chi$-$\mathsf{col}(\delta_{H, \rho}) \setminus \chi$-$\mathsf{col}(\delta)$ that is a subset of the $\chi$-coloring of the parcels of $\Pcal$, and assume $|\Ocal| \leq \ell$.
Then,
\begin{enumerate}
    \item if either
    \begin{itemize}
        \item $\Ocal$ is non-empty, or
        \item $\Ocal$ is empty but for each $i \in [2]$, either $\rho^{*}_{i}$ has a vortex, or there is a simple cell $c_{i}$ of $\rho$ contained in $c^{*}_{i}$ such that $\chi(c_{i}) \notin \chi$-$\mathsf{col}(\delta)$,
    \end{itemize}
    there exists an $(r, t_{0}, s_{0}, \text{-}, \text{-})$-$\Sigma$-schema $(G, \delta', W')$ controlled by $\Tcal_{W}$ such that
    \begin{enumerate}
    \item every cell of $\delta'$ that is not a cell of $\delta^{*}$ is a vortex,
    \item every vortex of $\delta'$ is a vortex of $\delta$ except for $c'_{1}$ and $c'_{2}$ which contain $c^{*}_{1}$ and $c^{*}_{2}$ respectively, and
    \item there exist non-negative integers $x_{1}$ and $x_{2}$ such that $x_{1} + x_{2} = x$ and for each $i \in [2],$ $\rho \cap \Delta_{c'_{i}}$ is a rendition of the vortex society of $c'_{i}$ in $\Delta_{c'_{i}}$ with breadth at most $x_{i}$ and depth at most $y.$
    \end{enumerate}
    Moreover, if $W$ $b$-represents $\chi$-$\mathsf{col}(\delta)$ in $\delta$, then $W'$ $b$-represents $\chi$-$\mathsf{col}(\delta) \cup \Ocal$ in $\delta'$.
    \item Otherwise, there exists a $\chi$-respectful $\rho$-pruning of $(G, \delta, W)$.
\end{enumerate}
Moreover, there exists an algorithm that finds one of the two outcomes above in time $\poly(t_{0} + s_{0} + \ell + b) \cdot |E(G)||V(G)|$.
\end{lemma}
\begin{proof}

We break the proof into five steps as follows.
During the proof we keep one bit of information called the \emph{flag} which will determine which of the two outcomes of the lemma we end up in.
In case the bit is $1,$ i.e., the \emph{flag is set}, we find the first outcome, otherwise, we will end up with a pruning of $(G, \delta, W).$
In the beginning, our bit equals $0$ and the flag is not set.
\begin{description}

    \item[Step 1:] This step is to establish the general setup.
    We divide $\mathcal{P}$ into five transactions (see \cref{fig:partition_transaction}) which will be used for different parts of the construction later on.
    
    \item[Step 2:] We then check if $\mathcal{O}$ is non-empty.
    If this is true we set the flag and introduce an intermediate setup that will allow us to extract a large storage segment that will later be used to $b$-represent the colors in $\Ocal$.
    Otherwise, the flag remains unset and we move on.
    
    \item[Step 3:] For each of the two residual vortices $c^{*}_{1}$ and $c^{*}_{2}$ of $\rho^{*}$, we test if either \textsl{a)} $\rho^{*}_{i}$ has a vortex, or \textsl{b)} there is a simple cell $c_{i}$ of $\rho$ contained in $c^{*}_{i}$ such that $\chi(c_{i}) \notin \chi$-$\mathsf{col}(\delta)$.
    If for each $c^{*}_{i}$, $i \in [2]$, one of \textsl{a)} or \textsl{b)} holds and the flag is unset, we set the flag.
    Otherwise, we do not touch the flag and move on.
    
    \item[Step 4:] If the flag has been set in one of the previous steps we now extract a storage segment, which will be converted in a sequence of flaps by \cref{lem:storage_to_flaps} each $b$-representing a different color of $\Ocal$, and two new vortex segments from the original vortex segment.
    This yields the first outcome of the assertion and our case distinction terminates here.
    
    \item[Step 5:] This step is only reached if the flag is not set at any point.
    By our assumptions on $(H, \Omega)$ this implies that exactly one of $c^{*}_{i}$, $i \in [2]$, satisfies \textsl{a)} or \textsl{b)} from \textbf{Step 3} and the $\chi$-coloring of the parcels of $\mathcal{P}$ is a subset of $\chi$-$\mathsf{col}(\delta)$.
    In this case, we make use of $\mathcal{P}$ to find a $\chi$-respectful $\rho$-pruning of $(G, \delta, W).$
    
\end{description}

\begin{figure}[h]
\centering
\includegraphics{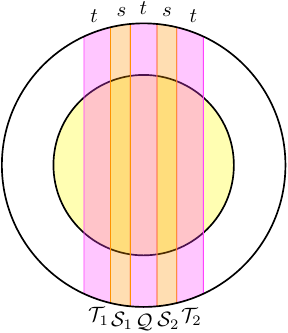}
\caption{\label{fig:partition_transaction}The partition of the transaction $\mathcal{P}$ into five smaller transactions in the proof of \cref{lem:splitting_extracting}.}
\end{figure}

We now dive directly into discussing the details of each step.

\smallskip
\textbf{Step 1: Setup.} Let $\{ C_{0}, \ldots, C_{s+1} \}$ be the nest and $\Rcal$ be the radial linkage of the vortex segment of $W$ corresponding to $(H, \Omega)$.
Let $\Ccal = \{ C_{1}, \ldots, C_{s} \}$.
Let $I_1$ and $I_2$ be the two minimal segments of $\Omega$ such that $\mathcal{P}$ is a $V(I_1)$-$V(I_2)$-linkage.
Let $p_1, p_2, \ldots, p_{p}$ be the endpoints of the paths in $\mathcal{P}$ in $V(I_1)$ ordered according to their appearance in $\Omega,$ and $p'_1, p'_2, \ldots, p'_p$ be the endpoints of the paths in $\mathcal{P}$ in $V(I_2)$.
We may number $\mathcal{P} = \{ P_1, P_2, \ldots, P_p \}$ such that $P_i$ has endpoints $p_i$ and $p'_i$ for each $i \in [p]$.
With this, we are ready to define the five transactions that provide the framework for all following construction steps.
\begin{align*}
\mathcal{T}_1 & \coloneqq \{ P_2, \ldots, P_{t+1}\}\\
\mathcal{S}_1 & \coloneqq \{ P_{t+2}, \ldots, P_{t+s+1}\}\\
\mathcal{Q} &\coloneqq \{ P_{t+s+2}, \ldots, P_{2t+s+1} \}\\
\mathcal{S}_2 & \coloneqq \{ P_{2t+s+2}, \ldots, P_{2t+2s+1}\}\\
\mathcal{T}_2 & \coloneqq \{ P_{2t+2s+2}, \ldots, P_{3t+2s+1} \}
\end{align*}
Notice that, with $\mathcal{P}$ being $\chi$-homogeneous, each of the five transactions above is also $\chi$-homogeneous and any member of the $\chi$-coloring of a parcel of one of these transactions belongs to the $\chi$-coloring of every parcel of any of these five transactions.
See \cref{fig:partition_transaction} for an illustration.

We also define two disks $\Delta_1$ and $\Delta_2$.
Let $U_1$ be the cycle consisting of a subpath of $C_{2(t_{0} + 14) + s_{0} + 2d + 2}$  together with a subpath of $P_{t+s+1}$ whose $\mathsf{nodes}_{\delta^{*}}(C_s)$-avoiding disk also avoids the paths in $\mathcal{Q}$.
Similarly, let $U_2$ be the cycle consisting of a subpath of $C_{2(t_{0} + 14) + s_{0} + 2d + 2}$ together with a subpath of $P_{2t+s+2}$ whose $\mathsf{nodes}_{\delta^{*}}(C_s)$-avoiding disk also avoids the paths in $\mathcal{Q}$.
Then, for each $i \in [2]$, we define $\Delta_i$ to be the $\mathsf{nodes}_{\delta^{*}}(C_s)$-avoiding disk bounded by the trace of $U_i$.
The two disks $\Delta_i$ are our candidates for the new vortex segments societies.

Finally, notice that for each $i \in [2]$, $c^{*}_{i}$ is fully contained in $\Delta_{i}$, and since $\rho$ has breadth at most $x$ and depth at most $y$, there exist non-negative integers $x_{1}$ and $x_{2}$ such that $x_{1} + x_{2} = x$ and the rendition $\rho^{*}_{i}$ of the vortex society of $c^{*}_{i}$ has breadth at most $x_{i}$ and depth at most $y$.
Moreover, let $c_{0}$ be the single vortex cell of the $\delta$-canonical rendition of $(H, \Omega)$.

\smallskip
\textbf{Step 2: Preparing the storage segment.} In case $\mathcal{O}$ is empty we completely skip this step and move on to \textbf{Step 3}.

So for the following let us assume that $\mathcal{O}$ is non-empty.
In this case, we \textbf{set the flag}.
Moreover, we construct a certain part of the infrastructure which can easily be turned into a large storage segment later.
To this end, we start by selecting a linkage $\mathcal{Q}'$ as follows.
For each $Q \in \mathcal{Q}$ let $Q'$ be the shortest subpath of $Q$ which starts in $V(I_1)$, contains an edge which is completely drawn inside $c_0$, and ends on $C_{2(t_{0} + 14)}$.
Then $\mathcal{Q}' \coloneqq \{ Q' \mid Q \in \mathcal{Q} \}.$
Notice that, with respect to $\mathcal{C}' \coloneqq \{ C_{2(t_{0} + 14) + 1}, \ldots C_{s} \}$, $\mathcal{Q}'$ is an orthogonal radial linkage.
Moreover, the drawing of any path in $\mathcal{Q}'$ is disjoint from $\Delta_1\cup \Delta_2$.

Now, we may select two paths $T_{2(t_{0} + 14)}$ and $T'_{2(t_{0} + 14)}$ as follows:
Let $q^+_1$ and $q^+_2$ be the endpoint of $P_{t+s+2}'$ and $P_{2t+s+1}'$ on $C_{2(t_{0} + 14)}$ respectively.
Then $C_{2(t_{0} + 14)}$ contains a unique $q^+_1$-$q^+_2$-subpath which contains all endpoints of $\mathcal{Q}'$ on $C_{2(t_{0} + 14)}$.
Let $T_{2(t_{0} + 14)}$ be this subpath.
Similarly, let $q^-_1$ and $q^-_2$ be the first vertex of $P_{t+s+2}'$ and $P_{2t+s+1}'$ on $C_{2(t_{0} + 14)}$ respectively which is encountered when traversing these paths starting in $V(I_1)$.
Then $C_{2(t_{0} + 14)}$ contains a unique $q_1^{-}$-$q_2^{-}$-subpath which does not contain any endpoint of $\mathcal{Q}'$ on $C_{2(t_{0} + 14)}$.
Let be this subpath.
Now notice that there exists a $\delta$-aligned disk $\Delta'$ which intersects the drawing of $G$ exactly in the nodes corresponding to the set  $$V(T_{2(t_{0} + 14)}) \cup V(P_{t+s+2}')\cup V(P_{2t+s+1}')\cup V(T'_{2(t_{0} + 14)}).$$
Let $G'$ be the subgraph of $G$ drawn in the disk $\Delta'.$

Then, for every $j \in [2(t_{0} + 14)]$, $C_j \cap G'$ contains two disjoint $V(P_{t+s+2}')$-$V(P_{2t+s+1}')$-subpaths, say $T_{j}$ and $T_{j}'$.
Now, observe that $\cupall \mathcal{Q} \cup\{ T_j \cup T'_{j} \mid j \in [2(t_{0} + 14)] \}$ is a system of $4(t_{0} + 12)$ horizontal paths and $t$ vertical paths which are pairwise orthogonal which, since $t \geq 8d$, can be seen to contain as a subgraph a $((2(t_{0} + 12) + 1) \times 2d)$-wall segment $W_Q$ together with a linkage $\mathcal{Q}''$ of order $2d$ such that each path in $\mathcal{Q}''$ is the continuation of a vertical path of $W_{Q}$ to $V(I_{1})$ and contains an edge which is drawn in the interior of $c_{0}$.
Moreover, $W_{Q}$ is disjoint from any cycle in $\Ccal'$.
Finally, notice that for every plot $X$ of $\mathcal{Q}$ there exists a $\delta$-aligned disk whose boundary intersects the drawing of $\rho$ only in vertices of $V(I_1)$, the bottom horizontal path of $W_Q$, and two consecutive members of $\mathcal{Q}''$ such that $X$ is contained in this disk.

With this our preparations are complete.
See \cref{fig:split_vortex_segment} for an illustration of some of the objects defined above.
Notice that, if we made $\mathcal{Q}''$ end-coterminal with $\mathcal{R}$ within the partial nest $\mathcal{C}'$ we would immediately obtain a $((t_{0} + 12) \times 2d)$-storage segment.
Moreover, by construction $W_{Q}$ is $\chi$-homogeneous in $\delta^{*}$ and the $\chi$-coloring of the plots of $\Qcal$ contain $\Ocal$.

\smallskip
\textbf{Step 3: Preparing the ``split''.} For each $i \in [2],$ let $\Ocal_{i} \coloneqq \chi$-$\mathsf{col}(\rho_{i}) \setminus \chi$-$\mathsf{col}(\delta)$.
We say that $c^{*}_{i}$ is \emph{promising} if $\rho^{*}_{i}$ has a vortex or $\mathcal{O}_{i}$ is non-empty.
If there is $i \in [2]$ such that $c^{*}_{i}$ is not promising, we move on to the next step.
Otherwise, $c^{*}_{i}$ is promising for both $i \in [2]$.
In this case, we \textbf{set the flag} (if it has not been set yet) and move on to the next step.

\begin{figure}[h]
\centering
\includegraphics{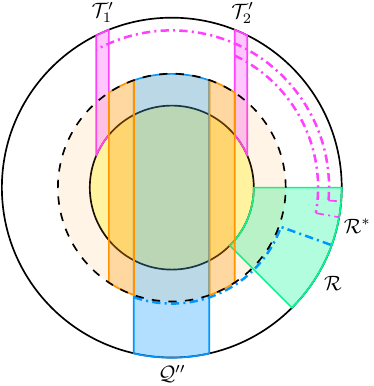}
\caption{\label{fig:split_vortex_segment}The ``splitting'' of a vortex segment into one large storage segment represented by the blue area and two smaller vortex segments whose nests are represented in orange and radial linkages in magenta in the proof of \cref{lem:splitting_extracting}.}
\end{figure}

\smallskip
\textbf{Step 4: The flag has been set.} If we reach this point and the flag has not been set so far, we directly move on to \textbf{Step 5}.
Otherwise, the flag has been set which means that $\mathcal{O}$ is non-empty and/or $c^{*}_{i}$ is promising for both $i \in [2]$.
Let $I \subseteq [2]$ be the maximal set such that $c^{*}_i$ is promising for every $i \in I$.

We now fix the following radial linkages which will be used to connect the newly constructed objects to a large portion of the original walloid $W$.
For every $Q\in\mathcal{Q}''$ as constructed in \textbf{Step 2} let $Q'$ be the shortest $V(I_1)$-$\pi_{\rho}(\mathsf{nodes}_{\rho}(C_{2(t_{0} + 14) + 1}))$-subpath of $Q$.
We set
\begin{align*}
    \mathcal{Q}''' \coloneqq \{ Q' \mid Q \in \mathcal{Q}'' \}.
\end{align*}
For every $T \in \mathcal{T}_1 \cup \mathcal{T}_2$ let $T'$ be the shortest $V(I_2)$-$\pi_{\rho}(\mathsf{nodes}_{\rho}(C_{2(t_{0} + 14) + 1}))$-subpath of $T$.
We then set
\begin{align*}
    \mathcal{T}_1' &\coloneqq \{ P_i' \mid i \in [2, 2t_{0} + 4d + 1] \}\text{ and}\\
    \mathcal{T}_2' &\coloneqq \{ P_i' \mid i \in [3t + 2s + 1 - 2{t_{0}} - 4d, 3t + 2s + 1] \}.
\end{align*}
In total this means that $\mathcal{R}' \coloneqq \mathcal{Q}''' \cup \mathcal{T}_1' \cup \mathcal{T}_2'$ is an orthogonal radial linkage of order $4t_0 + 10d$.
We now call upon \cref{lem:coterminal_radial_linkage} on the (partial) nest $\{ C_{s_{0} + 2(t_{0} + 14) + 2d + 3} , \ldots, C_{s} \}$ to obtain an orthogonal radial linkage $\mathcal{R}^*$ from $\mathcal{R}'$ such that $\mathcal{R}^*$ is start-coterminal with $\Rcal'$, end-coterminal with $\Rcal$, and $V(\mathcal{R}^*) \subseteq V(\mathcal{R}) \cup V(\mathcal{R}') \cup \bigcup_{i \in [s_{0} + 2(t_{0} + 14) + 2d + 3, s]} V(C_i)$.
See \cref{fig:split_vortex_segment} for an illustration.
For each $i \in I$, let $\mathcal{T}^{*}_{i}$ be the inclusion-maximum subset of $\Rcal^{*}$ that is start-coterminal with $T'_{i}$ and $\Qcal^{*} \coloneqq \Rcal^{*} \setminus (T^{*}_{1} \cup T^{*}_{2})$.

Next, we need to create new nests for each of the future vortex segments arising from $c^{*}_i$, $i \in I.$
We will define, for each $i \in [I],$ a cycle family $\mathcal{C}_i'$ of $s_{0} + 2d + 2$  many cycles ``around'' $c^{*}_{i}$.
Fix some point $x_1$ in the interior of $c^{*}_1$ which belongs to $c_0$ and is disjoint from the disk defined by the strip of $\mathcal{P}$.
For each $j \in [s_{0} + 2d + 2]$, let $C^1_{j}$ be the cycle contained in $C_{2(t_{0} + 14) + j} \cup P_{t + 1 + j}$ which is fully drawn within $\Delta_1$ and which separates $x_1$ from the trace of $C_s$.
Then
\begin{align*}
    \mathcal{C}_1'\coloneqq \{ C^1_j \mid j \in [s_{0} + 2d + 2] \}.
\end{align*}
Similarly, we fix some point $x_2$ in the interior of $c^{*}_2$ which belongs to $c_0$, is fully drawn within $\Delta_2$, and does not belong to the disk defined by the strip of $\mathcal{P}$.
Then, for each $j \in [s_{0} + 2d + 2]$, let $C^2_{j}$ be the cycle contained in $C_{2(t_{0} + 14) + j} \cup P_{2t+s+1+j}$ whose trace separates $x_2$ from the trace of $C_s$.
Then 
\begin{align*}
    \mathcal{C}_2' \coloneqq \{ C^2_j \mid j \in [s_{0} + 2d + 2] \}.
\end{align*}

Next, notice that by construction $\mathcal{T}^{*}_{i}$ is also orthogonal to $\mathcal{C}_i'$.
Similarly, we define the nest $$\mathcal{C}_i^* \coloneqq \{ C^i_2, \ldots, C^i_{s_0 + 2d +2}\}.$$
Finally, we select $c'_i$ to be the cell obtained from the disk bounded by the trace of $C^i_1$.
We know that for each $i \in [2] \setminus I$, the rendition $\rho^{*}_{i}$ is vortex-free and $\chi$-$\mathsf{col}(\rho'_{i}) \subseteq \chi$-$\mathsf{col}(\delta)$.
It follows that we can augment $\rho^{*}$ to a $\delta$-compatible rendition $\rho'$ of $(H, \Omega)$ in $\Delta_{H}$ where exactly $\{ c'_i \mid i \in I \}$ are the vortex cells of $\rho'$ and every other cell is a cell of $\rho$.
Moreover, since the renditions $\rho^{*}_{1}$ and $\rho^{*}_{2}$ have breadth at most $x_{1}$ and $x_{2}$ respectively, and both have depth at most $y,$ it follows from our construction that for each $i \in I$, $\rho \cap \Delta_{c'_{i}}$ is a rendition of the vortex society of $c'_{i}$ in $\Delta_{c'_{i}}$ with breadth at most $x_{i}$ and depth at most $y$ as required.

What remains is to describe the final step of the construction.
That is, we have to describe how to obtain the desired $(r, t_{0}, s_{0})$-$\Sigma$-schema.
The first thing to do is observe that by a simple planarity argument, in each of $T^{*}_{1}$, $T^{*}_{2}$ and $\Qcal^{*}$, there must be $t_{0} + 2d$, $t_{0} + 2d$, and $d$ members corresponding to disjoint minimal segments on $V(\Omega)$.
We keep these paths and the rest we discard.
Next, we extend these paths all they way down through the cycles of our walloid.
The infrastructure described above yields, an $((t_{0} + 12) \times d)$-storage segment and for each $i \in I$ a $((r + t) \times (t_{0} + 2d), (s_{0} + 2d))$-vortex segment.
Now, \cref{lem:storage_to_flaps} allows us to transform the storage segment into a sequence of $(t_{0} \times t_0)$-flap segments such that for each member of $\mathcal{O}$ there exist $b$ consecutive segments that represent that member.
Moreover, by extending their vertical paths all the way down through the cycles of our walloid these can be seen as $((r + t) \times t_{0})$-flap segments.
Together with the (up to two) new vortex segments we have now reached the situation of \cref{lem:moving_flaps}.
By applying \cref{lem:moving_flaps} at most two times and then forgetting about some of the excess infrastructure of the walloid $W$ and its segments, we finally obtain the desired $(r, t_{0}, s_{0}, \text{-}, \text{-})$-$\Sigma$-schema $(G, \delta', W')$, where $\delta'$ is obtained by combining $\delta$ with $\rho'$.
It follows, by construction,  that all necessary conditions are met.

\smallskip
\textbf{Step 5: Pruning.} All that is left is to discuss the case where the flag was not set.

In the case where neither $c^{*}_1$ nor $c^{*}_2$ is promising it follows that $\rho$ does not have a vortex and that the set $\chi$-$\mathsf{col}(\rho) \setminus \chi$-$\mathsf{col}(\delta)$ is empty.
By assumption this cannot happen and as we shall see later in this case the entire vortex segment can be safely forgotten to begin with and thus we may ignore this case.

The only remaining option is that exactly one of $c^{*}_1$ or $c^{*}_2$ is promising.
Without loss of generality we may assume that $c^{*}_2$ is not promising.

For every path $T$ in $\mathcal{T}_1$ let $T'$ be a shortest $V(I_2)$-$\pi_{\rho}(\mathsf{nodes}_{\rho}(C_0))$-subpath of $T$.
We set
\begin{align*}
    \mathcal{T} \coloneqq \{T' \mid T \in \mathcal{T}_1 \}.
\end{align*}
Then $\mathcal{T}$ is a radial linkage of order $t$ that is orthogonal to $\mathcal{C}$.
We then apply \cref{lem:coterminal_radial_linkage} and obtain an orthogonal radial linkage $\mathcal{L}$ that is start-coterminal with $\Tcal$ and end-coterminal with $\mathcal{R}$.

Next we define our new nest.
This construction is similar to the construction from \textbf{Step 4}.
We begin by fixing a point $x_1$ in the interior of $c^{*}_1 \cap c_0$ which avoids the disk defined by the strip of $\mathcal{P}$.
Now, for each $i \in [0, s+1],$ let $C_i^*$ be the cycle in $C_i \cup P_{t+2+i}$ whose trace in $\rho$ separates $x_1$ from the trace of $C_{s+1}$ in $\delta$.
Moreover, let $\mathcal{C}^* \coloneqq \{ C_i^* \mid i \in [0, s]\}.$

Now, by applying \cref{prop:orthogonal_radial_linkage} we may assume (up to possibly making the nest $\mathcal{C^{*}}$ even tighter) that $\mathcal{L}$ is orthogonal to $\mathcal{C}^{*}.$
Let $c_0^*$ be a cell obtained from the disk bounded by the trace in $\rho$ of $C_0^*$ which contains the point $x_1$.
It follows that $(H, \Omega)$ has a rendition $\rho'$ in the disk where $c^{*}_0$ is its only vortex cell, every other cell of $\rho'$ is a cell of $\rho$, and $\chi$-$\mathsf{col}(\rho') \subseteq \chi$-$\mathsf{col}(\delta)$.

With this we may augment $\delta$ to a $\Sigma$-decomposition $\delta'$ of $G$ by combining it with $\rho'$ and adjust $W$ by replacing the vortex segment with nest $\mathcal{C}$ and radial linkage $\mathcal{R}$ with a vortex segment with nest $\mathcal{C}^*$ and radial linkage $\mathcal{L},$ thereby obtaining a new walloid $W'.$
It is now straightforward to verify that $(G,\delta',W')$ is indeed a $\chi$-respectful $\rho$-pruning of $(G, \delta, W)$ since no path in $\mathcal{T}_2$ (except its first one), which is a $\rho$-exposed transaction in $(H, \Omega),$ belongs to any of the vortex segments of $(G, \delta', W')$.
With this, our proof is complete.
\end{proof}

\subsection{The coarsening lemma} \label{sec_coarsening_lemma}

Before we begin with the proof of \cref{lem:coarsening}, we fix three recursive functions $\mathsf{s},$ $\mathsf{t},$ and $\mathsf{p}.$
These functions represent the recursive application of \cref{lem:splitting_extracting} and ensure that the infrastructure of the final walloid has order $t^*,$ while its vortex segments are surrounded by nests of order $s^*$ for any choice of $t^*$ and $s^*.$

Let $t^{*} \in \Nbbb_{\geq 4}$ and $s^{*}, \ell, b, z \in \Nbbb_{\geq 1}$.
We define:
\begin{align*}
   \mathsf{t}(t^{*}, 0) \ &\coloneqq \ t^{*}\\
   \mathsf{t}(t^{*}, z) \ &\coloneqq \ (32 \ell b + 2)\mathsf{t}(t^{*}, z-1) + 192 \ell b + 2\text{, and}\\
   ~&~\\
   \mathsf{s}(s^*,t^*,0) \ &\coloneqq \ s^{*} + 6t^{*} + 30\\
   \mathsf{s}(s^*,t^*,z) \ &\coloneqq \ s^{*} + 1 + \mathsf{s}(s^{*}, t^{*}, z - 1) + (48 \ell b + 2)\mathsf{t}(t^{*}, z-1) + 288 \ell b + 30.
\end{align*}
From the recursive definitions of the functions above, one can derive the following observations on their order:
\begin{align*}
    \mathsf{t}(t^{*},z) \ &\in \ (\ell b + 1)^{\mathcal{O}(z)} \cdot t^{*}\text{ and}\\
    \mathsf{s}(s^{*},t^{*},z) \ &\in \ \Ocal(z \cdot s^{*} + z (\ell b + 1)^{\Ocal(z)} \cdot t^{*}).
\end{align*}

Moreover, whenever we start with a schema whose decomposition has breadth at most $x \in \Nbbb_{\geq 1}$ and depth at most $y \in \Nbbb_{\geq 2}$, the following upper bound $z^*$ on the maximum value $z$ may attain will follow from our proof:
\begin{align*}
    z^* \ \coloneqq \ x + \ell \cdot (\ell b + 1) \cdot 2^{\ell}
\end{align*}

We also fix a value $p^*$ which will act as an upper bound on the depth of the final schema as follows:
\begin{align*}
    p^* \ \coloneqq \ x^2y + (\ell b + 1)^{(x + b)2^{\Ocal(\ell)}} \cdot (s^{*} + t^{*})
\end{align*}

Based on these functions let us briefly provide explicit bounds for the three functions from the statement of \cref{lem:coarsening}.
\begin{align*}
f^1_{\ref{lem:coarsening}}(t^*, \ell, b, x) \coloneqq 2 \mathsf{t}(t^*, z^*) + \mathsf{s}(s^*,t^*,z^*) + 2 \ &\in \ \Ocal(x + b \cdot 2^{\Ocal(\ell)}) \cdot s^{*} + (\ell b + 1)^{\Ocal(x + b \cdot 2^{\Ocal(\ell)})} \cdot t^{*},\\
f^2_{\ref{lem:coarsening}}(\ell, b, x) &\coloneqq x + (\ell b + 1) \cdot 2^{\ell}\text{, and}\\
f^3_{\ref{lem:coarsening}}(t^*, \ell, b, x, y) \coloneqq p^{*} \ &\in \ x^2y + (\ell b + 1)^{(x + b)2^{\Ocal(\ell)}} \cdot (s^{*} + t^{*}).
\end{align*}

Please note that we do not derive these bounds explicitly in the proof below.
The bounds we provide will all be in terms of $z^*,$ $\mathsf{t},$ and $\mathsf{s},$ however, the bounds above then follow from the respective bounds on these functions.

With this, we are ready for the proof of \cref{lem:coarsening}.

\begin{lemma}[Coarsening lemma]\label{lem:coarsening}
There exist functions $f^{1}_{\ref{lem:coarsening}}$, $f^{3}_{\ref{lem:coarsening}} \colon \Nbbb^{6} \to \Nbbb$, and $f^{2}_{\ref{lem:coarsening}} \colon \Nbbb^{4} \to \Nbbb$ such that, for every $t^{*} \in \Nbbb_{\geq 3}$, every $s^{*}, \ell, b \in \Nbbb_{\geq 1}$, every $r, x, y \in \Nbbb$, every $(r, t' \coloneqq  f^{1}_{\ref{lem:coarsening}}(t^{*}, s^{*}, \ell, b, x), \text{-}, x, y)$-$\Sigma$-schema $(G, \delta, W)$ where $W$ is an $(r, t', \text{-})$-$\Sigma$-walloid with no vortex segments, and every cell-coloring $\chi$ of $\delta$ of capacity at most $\ell$ such that $W$ is $\chi$-homogeneous in $\delta$ and $b$-represents the $\chi$-coloring of every enclosure of $W$ in $\delta$, there is an $(r, t^{*}, s^{*}, f^{2}_{\ref{lem:coarsening}}(\ell, b, x), f^{3}_{\ref{lem:coarsening}}(t^{*}, s^{*}, \ell, b, x, y))$-$\Sigma$-schema $(G, \delta', W')$ controlled by $\Tcal_{W}$ such that
\begin{itemize}
\item $\delta' \sqsubseteq \delta$,
\item every vortex segment society of $(G, \delta', W')$ has depth at most $f^{3}_{\ref{lem:coarsening}}(t^{*}, s^{*}, \ell, b, x, y)$, and
\item $W'$ $b$-represents $\chi$ in $\delta'$.
\end{itemize}

Moreover, there exists an algorithm that finds the outcome above in time $\Ocal_{t^{*}, s^{*}, \ell, b, x, y}(|E(G)|^{2} |V(G)|)$.
\end{lemma}
\begin{proof}

We begin the proof with an application of \cref{obs:single_vortex_segment_schema}, thereby obtaining the implied $(r, \mathsf{t}(t^{*}, z^{*}), \mathsf{s}(s^{*}, t^{*}, z^{*}), 1, \text{-})$-$\Sigma$-schema $(G, \delta_{1}, W_{1})$.
Note that \cref{obs:single_vortex_segment_schema} implies that the single vortex segment society $(H, \Omega)$ of $(G, \delta_{1}, W_{1})$ has a $\delta$-compatible rendition $\rho_{H}^{\delta} \coloneqq \delta \cap \Delta_{H}$ in $\Delta_{H}$ with breadth at most $x$ and depth at most $y.$
For the remainder of this proof, we shall call $\rho_{H}^{\delta}$ the \emph{$\delta$-induced} rendition of $(H, \Omega)$.

We prove the following slightly stronger claim by induction on $z.$

\medskip
\noindent
\textbf{Inductive claim:} There exists an algorithm that, given an $(r, \mathsf{t}(t^{*}, z), \mathsf{s}(s^{*}, t^{*}, z), \text{-}, \text{-})$-$\Sigma$-schema $(G, \delta_{1}, W_{1})$ controlled by $\Tcal_{W}$ such that
\begin{itemize}
\item[a)] $\delta_{1} \sqsubseteq \delta,$
\item[b)] for every vortex segment society $(H_{1}, \Omega_{1})$ of $(G, \delta_{1}, W_{1}),$ $\rho^{\delta}_{H_{1}} \coloneqq \delta \cap \Delta_{H_{1}}$ is a $\delta$-compatible rendition of $(H_{1}, \Omega_{1})$ in $\Delta_{H_{1}}$ with breadth at most $x$ and depth at most $y,$ and 
\item[c)] $W_{1}$ is $\chi$-homogeneous in $\delta_{1}$ and $b$-represents the $\chi$-coloring of every enclosure of $W_{1}$ in $\delta_{1}$,
\end{itemize}
in time $\mathcal{O}_{t, s, \ell, b, x, y}(|E(G)|^{2} |V(G)|),$
outputs an $(r, t^{*}, s^{*}, z^{*}, p^{*})$-$\Sigma$-schema $(G, \delta_{2}, W_{2})$ controlled by $\Tcal_{W}$ such that
\begin{itemize}
\item[1)] $\delta_{2} \sqsubseteq \delta$ and
\item[2)] $W_{2}$ $b$-represents $\chi$ in $\delta_{2}.$
\end{itemize}

Here $z$ represents an upper bound on the number of times \cref{lem:splitting_extracting} can yield the first outcome.
The two main points here are that, since $\delta$ has breadth at most $x$, repeated applications of \cref{lem:splitting_extracting} can never produce strictly more than $x$ vortex segments whose vortex segment society admits a $\delta$-induced rendition that has a vortex.
Moreover, for any vortex segment with society $(H, \Omega)$ such that $\rho^{\delta}_{H}$ is vortex-free, the set $\chi$-$\mathsf{col}(\rho^{\delta}_{H})$ must contain a color that does not belong to $\chi$-$\mathsf{col}(\delta')$, where $\delta'$ is the decomposition of the input schema of the respective application of \cref{lem:splitting_extracting}, since otherwise we could simply discard the vortex segment as it does not offer anything in terms of representing $\chi$ in the target decomposition.

Under these principles, there are three situations under which an application of \cref{lem:splitting_extracting} on $(G, \delta', W')$ does not yield a pruning (which is a situation we will handle online).
\begin{enumerate}
    \item the outcome increases the number of vortex segments whose $\delta$-induced renditions have a vortex,
    \item the outcome produces a sequence of $b$ flap segments representing a color that did not belong to $\chi$-$\mathsf{col}(\delta')$, or
    \item the outcome increases the number of ``flat'' vortex segments whose $\delta$-induced renditions contain some simple cell whose color under $\chi$ does not belong to $\chi$-$\mathsf{col}(\delta')$.
\end{enumerate}

As discussed above, the first case can only occur $x$ times.
Moreover, the second case can occur at most $\ell$ times.
For the third case, we will see that whenever the total number of ``flat'' vortex segments goes above a certain threshold we are able to apply a pigeonhole argument to turn some of them into a sequence of $b$ flap segments representing some not yet represented color and thereby increasing $\chi$-$\mathsf{col}(\delta'')$, where $\delta''$ is the decomposition of the resulting schema of the respective application of \cref{lem:splitting_extracting}.
For this reason, also the third case can only occur a bounded number of times.

\smallskip
\textbf{Base of the induction:} For the base of the induction, we consider the case $z = 0.$
This case reflects the situation where $(G, \delta_{1}, W_{1})$ has $x$ vortex segments whose $\delta$-induced renditions have a vortex and where $\chi$-$\mathsf{col}(\delta_{1}) = \chi$-$\mathsf{col}(\delta)$.
In this situation, by the discussion above, there cannot exist a vortex segment whose $\delta$-induced rendition is vortex-free.
We proceed to show how $(G, \delta_{1}, W_{1})$ can be turned into a $(t^{*}, s^{*} + 6t^{*} + 30)$-$\Sigma$-schema satisfying the desired properties.

Let $(H, \Omega)$ be the vortex segment society of some vortex segment of $(G, \delta_{1}, W_{1})$.
Let $\mathcal{C} = \{ C_0, \ldots, C_{s^{*} + 6t^{*} + 30} \}$ be the nest of $(H, \Omega)$ and let $\Delta$ be the $\mathsf{nodes}_{\delta_{1}}(C_{s^{*} + 6t^{*} + 29})$-avoiding disk defined by the trace of $C_{6t^{*} + 30}.$
Moreover, let $(H', \Omega')$ be the society defined by $\Delta$ and consider the $\delta$-induced rendition $\rho^{\delta}_{H'}$ of $(H', \Omega')$ in $\Delta$.
Notice that $\mathcal{C}' \coloneqq \{ C_{0}, \ldots, C_{6t^{*} + 29} \}$ is a nest in $\rho^{\delta}_{H'}$ and by cropping the radial linkage of $(H, \Omega)$ we obtain a radial linkage $\mathcal{R}'$ of order $t^{*}$ for $(H', \Omega')$.
If $(H', \Omega')$ is of depth at most
\begin{align*}
    & 12t^{*} + 60 + (x + 1) \big( 2xy + (6t^{*} + 31)((15t^{*} + 60)^{2^{\ell}} + 2) \big) - 1\\
    =~ & 2s' + 2 + (x + 1) \big( 2xy + (s' + 1)((3t^{*} + 2s' + 2)^{2^{\ell}} + 2)  \big) - 1
\end{align*}
where $s' \coloneqq 6t^{*} + 30 = |{\mathcal{C}}'|,$ we are done.
Hence, we may assume that there exists a transaction $\mathcal{P}_0$ of order $12t^{*} + 60 + (x+1)\big( 2xy + (6t^{*} + 31)((15t^{*} + 60)^{2^{\ell}} + 2) \big)$ in $(H', \Omega')$.
By \cref{lem:exposed_transaction}, there is either a $\chi$-respectful $\rho^{\delta}_{H'}$-\textbf{pruning} of $(G, \delta_{1}, W_{1}),$ or there exists a transaction $\mathcal{P}_1 \subseteq \mathcal{P}_0$ of order $(x+1)\big( 2xy + (6t^{*} + 12)((15t^{*} + 60)^{2^{\ell}} + 2) \big)$ which is $\rho^{\delta}_{H'}$-exposed.
By \cref{lem:proper_transaction}, $\mathcal{P}_1$ contains a transaction $\mathcal{P}_2$ of order $(6t^{*} + 31)((15t^{*} + 60)^{2^{\ell}} + 2)$ that is $\rho^{\delta}_{H'}$-proper.
Then, by \cref{lem:orthogonal_transaction}, we can assume that (up to possibly slightly changing the cycles in $\mathcal{C}'$ and the paths in the radial linkage $\mathcal{R}'$) that there exists a transaction $\mathcal{P}'_{2}$ of order $(15t^{*} + 60)^{2^{\ell}}$ that is $\rho^{\delta}_{H'}$-splitting, which we can find in time $\Ocal(|E(G)|).$
Next we apply \cref{lem:homogeneous_transaction} to obtain a transaction $\mathcal{P}_3 \subseteq \mathcal{P}'_2$ of order $15t^{*} + 60 = 3t^{*} + 2s' + 2$ where $s' + 1 = |\mathcal{C}'|$ that is $\chi$-homogeneous in $\rho^{\delta}_{H'}$.
This finally allows us to use \cref{lem:splitting_extracting} on $(H', \Omega')$ and $\mathcal{P}_3.$
Notice that, since $(G, \delta_{1}, W_{1})$ has exactly $x$ vortex segments with a cross and $\chi$-$\mathsf{col}(\delta_{1}) = \chi$-$\mathsf{col}(\delta)$, the first outcome of \cref{lem:splitting_extracting} is impossible.
Thus, we must find a $\chi$-respectful $\rho^{\delta}_{H'}$-\textbf{pruning} of $(G, \delta_{1}, W_{1})$ in time $\mathcal{O}(|E(G)||V(G)|)$.

Notice that every time we find a pruning as the outcome of one of the above steps, at least one edge of $G$ is ``pushed'' further to the outside, or fully outside, of a vortex segment.
Since there are at most $|E(G)|$ many edges and only $6t^{*} + 12$ many cycles in $\mathcal{C}',$ we cannot find a pruning more than $(6t^{*} + 12) \cdot |E(G)|$ times.
As a result, after $\mathcal{O}(|E(G)|^2 |V(G)|)$ many iterations of the procedure above the only possible case left is that $(H', \Omega')$ is of depth at most
\begin{align*}
    12t^{*} + 60 + (x+1)\big( 2xy + (6t^{*}+31)((15t^{*}+60)^{2^{\ell}} + 2) \big)-1. 
\end{align*}
Since there are only $x$ vortex segments, these arguments establish the base of the induction.

\smallskip
\textbf{The inductive step:} The arguments that bounded the number of times we may encounter a pruning in the base of our induction still apply when we find a pruning in the inductive step.
In essence, we cannot encounter a pruning more than $\mathcal{O}(|E(G)|)$ times in total and thus, by iterating the arguments enough times, we always encounter an outcome that is not a pruning.
To avoid repetition, we continue to properly mark all places in the arguments were a pruning might occur, but we always assume that in these cases we encounter one of the non-pruning outcomes.
Formally, we can establish such a situation by assuming that $(G, \delta_{1}, W_{1})$ is a minimal counterexample to our inductive claim with respect to improvements via pruning.
Please note that, according to \cref{lem:exposed_transaction}, it is imperative that we perform all possible improvements via pruning, if we are to also guarantee in the end that not only the vortex societies within vortex segments have bounded depth but also their vortex segment societies.

From here on we assume that $z \geq 1$.

\smallskip
\textbf{Step 1: No useless vortex segments.} First, we may assume that for every vortex segment $\widehat{W}$ if the $\delta$-induced rendition of the vortex segment society of $\widehat{W}$ is vortex-free, then it contains a simple cell whose color under $\chi$ does not belong to $\chi$-$\mathsf{col}(\delta_{1})$.
We may assume this since any vortex segment that does not meet this criterion can safely be forgotten.

\smallskip
\textbf{Step 2: A bounded number of vortex segments.} Next, let us assume that there exist more than $(\ell b + 1) \cdot 2^{\ell}$ many vortex segments whose society has a vortex-free $\delta$-induced rendition.
Let $(H, \Omega)$ be the vortex segment society of a vortex segment of $(G, \delta_{1}, W_{1})$.
We define the \emph{$\chi$-coloring} of $(H, \Omega)$ as the set $\chi(C_{\mathsf{s}}(\rho^{\delta}_{H}))$.
Since there are at most $2^{\ell}$ many possible such $\chi$-colorings, there must be at least $\ell b + 1$ many vortex segments with the same $\chi$-coloring, say $\mathcal{F}'$.
Let $\mathcal{F} \coloneqq \mathcal{F}' \setminus \chi$-$\mathsf{col}(\delta_{1})$.
Notice that $|\mathcal{F}|\geq 1$ by the assumption from \textbf{Step 1}.
Moreover, $z \geq |\mathcal{F}|$ must hold since we at least allow one ``split'' for each unrepresented color.

Now, for each member $\alpha \in \mathcal{F}$ we wish to produce $b$ many $((r + \mathsf{t}(t^{*}, z-|\mathcal{F}|)) \times \mathsf{t}(t^{*}, z-|\mathcal{F}|))$-flap segments to represent $\alpha$.
To achieve this, the amount of infrastructure of our orchard we must sacrifice is at most $w \coloneqq |\mathcal{F}| \cdot b \cdot (2\mathsf{t}(t^{*}, z - |\mathcal{F}|) + 3)$ cycles from the nest and at most $2w$ paths from the radial linkage of every vortex segment.
By doing so, similar arguments as those for \cref{lem:moving_flaps} allow us to produce a new schema whose walloid is of order $\mathsf{t}(t^{*}, z - |\mathcal{F}|)$ and $b$-represents $\chi$-$\mathsf{col}(\delta_{1}) \cup \mathcal{F}$.
Since these arguments are essentially the same, we omit the discussion on how to move and place the extracted flap segments through the vortex segments.
Note that to achieve all this, for each such vortex segment society $(H, \Omega)$ we have to locate a cell $c \in C(\rho^{\delta}_{H})$ with $\chi(c) = \alpha$, and connect its boundary $\pi_{\rho^{\delta}_{H}}(\tilde{c})$ to distinct endpoints on $V(\Omega)$ of the radial linkage of $(H, \Omega)$.
This works similarly to the arguments in \cref{subsec:routing_cells} and we shall omit the discussion.
Notice that the following two inequalities follow directly from the recursive definitions of $\mathsf{s}$ and $\mathsf{t}$ for every pair $z \geq z'$ of positive integers.
\begin{align}
\mathsf{t}(t^{*}, z-z') + 2(z' \cdot b \cdot (\mathsf{t}(t^{*}, z - z') + 3)) \ &\leq \ \mathsf{t}(t^{*}, z)\\
\mathsf{s}(s^{*}, t^{*}, z - z') + z' \cdot b \cdot (\mathsf{t}(t^{*}, z-z') + 3) \ &\leq \ \mathsf{s}(s^{*}, t^{*}, z)
\end{align}
Therefore, if there are at least $x + (\ell b + 1)2^{\ell} + 1$ many vortex segments, under the assumption from \textbf{Step 1}, we may first extract $b \cdot |\mathcal{F}|$ flap segments to $b$-represent $|\mathcal{F}| \geq 1$ many unrepresented colors and then use the induction hypothesis to complete the proof.
Thus, from now on we may assume that there are at most $x + (\ell b + 1) \cdot 2^{\ell}$ many vortex segments in $(G, \delta_{1}, W_{1})$ as desired.

\smallskip
\textbf{Step 3: Splitting a vortex segment.} What follows is, essentially, a second iteration of the steps from the base of our induction.
The only two changes that occur are 1) the numbers are now bigger and depend directly on the functions $\mathsf{t}$ and $\mathsf{s}$ and 2) we may now also encounter the first outcome of \cref{lem:splitting_extracting} which will lead to a proper ``split'' of the vortex segment.

We begin by selecting $(H, \Omega)$ to be the vortex segment society of some vortex segment of $(G, \delta_{1}, W_{1})$ as before.
Let $$s'\coloneqq \mathsf{s}(s^{*}, t^{*}, z - 1) + 24 \ell b \cdot \big(2\mathsf{t}(t^{*}, z-1) + 12 \big) + 2(\mathsf{t}(t^{*}, z-1) + 14) + 2,$$ let $\mathcal{C} = \{ C_0, \ldots, C_{s^{*} + s' + 1} \}$ be the nest of $(H, \Omega)$, and let $\Delta$ be the $\mathsf{nodes}_{\delta_{1}}(C_{s^{*} + s' + 1})$-avoiding disk defined by the trace of $C_{s' + 1}$.
Moreover, let $(H', \Omega')$ be a society defined by $\Delta$ and consider the $\delta$-induced rendition $\rho^{\delta}_{H'}$ of $(H', \Omega')$ in $\Delta$.
Notice that $\mathcal{C}' \coloneqq \{ C_{0}, \ldots, C_{s'} \}$ is a nest in $\rho^{\delta}_{H'}$ and by cropping the radial linkage of $(H, \Omega)$ we obtain a radial linkage $\mathcal{R}'$ of order $\mathsf{t}(t^{*}, z)$ for $(H', \Omega')$.
If $(H', \Omega')$ is of depth at most
\begin{align*}
    &s''\coloneqq  2s' + 2 + (x+1)\big( 2xy + (s'+ 1)((3\mathsf{t}(t^{*},z)+2s' +2)^{2^{\ell}}+2) \big)-1 
\end{align*}

we may, by transforming $\Delta$ into a cell of $\delta_{1}$ and declaring it the hyperedge of our current vortex segment, obtain a vortex segment with a nest of order $s^{*}$ whose hyperedges corresponds to a vortex of depth at most $s''$.
In case this is the outcome for all vortex segments we are done.

Hence, we may assume that there exists a transaction $\mathcal{P}_0$ of order $2s' + 2 + (x+1)\big( 2xy + (s'+1)((3\mathsf{t}(t^{*},z)+2s' +2)^{2^{\ell}}+2) \big)$ in $(H',\Omega')$.
By \cref{lem:exposed_transaction}, there is either a $\chi$-respectful $\rho^{\delta}_{H'}$-\textbf{pruning} of $(G, \delta_{1}, W_{1})$, or there exists a transaction $\mathcal{P}_1 \subseteq \mathcal{P}_0$ of order $(x+1)\big( 2xy + (s'+1)((3\mathsf{t}(t^{*},z)+2s' +2)^{2^{\ell}}+2) \big)$ which is $\rho^{\delta}_{H'}$-exposed.
By \cref{lem:proper_transaction}, $\mathcal{P}_1$ contains a transaction $\mathcal{P}_2$ of order $$(s'+1)((3\mathsf{t}(t^{*},z)+2s' +2)^{2^{\ell}}+2)$$ that is $\rho^{\delta}_{H'}$-proper.
Then, by \cref{lem:orthogonal_transaction}, we can assume that (up to possibly slightly changing the cycles in $\mathcal{C}'$ and the paths in the radial linkage $\mathcal{R}'$) that there exists a transaction $\mathcal{P}'_{2}$ of order $(3\mathsf{t}(t^{*},z)+2s' +2)^{2^{\ell}}$ that is $\rho^{\delta}_{H'}$-splitting.
Next we apply \cref{lem:homogeneous_transaction} to obtain a transaction $\mathcal{P}_3 \subseteq \mathcal{P}_2$ of order $3\mathsf{t}(t^{*},z)+2s' +2$ that is $\chi$-homogeneous in $\rho^{\delta}_{H'}$.
Recall that $s' + 1 = |\mathcal{C}'|.$

We are now ready to call upon \cref{lem:splitting_extracting} for $(H', \Omega')$ and $\mathcal{P}_3.$
The second outcome of \cref{lem:splitting_extracting} is a $\chi$-respectful $\rho^{\delta}_{H'}$-\textbf{pruning} of $(G, \delta_{1}, W_{1})$.
Hence, we may assume that \cref{lem:splitting_extracting} returns the first outcome.
Let us discuss the numbers involved to see that we may now complete the proof by using the induction hypothesis.
For the purpose of applying \cref{lem:splitting_extracting} let us set
\begin{align*}
    t_0 \ &\coloneqq \ \mathsf{t}(t^{*}, z-1)\text{, and}\\
    s_0 \ &\coloneqq \ \mathsf{s}(s^{*}, t^{*}, z-1).
\end{align*}
It then follows from the recursive definition of $\mathsf{s}$ and $\mathsf{t}$ that the size of the nest, the radial linkage of $(G, \Omega)$, and the transaction $\mathcal{P}_3$ are large enough such that \cref{lem:splitting_extracting} outputs an $(r, t_{0}, s_{0},\text{-},\text{-})$-$\Sigma$-schema $(G, \delta_{2}, W_{2})$ satisfying the assumptions of the inductive hypothesis and such that one of the following holds:
\begin{itemize}
\item $\chi$-$\mathsf{col}(\delta_{1})$ is a proper subset of $\chi$-$\mathsf{col}(\delta_{2})$ or
\item $W_{2}$ has exactly one vortex segment more than $W_{1}$.
\end{itemize}
In both cases, the total number of times we are allowed to ``split'' vortex segments has been reduced by at least one.
Hence, by applying the induction hypothesis for $z-1$ to $(G,\delta_{2}, W_{2})$ the proof is complete.
\end{proof}

\section{The local structure theorem with finite index}

In this final section, we present two versions of the LST.
The first, namely \cref{thm_local_structure}, combines \cref{prop_lst} with the tools we have developed in the previous sections.
The second, namely \cref{cor_bidim_structure}, is a corollary that witnesses the bidimensionality of our representation.

\subsection{Finding the final $\Sigma$-schema}\label{sec_main_lemma}

We first prove our core lemma on representing cell-colorings in $\Sigma$-schemas.

\begin{lemma}\label{main_lemma}
There exist functions $\mathsf{r}_{\ref{main_lemma}} \colon \Nbbb^{7} \to \Nbbb$, $\mathsf{t}_{\ref{main_lemma}}$, $\mathsf{depth}_{\ref{main_lemma}} \colon \Nbbb^{6} \to \Nbbb$, and $\mathsf{breadth}_{\ref{main_lemma}} \colon \Nbbb^{3} \to \Nbbb$ such that, for every $r, \mathsf{h}, \mathsf{c} \in \Nbbb$, every $t \in \Nbbb_{\geq 4}$, every $s, x, \ell, b \in \Nbbb_{\geq 1}$, every $y \in \Nbbb _{\geq 2}$, every $(r', t', \text{-}, x, y)$-$\Sigma$-schema $(G, \delta, W)$ such that
\begin{itemize}
\item $t' = \mathsf{t}_{\ref{main_lemma}}(t, s, \mathsf{h} + \mathsf{c}, \ell, b, x)$ and $r' = \mathsf{r}(r, t, s, \mathsf{h} + \mathsf{c}, \ell, b, x)$,
\item $\Sigma$ is a surface of Euler genus $2\mathsf{h} + \mathsf{c}$, and
\item $W$ is an $(r', t')$-$\Sigma$-annulus wall,
\end{itemize}
and every cell-coloring $\chi$ of $\delta$ of capacity at most $\ell$, there is an $(r, t, s, z, p)$-$\Sigma$-schema $(G, \delta', W')$ controlled by $\Tcal_{W}$ such that
\begin{enumerate}
\item $\delta'$ is a coarsening of $\delta$, i.e, $\delta' \sqsubseteq \delta$,
\item $z = \mathsf{breadth}_{\ref{main_lemma}}(\ell, b, x)$ and $p = \mathsf{depth}_{\ref{main_lemma}}(t, \ell, b, x, y)$,
\item every vortex segment society of $(G, \delta', W')$ has depth at most $\mathsf{depth}_{\ref{main_lemma}}(t, s, \ell, b, x, y)$, and
\item $W'$ $b$-represents $\chi$ in $\delta'$.
\end{enumerate}
Moreover, it holds that
\begin{align*}
\mathsf{t}_{\ref{main_lemma}}(t, s, \mathsf{h} + \mathsf{c}, \ell, b, x) \ &\in \ (\ell  b + 1)^{(x + b) \cdot 2^{\Ocal(\ell)}} \cdot (t + s)^{2^{\Ocal(\ell)}} \cdot (\mathsf{h} + \mathsf{c} + 1)^{2^{\Ocal(\ell)}},\\
\mathsf{r}_{\ref{main_lemma}}(r, t, s, \mathsf{h} + \mathsf{c}, \ell, b, x) \ &\in \ (r + (\ell  b + 1)^{(x + b) \cdot 2^{\Ocal(\ell)}} \cdot (t + s))^{2^{\Ocal(\ell)}} \cdot (\mathsf{h} + \mathsf{c} + 1)^{2^{\Ocal(\ell)}},\\
\mathsf{breadth}_{\ref{main_lemma}}(\ell, b, x) \ &\in \ x + b \cdot 2^{\Ocal(\ell)},\text{ and}\\
\mathsf{depth}_{\ref{main_lemma}}(t, s, \ell, b, x, y) \ &\in \ x^{2}y + (\ell  b + 1)^{(x + b) \cdot 2^{\Ocal(\ell)}} \cdot (t + s).
\end{align*}
There also exists an algorithm that finds the outcome above in time $\Ocal_{r, t, s, \mathsf{h} + \mathsf{c}, \ell, b, x, y}(|E(G)|^{2} |V(G)|)$.
\end{lemma}
\begin{proof} We define the four functions of the lemma as follows:
\begin{align*}
\mathsf{t}_{\ref{main_lemma}}(t, s, \mathsf{h} + \mathsf{c}, \ell, b, x) \ &\coloneqq \ f^{1}_{\ref{lem:homogenization}}\Big(f_{\ref{lem:representation}}\big( f^{1}_{\ref{lem:coarsening}}(t, s, \ell, b, x), \ell, b \big), \ell, \mathsf{h} + \mathsf{c} \Big),\\
\mathsf{r}_{\ref{main_lemma}}(r, t, s, \mathsf{h} + \mathsf{c}, \ell, b, x) \ &\coloneqq \ f^{2}_{\ref{lem:homogenization}}\Big(r, f_{\ref{lem:representation}}\big( f^{1}_{\ref{lem:coarsening}}(t, s, \ell, b, x), \ell, b \big), \ell, \mathsf{h} + \mathsf{c} \Big),\\
\mathsf{breadth}_{\ref{main_lemma}}(\ell, b, x) \ &\coloneqq \ f^{2}_{\ref{lem:coarsening}}(\ell, b, x),\text{ and}\\
\mathsf{depth}_{\ref{main_lemma}}(t, s, \ell, b, x, y) \ &\coloneqq \ f^{3}_{\ref{lem:coarsening}}(t, s, \ell, b, x, y).
\end{align*}

Let $t' = t_{1} = f^{1}_{\ref{lem:homogenization}}(t_{2}, \ell, \mathsf{h} + \mathsf{c})$, $t_{2} = f_{\ref{lem:representation}}(t_{3}, \ell, b)$, and $t_{3} = f^{1}_{\ref{lem:coarsening}}(t, s, \ell, b, x)$.
We first apply \cref{lem:homogenization} on the initial schema in order to obtain an $(r, t_{2}, \text{-}, x, y)$-$\Sigma$-schema $(G, \delta, W_{2})$ such that $W_{2}$ is an $(r, t_{2}, \text{-})$-$\Sigma$-annulus wall that is $\chi$-homogeneous in $\delta$.
Afterwards, we apply \cref{lem:representation} on $(G, \delta, W_{2})$ thus obtaining an $(r, t_{3}, \text{-}, x, y)$-$\Sigma$-schema $(G, \delta, W_{3})$ such that $W_{3}$ is an $(r, t_{3}, \text{-})$-$\Sigma$-walloid with no vortex segments that is $\chi$-homogeneous in $\delta$ and $b$-represents the $\chi$-coloring of every enclosure of $W_{3}$ in $\delta$.
The last step is to apply \cref{lem:coarsening} on $(G, \delta, W_{3})$.
Note that we can do so as $\delta$ by assumption has breadth at most $x$ and depth at most $y$.
This gives us an $(r, t, s, z, p)$-$\Sigma$-schema $(G, \delta', W')$ such that $\delta' \sqsubseteq \delta$ and 1) every vortex segment society of $(G, \delta', W')$ has depth at most $\mathsf{depth}_{\ref{main_lemma}}(t, s, \ell, b, x, y)$, 2) $W'$ $b$-represents $\chi$ in $\delta'$, and 3) $\Tcal_{W'} \subseteq \Tcal_{W}$.
With this our proof is complete.
\end{proof}

We remark that the dependency of the functions $\mathsf{t}_{\ref{main_lemma}}$ and $\mathsf{r}_{\ref{main_lemma}}$ on $s$ is in fact better that our gross estimations here.
However we choose to simplify the presentation as this does not really affect the order of the functions in the end.

\subsection{Grid models representing cell-colorings}

We first formalize the notion of bidimensional representation of a cell-coloring via grid models.

\medskip
Given a $\Sigma$-schema $(G, \delta, W)$ and $t \in \Nbbb_{\geq 5}$, we say that a minor model $\mu$ of a $(t \times t)$-grid $\Gamma$ in $G$ \emph{represents} $\chi$ in $\delta$ if
\begin{itemize}
\item $\mu(V(\Gamma))$ is disjoint from any vertex drawn the $\mathsf{nodes}_{\delta}(C^{\mathsf{si}}(W))$-avoiding disk of $\mathsf{trace}_{\delta}(C)$ of the outer cycle $C$ of every vortex segment of $(G, \delta, W)$ and
\item for every vertex $u \in V(\Gamma)$, for every color $\alpha \in \chi$-$\mathsf{col}(\delta)$, there is a simple cell $c$ of $\delta$ with $\chi(c) = \alpha$ and $\pi_{\delta'}(\tilde{c}) \subseteq \mu(u)$.
\end{itemize}

The following lemma shows how we can utilize the rich infrastructure provided by a walloid that sufficiently represents a cell-coloring in a $\Sigma$-schema, in order to obtain a large grid that represents the cell-coloring.

\begin{lemma}\label{lem:colorful_grid} Let $t \in \Nbbb_{\geq 5}$, $(G, \delta, W)$ be a $(\text{-}, t^{2}, \text{-}, \text{-}, \text{-})$-$\Sigma$-schema and $\chi$ be a coloring of $\delta$ such that $W$ $t^{2}$-represents $\chi$ in $\delta$.
Then there is a minor model $\mu$ of the $(t \times t)$-grid in $G$ that represents $\chi$ in $\delta$ and is controlled by the tangle $\Tcal_{W}$.
\end{lemma}
\begin{proof} As a first step, we define $W'$ by keeping only the infrastructure of $W$ necessary to see $W'$ as the cylindrical closure of all flap segments of $W$ (with their order adjusted to $t$) and a $(2t \times t^{2})$-wall segment.
The fact that $W$ $t^{2}$-represents $\chi$ in $\delta$, implies that for every color $\alpha \in \chi$-$\mathsf{col}(\delta)$, there are $t^{2}$ simple cells $c_{1}, \ldots, c_{t}$ of $\delta$ whose disks correspond to consecutive hyperedges of $t$-flap segments of $W'$, and such that for every $i \in [t]$, $\chi(c_{i}) = \alpha$.

As a second step, we show how to use the infrastructure provided by $W'$ to obtain a minor model $\mu$ of a $(2t \times t^{2})$-grid $\Gamma$ in $G$ such that for every $u \in V(\Gamma)$ of the $i$-th row of $\Gamma$, for every color $\alpha \in \chi$-$\mathsf{col}(\delta)$, $\pi_{\delta}(\tilde{c}) \subseteq \mu(u)$, where $c$ is a simple cell of $\delta$ whose disk corresponds to the hyperedge of one of the consecutive $t'$-flap segments of $W'$ that represent $\chi(c) = \alpha$.
Finding this grid is straightforward but quite tedious to formally present, so we instead present a proof by picture.
See \cref{fig_colorful_grid}.

\begin{figure}[th]
\centering
\includegraphics{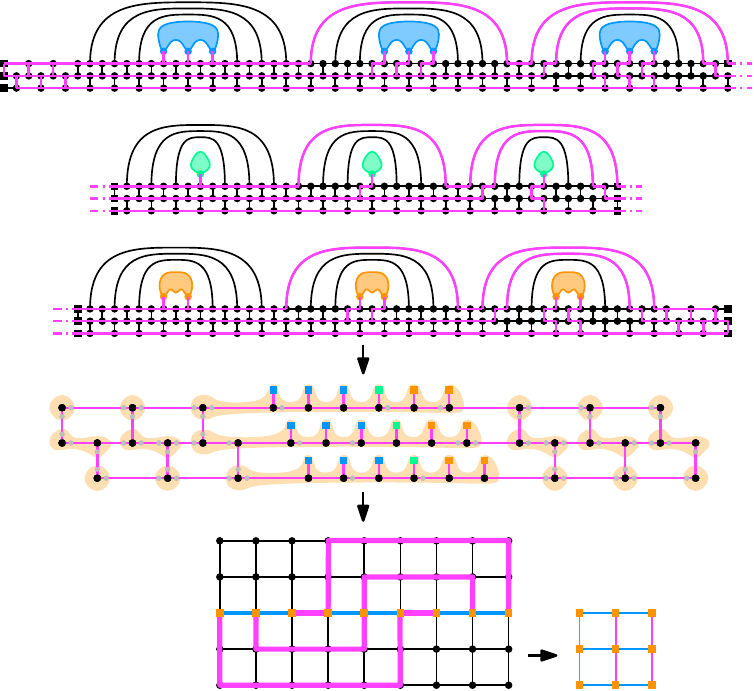}
\caption{\label{fig_colorful_grid} In the first two illustrations, a minor model in orange of a $(2t, t')$-grid in $W'''$ in magenta as required in the proof of \cref{lem:colorful_grid} for $t = t' = 3$. In the last two, the desired $(t, t)$-grid in the $(2t - 1, t^{2})$-grid as required in the proof of \cref{lem:colorful_grid} for $t = 3$.}
\end{figure}

The final step is to show how we can use $\mu$ to define a minor model $\mu'$ of a $(t \times t)$-grid $\Gamma'$ in $G$ that satisfies our claim.
As before we resort to a proof by picture.
See \cref{fig_colorful_grid}.

The fact that $\mu'$ is controlled by $\Tcal_{W}$ follows directly by construction of $\mu'$ from $W'$.
With this our proof concludes.
\end{proof}

The next corollary finally allows us to formally describe how \cref{main_lemma} bidimensionally represents the subgraphs that attach to our decomposition via simple cells.
It follows in a straightforward manner by applying \cref{lem:colorful_grid} to the outcome of \cref{main_lemma} applied with $b$ being $t^{2}$ and $s$ being $1$.

\begin{corollary}\label{bidim_main_lemma}
There exist functions $\mathsf{t}_{\ref{bidim_main_lemma}} \colon \Nbbb^{4} \to \Nbbb$, $\mathsf{r}_{\ref{bidim_main_lemma}} \colon \Nbbb^{5} \to \Nbbb$, $\mathsf{breadth}_{\ref{bidim_main_lemma}} \colon \Nbbb^{2} \to \Nbbb$, and $\mathsf{depth}_{\ref{bidim_main_lemma}} \colon \Nbbb^{4} \to \Nbbb$ such that, for every $r, \mathsf{h}, \mathsf{c} \in \Nbbb$, every $t \in \Nbbb_{\geq 4}$, every $x, \ell \in \Nbbb_{\geq 1}$, every $y \in \Nbbb _{\geq 2}$, every $(r', t', \text{-}, x, y)$-$\Sigma$-schema $(G, \delta, W)$ such that
\begin{itemize}
\item $t' = \mathsf{t}_{\ref{bidim_main_lemma}}(t, \mathsf{h} + \mathsf{c}, \ell, x)$ and $r' = \mathsf{r}_{\ref{bidim_main_lemma}}(r, t, \mathsf{h} + \mathsf{c}, \ell, x)$,
\item $\Sigma$ is a surface of Euler genus $2\mathsf{h} + \mathsf{c}$, and
\item $W$ is an $(r', t')$-$\Sigma$-annulus wall,
\end{itemize}
and every cell-coloring $\chi$ of $\delta$ of capacity at most $\ell$, there is a $\Sigma$-decomposition $\delta'$ of $G$ and an $(r + t)$-wall $W' \subseteq G$ such that $\delta'$ is a coarsening of $\delta$, i.e., $\delta' \sqsubseteq \delta$, and
\begin{enumerate}
\item $\delta'$ has breadth at most $\mathsf{breadth}_{\ref{bidim_main_lemma}}(\ell, x)$ and depth at most $\mathsf{depth}_{\ref{bidim_main_lemma}}(t, \ell, x, y)$,
\item $W'$ is grounded in $\delta'$ and controlled by $\Tcal_{W}$, and
\item there is a minor model of the $(t \times t)$-grid in $G$ controlled by $\Tcal_{W'}$ that represents $\chi$ in $\delta'$.
\end{enumerate}
Moreover, it holds that
\begin{align*}
\mathsf{t}_{\ref{bidim_main_lemma}}(t, \mathsf{h} + \mathsf{c}, \ell, x) \ &\in \ 2^{(x + t^{2}) \cdot \log t \cdot 2^{\Ocal(\ell)}} \cdot (\mathsf{h} + \mathsf{c} + 1)^{2^{\Ocal(\ell)}},\\
\mathsf{r}_{\ref{bidim_main_lemma}}(r, t, \mathsf{h} + \mathsf{c}, \ell, x) \ &\in \ (r + 2^{(x + t^{2}) \cdot \log t \cdot 2^{\Ocal(\ell)}})^{2^{\Ocal(\ell)}} \cdot (\mathsf{h} + \mathsf{c} + 1)^{2^{\Ocal(\ell)}},\\
\mathsf{breadth}_{\ref{bidim_main_lemma}}(\ell, x) \ &\in \ x + t^{2} \cdot 2^{\Ocal(\ell)},\text{ and},\\
\mathsf{depth}_{\ref{bidim_main_lemma}}(t, \ell, x, y) \ &\in \ x^{2}y + 2^{(x + t^{2}) \cdot \log t \cdot 2^{\Ocal(\ell)}}.
\end{align*}
There also exists an algorithm that finds the outcome above in time $\Ocal_{r, t, \mathsf{h} + \mathsf{c}, \ell, x, y}(|E(G)|^{2} |V(G)|)$.
\end{corollary}

\subsection{Proving the local structure theorem} \label{sec_proving_lst}

We now introduce the concept of a boundaried index as well as the concept of index representation.
Based on these concepts, we translate \cref{bidim_main_lemma} in order to obtain our extended version of the LST.

\paragraph{Universal ranking.}

A \emph{universal ranking} of a graph $G$ is an injection $\zeta \colon V(G) \to |V(G)|$.
From this point onward we assume that every graph is accompanied with a universal ranking that we denote by $\zeta_{G}$.
Given a set $S \subseteq V(G)$, $\zeta_{G}(S)$ denotes the set $\{ \zeta_{G}(u) \mid u \in S \}$.

\paragraph{Boundaried Indices.}

Let $q$ be a positive integer.
A \emph{$q$-boundaried graph} is a triple $(G, B, \rho)$ where $G$ is a graph, $B \subseteq V(G)$, $|B| = q$, and $\rho \colon B \to [q]$ is a bijection.
We use $\Bcal^{(q)}$ for the set of $q$-boundaried graphs and we define $\Bcal \coloneqq \bigcup_{i \in [3]} \Bcal^{(i)}$, i.e. the set of all $q$-boundaried graphs with $q \in [3]$.

A \emph{boundaried index} is a function $\iota \colon \Bcal \rightarrow \Nbbb_{\geq 1}$, such that if for two $\mathbf{B}_{1} = (G_{1}, B_{1}, \rho_{1}) \in \Bcal$ and $\mathbf{B}_{2} = (G_{2}, B_{2}, \rho_{2}) \in \Bcal$, $\iota(\mathbf{B}_{1}) = \iota(\mathbf{B}_{2})$, then $|B_{1}| = |B_{2}|$ and $\rho_{2}^{-1} \circ \rho_{1}$ is an isomorphism from $G[B_{1}]$ to $G[B_{2}]$.
Moreover, we demand that if $|B_{1}| < |B_{2}|$, then $\iota(\mathbf{B}_{1}) < \iota(\mathbf{B}_{2})$.
We define the capacity of $\iota$ as $\mathsf{cap}(\iota) = \max_{\mathbf{B} \in \Bcal} \iota(\mathbf{B})$.
We moreover demand that $\mathsf{cap}(\iota) = [\ell]$, for some positive integer $\ell$.

The fact that every graph $G$ comes with a universal ranking allows us to uniquely define boundaried graphs from its subgraphs. For instance, if $H$ is a subgraph of $G$ and $B \subseteq V(H)$, then we may define the boundaried graph $(H, B, \rho)$, where for every $u \in B$, $\rho(u) = i$ if and only if $\zeta_{G}(u)$ is the $i$-th smallest integer in $\zeta_{G}(B)$.
The universal ranking allows us, when dealing with subgraphs of a particular graph, to use the simpler notation $(H, B)$ for a boundaried graph instead of $(H, B, \rho)$. In other words, the universal ranking can be seen as a convenient way to assign an “identity” to each of the boundary vertices.
In what follows, we always consider boundaried graphs $(H, B)$, where every vertex $v \in V(H)$ that is adjacent to a vertex in $V(G) \setminus V(H)$ is also a vertex of $B$.

\paragraph{Equivalent cells.}

Let $\delta$ be a $\Sigma$-decomposition of a graph $G$.
We say that two simple cells $c$ and $c'$ of $\delta$ are \emph{$\Sigma$-equivalent} if $q \coloneqq |\tilde{c}| = |\tilde{c}'|$ and, given that $\zeta_{G}(\pi_{\delta}(\tilde{c})) = \lin{\pi_{\delta}(p_{1}), \ldots, \pi_{\delta}(p_{q})}$ and $\zeta_{G}(\pi(\tilde{c}')) = \lin{\pi_{\delta}(p_{1}'), \ldots, \pi_{\delta}(p_{q}')}$, there is no set $\{ I_{i} \mid i \in [q] \}$ of pairwise disjoint arcs in $\Sigma$ with $p_{i}$ and $p_{i}'$ being the two endpoints of $I_{i}$, $i \in [q]$, such that $(\bigcup_{i \in [q]} I_{i}) \cap (c \cup c') = \tilde{c} \cup \tilde{c}'$.
Notice that if $\Sigma$ is a non-orientable surface, then every two simple cells of $\delta$ are $\Sigma$-equivalent, independent of the choice of the accompanying universal ranking $\zeta_{G}$.
However, this is not true in the particular cases where $\Sigma$ is an orientable surface and $\delta$ contains simple cells with three nodes on their boundary. 

\paragraph{Index representation.}

Let $(G - A, \delta, W)$ be a $\Sigma$-schema for some $A \subseteq V(G)$.
Given a boundaried index $\iota$ we define the function $\iota_{\delta} \colon C_{\mathsf{s}}(\delta) \to \Nbbb_{\geq 1}$ so that for every simple cell $c$ of $\delta$, $\iota_{\delta}(c) \coloneqq \iota(\sigma_{\delta}(c), \pi_{\delta}(c))$.

We say that $(G - A, \delta, W)$ \emph{$b$-represents $\iota$} if for every simple cell $c$ of $\delta$, there are $b$ simple cells $c_{1}, \ldots, c_{b}$ of $\delta$ whose disks correspond to consecutive hyperedges of $t$-flap segments of $W$, and such that for every $i \in [b]$, $c_{i}$ is $\Sigma$-equivalent to $c$ and $\iota_{\delta}(c) = \iota_{\delta}(c_{i})$.

Moreover, an algorithm $\mathbf{A}$ for a boundaried index $\iota$ is one that, given a boundaried graph $\mathbf{B}$, computes $\iota(\mathbf{B})$. Typically, $\iota(\mathbf{B})$ may represent a property of boundaried graphs that has finite index.
The specific choice of $\iota$ may vary depending on the intended application of these results.

{In the above notion of representation, we also take into account the orientation of the simple cells of $\delta$, a feature that will prove useful in later applications. Moreover, all our definitions and results can be extended to accommodate more refined boundaried indices, where boundaried graphs carry additional structure—for example, vertex or edge labels.}

\medskip
We are now ready to prove our main result.

\begin{theorem}\label{thm_local_structure}
There exist functions $\mathsf{wall}_{\ref{thm_local_structure}} \colon \Nbbb^{6} \to \Nbbb$, $\mathsf{apex}_{\ref{thm_local_structure}}$, $\mathsf{depth}_{\ref{thm_local_structure}} \colon \Nbbb^{5} \to \Nbbb$, and $\mathsf{breadth}_{\ref{thm_local_structure}} \colon \Nbbb^{3} \to \Nbbb$ such that, for every boundaried index $\iota$ with $\mathsf{cap}(\iota) = \ell \in \Nbbb_{\geq 1}$, for every $r \in \Nbbb$, $k, t \in \Nbbb_{\geq 5}$, every $s, b \in \Nbbb_{\geq 1}$, every graph $H$ on $k$ vertices, every graph $G$, and every $\mathsf{wall}_{\ref{thm_local_structure}}(r, k, t, s, \ell, b)$-wall $W \subseteq G$, one of the following holds.
\begin{enumerate}
\item there is a minor model of $H$ in $G$ controlled by $\Tcal_{W}$, or
\item there is a set $A \subseteq V(G)$ and an $(r, t, s, z, p)$-$\Sigma$-schema $(G - A, \delta, W')$ controlled by $\Tcal_{W}$ such that
\begin{itemize}
\item $|A| \leq \mathsf{apex}_{\ref{thm_local_structure}}(k, t, s, \ell, b)$,
\item $z = \mathsf{breadth}_{\ref{thm_local_structure}}(k, \ell, b)$ and $p = \mathsf{depth}_{\ref{thm_local_structure}}(k, t, s, \ell, b)$,
\item $\Sigma$ is a surface where $H$ does not embed,
\item every vortex segment society of $(G - A, \delta, W')$ has depth at most $\mathsf{depth}_{\ref{thm_local_structure}}(k, t, s, \ell, b)$, and
\item $(G - A, \delta, W')$ $b$-represents $\iota$.
\end{itemize}
\end{enumerate}
Moreover, it holds that
\begin{align*}
\mathsf{wall}_{\ref{thm_local_structure}}(r, k, t, s, \ell, b) \ &\in \ (r + (\ell b + 1)^{(k^{2} + b) \cdot 2^{\Ocal(\ell)}} \cdot (t+s))^{2^{\Ocal(\ell)}} \cdot k^{2^{\Ocal(\ell)}}\\ \mathsf{apex}_{\ref{thm_local_structure}}(k, t, s, \ell, b), \mathsf{depth}_{\ref{thm_local_structure}}(k, t, s, \ell, b) \ &\in \ (\ell b + 1)^{(k^{2} + b) \cdot 2^{\Ocal(\ell)}} \cdot (t+s)^{2^{\Ocal(\ell)}} \cdot k^{2^{\Ocal(\ell)}}\text{, and}\\
\mathsf{breadth}_{\ref{thm_local_structure}}(k, \ell, b) \ &\in \ k^{2} + b \cdot 2^{\Ocal(\ell)}
\end{align*}
Also, given an algorithm $\mathbf{A}$ for $\iota$, there exists an algorithm that finds one of these outcomes in time $\Ocal_{r, k, t, s, \ell, b}(|E(G)|^{2} |V(G)|)$.
\end{theorem}
\begin{proof}
We define the four functions of the theorem as follows:
\begin{align*}
\mathsf{wall}_{\ref{thm_local_structure}}(r, k, t, s, \ell, b) \ &\coloneqq \ \mathsf{wall}_{\ref{prop_lst}}\Big( k, \mathsf{t}_{\ref{main_lemma}}\big( t, s, k^{2}, 2\ell, b, k^{2} \big), \mathsf{r}_{\ref{main_lemma}}\big( r, t, s, k^{2}, 2\ell, b, k^{2} \big) \Big),\\
\mathsf{apex}_{\ref{thm_local_structure}}(k, t, s, \ell, b) \ &\coloneqq \ \mathsf{apex}_{\ref{prop_lst}}\Big( k, \mathsf{t}_{\ref{main_lemma}}\big( t, s, k^{2}, 2\ell, b, k^{2} \big) \Big),\\
\mathsf{breadth}_{\ref{thm_local_structure}}(k, \ell, b) \ &\coloneqq \ \mathsf{breadth}_{\ref{main_lemma}}\big( 2\ell, b, k^{2} \big),\text{ and}\\
\mathsf{depth}_{\ref{thm_local_structure}}(k, t, s, \ell, b) \ &\coloneqq \ \mathsf{depth}_{\ref{main_lemma}}\bigg( t, s, 2\ell, b, k^{2}, \mathsf{depth}_{\ref{prop_lst}}\Big( k, \mathsf{t}_{\ref{main_lemma}}\big( t, s, k^{2}, 2\ell, b, k^{2} \big) \Big) \bigg).
\end{align*}

Let $t_{1} = \mathsf{t}_{\ref{main_lemma}}(t, s, k^{2}, 2\ell, b, k^{2})$ and $r_{1} = \mathsf{r}_{\ref{main_lemma}}(r, t, s, k^{2}, 2\ell, b, k^{2})$.
We first call upon \cref{prop_lst} in order to obtain 1) either a minor model of $H$ in $G$ controlled by $\Tcal_{W}$, in which case we conclude, or 2) a set $A \subseteq V(G)$ and an $(r_{1}, t_{1}, \text{-})$-$\Sigma$-schema $(G - A, \delta_{1}, W_{1})$ such that 1) $|A| \leq \mathsf{apex}_{\ref{thm_local_structure}}(k, t, s, \ell, b)$, 2) $\Sigma$ is a surface where $H$ does not embed, 3) $\delta_{1}$ has breadth at most $\nicefrac{1}{2}(k - 3)(k - 4)$ and depth at most $\mathsf{depth}_{\mathsf{\ref{prop_lst}}}(k, t_{1})$, 4) $W_{1}$ is an $(r_{1}, t_{1}, \text{-})$-$\Sigma$-annulus wall, and 5) $\Tcal_{W_{1}} \subseteq \Tcal_{W}$.

We now define a cell-coloring $\chi$ of $\delta_{1}$ as follows.
First, observe that the definition of $\Sigma$-equivalence between simple cells of $\delta_{1}$ defines an equivalence relation with at most two equivalence classes.
Let $C_{1}$, $C_{2}$ be these two equivalence classes (it may be that either $C_{1}$ or $C_{2}$ is empty but not both).
W.l.o.g. we may assume that $C_{1}$ is non-empty.
Given a cell $c_{i} \in C_{i}$, $i \in [2]$, we define $\chi(c_{i}) = (i - 1) \cdot \ell + \iota_{\delta_{1}}(c_{i})$.
Clearly this defines a cell-coloring of $\delta_{1}$ of capacity at most $2\ell$.

Moreover, note that since $H$ does not embed in $\Sigma$, it follows from known results (see e.g., \cite{gross1987topological,Ringel68Heawood,MoharT01Graphs}) that $\mathsf{h} + \mathsf{c} \leq k^{2}$, where $\Sigma$ is a surface of Euler genus $2\mathsf{h} + \mathsf{c}$ for some $\mathsf{h}, \mathsf{c} \in \Nbbb$.

We are now in the position to apply \cref{main_lemma} on $(G - A, \delta_{1}, W_{1})$ and obtain an $(r, t, s)$-$\Sigma$-schema $(G - A, \delta, W')$ such that a) $\delta \sqsubseteq \delta_{1}$, b) $\delta'$ has breadth at most $\mathsf{breadth}_{\ref{thm_local_structure}}(k, \ell, b)$ and depth at most $\mathsf{depth}_{\ref{thm_local_structure}}(k, t, s, \ell, b)$, c) every vortex segment society of $(G, \delta', W')$ has depth at most $\mathsf{depth}_{\ref{thm_local_structure}}(k, t, s, \ell, b)$, d) $W'$ $b$-represents $\chi$ in $\delta$, and e) $\Tcal_{W'} \subseteq \Tcal_{W_{1}}$.

By definition of $\chi$ and property d), it is implied that $(G - A, \delta, W')$ $b$-represents $\iota$.
Moreover, since $\Tcal_{W'} \subseteq \Tcal_{W_{1}}$ and $\Tcal_{W_{1}} \subseteq \Tcal_{W}$, then $\Tcal_{W'} \subseteq \Tcal_{W}$.
This concludes our proof.
\end{proof}

\subsection{Grid models representing boundaried indices}

Our last goal is to prove a version of \cref{thm_local_structure} that witnesses the bidimensionality of the representation in the same way as we do in \cref{bidim_main_lemma}.
For this, we first need the corresponding definitions of grids representing boundaried indices.

\paragraph{Grids representing indices in decompositions.}

Let $\iota$ be a boundaried index, $G$ be a graph, and $\delta$ be a $\Sigma$-decomposition of $G$.
Given $t \in \Nbbb_{\geq 5}$, we say that a minor model $\mu$ of a $(t \times t)$-grid $\Gamma$ in $G$ \emph{represents} $\iota$ in $\delta$ if
\begin{itemize}
\item $\mu(V(\Gamma))$ is disjoint from any vertex drawn in the closure of a vortex cell of $\delta$ and
\item for every vertex $u \in V(\Gamma)$, for every simple cell $c^{*}$ of $\delta$, there is a simple cell $c$ of $\delta$ that is $\Sigma$-equivalent to $c^{*}$ with $\iota_{\delta}(c^{*}) = \iota_{\delta}(c)$ and $\pi_{\delta}(\tilde{c}) \subseteq \mu(u)$.
\end{itemize}
 

Now, the next corollary follows in a straightforward manner by applying \cref{lem:colorful_grid} to the outcome of \cref{thm_local_structure} applied with $b$ being $t^{2}$ and $s$ being $1$.

\begin{corollary}\label{cor_bidim_structure}
There exists functions $\mathsf{wall}_{\ref{cor_bidim_structure}} \colon \Nbbb^{4} \to \Nbbb$, $\mathsf{apex}_{\ref{cor_bidim_structure}}, \mathsf{breadth}_{\ref{cor_bidim_structure}}, \mathsf{depth}_{\ref{cor_bidim_structure}} \colon \Nbbb^{3} \to \Nbbb$ such that, 
for every boundaried index $\iota$ with $\mathsf{cap}(\iota) = \ell \in \Nbbb_{\geq 1}$, every $k, t \in \Nbbb_{\geq 5}$, every $r \in \Nbbb$, every graph $H$ on $k$ vertices, every graph $G$, and every $\mathsf{wall}_{\ref{cor_bidim_structure}}(r, k, t, \ell)$-wall $W \subseteq G$, one of the following holds.
\begin{enumerate}
\item there is a minor model of $H$ in $G$ controlled by $\Tcal_{W}$, or
\item there is a set $A \subseteq V(G)$, a $\Sigma$-decomposition $\delta$ of $G - A$, and an $(r + t)$-wall $W'$ such that
\begin{itemize}
\item $|A| \leq \mathsf{apex}_{\ref{cor_bidim_structure}}(k, t, \ell)$,
\item $\Sigma$ is a surface where $H$ does not embed,
\item $\delta$ has breadth at most $\mathsf{breadth}_{\ref{cor_bidim_structure}}(k, t, \ell)$ and depth at most $\mathsf{depth}_{\ref{cor_bidim_structure}}(k, t, \ell)$, 
\item $W'$ is grounded in $\delta$ and controlled by $\Tcal_{W}$, and
\item there is a minor model of a $(t \times t)$-grid in $G - A$ controlled by $\Tcal_{W'}$ that represents $\iota$ in $\delta$.
\end{itemize}
\end{enumerate}
Moreover, it holds that
\begin{align*}
\mathsf{wall}_{\ref{cor_bidim_structure}}(r, k, t, \ell) \ &\in \ (r + 2^{(k^{2} + t^{2}) \cdot \log t \cdot {2^{\Ocal(\ell)}}})^{2^{\Ocal(\ell)}} \cdot k^{2^{\Ocal(\ell)}},\\
\mathsf{apex}_{\ref{cor_bidim_structure}}(k, t, \ell), \mathsf{depth}_{\ref{cor_bidim_structure}}(k, t, \ell) \ &\in \ 2^{(k^{2} + t^{2}) \cdot \log t \cdot {2^{\Ocal(\ell)}}}\text{, and}\\
\mathsf{breadth}_{\ref{cor_bidim_structure}}(k, \ell) \ &\in \ k^{2} + t^{2} \cdot 2^{\Ocal(\ell)}.
\end{align*}
Also, given an algorithm $\mathbf{A}$ for $\iota$, there exists an algorithm that finds one of these outcomes in time $\Ocal_{r, k, t,\ell}(|E(G)|^{2} |V(G)|)$.
\end{corollary}

As a final remark, we should stress that our LSTs are presented with all the parameters that are necessary to extract GSTs with standard techniques (see \cite{KawarabayashiTW2021Quickly,RobertsonS03GraphMinorsXVI}).
However, we think that this goes beyond the scope of this work and should be handled on a per application basis.

\section{Remarks and extensions}

As already mentioned, the representation of a boundary index $\iota$ can be further refined when we consider graphs with more enhanced features as edge/vertex labelings.

In this paper, we considered indices of boundaried graphs whose boundary has at most $3$ vertices. This restriction arises from the fact that indices are computed in $G-A$ rather than in $G$, i.e., the apices in $A$ are excluded from the definition of the index.  

We remark that versions of both \cref{thm_local_structure} and \cref{cor_bidim_structure} could also be proven—with worse parametric dependencies—by considering indices of boundaried graphs whose boundary is computed in $G$ (instead of $G-A$) and has at most $3+\alpha(H)-1=\alpha(H)+2$ vertices, where $\alpha(H)$ denotes the apex number\footnote{The \emph{apex number} of $H$, denoted by $\alpha(H)$, is the minimum number of vertices whose removal makes $H$ planar.} of the excluded graph.  
For this purpose, one could use as a starting point the version of the LST proved by Dvořák and Thomas in \cite[Theorem 12]{DvorakT14List}, in place of \cref{prop_lst}, where they distinguish between two types of apices: The \emph{major} ones, which may have neighbors anywhere in $G$, and the rest—let us call them \emph{pseudo-apices}—which have neighbors only inside the vortices.
The LST in \cite[Theorem 12]{DvorakT14List} asserts that the number of major apices can be restricted to at most $\alpha(H) - 1$.
The property that pseudo-apices have all their neighbors confined to the vortices is preserved throughout the constructions in our proofs.

This yields versions of \cref{thm_local_structure} and \cref{cor_bidim_structure} in which pseudo-apices are bounded by some function of $k,t$, and $l$, while the major apices remain bounded by $\alpha(H)-1$.
Furthermore, instead of considering only the $(\leq 3)$-boundaried graphs defined by the cells of $G-A$, one can now consider $(\leq \alpha(H) + 2)$-boundaried graphs defined by the cells of $G$, together with the at most $\alpha(H) - 1$ major apices and the edges between them.
This version is strictly stronger, as it incorporates the apices into the representation, which we expect to be useful in future applications.

Since the parametric dependencies of the LST in \cite{DvorakT14List} are not specified, we chose to use as a starting point of our proofs the more recent LST of \cite[Theorem 15.1]{GorskySW2025Polynomial}, i.e., \cref{prop_lst}, which guarantees polynomial dependencies and thus optimizes the parameters in our results. We expect that a future version of \cite[Theorem 15.1]{GorskySW2025Polynomial} will preserve polynomial dependencies while also distinguishing between the at most $\alpha(H)-1$ major apices and polynomially many pseudo-apices.

\paragraph{Acknowledgments.} The second and third author wish to thank Professor Panos Rondogiannis for his pivotal influence,
which contributed to the realization of this paper.

\newpage
\phantomsection
\addcontentsline{toc}{section}{References}
\bibliographystyle{alphaurl}

\end{document}